\documentclass[10pt]{article}
\usepackage{amsmath}
\usepackage{amscd}
\usepackage{amssymb}
\usepackage{graphicx}
\usepackage{amsthm}
\newtheorem{thm}{Theorem}[section]
\newtheorem{prop}[thm]{Proposition}
\newtheorem{lem}[thm]{Lemma}
  \theoremstyle{definition}
\newtheorem{df}[thm]{Definition}   \theoremstyle{definition}

\newtheorem{rem}[thm]{Remark}                \theoremstyle{plain}
 \theoremstyle{definition}
   
\def\CC{\Bbb{C}}
\def\RR{\Bbb{R}}  
\def\CCI{\hat{\CC}}        \def\NN{\Bbb{N}} 

\def\B1{{\rm\kern.32em\vrule    width.12em       height1.4ex
depth-.05ex\kern-.28em 1}}
\def\G{\Gamma}
\def\g{\gamma}
\def\GN{\Gamma ^{\NN }}

\def\Rat{\text{Rat}}
\def\emRat{\text{{\em Rat}}}

\def\Cpt{\text{Cpt}}

\def\Hol{\text{H\"{o}l}}
\def\emHol{\text{{\em H\"{o}l}}}
\def\Min{\text{Min}}

\begin{document}
\vspace{-5mm} 
\title{The space of $2$-generator postcritically bounded polynomial semigroups and 
random complex dynamics
\footnote{Date: January 6, 2016. Published in Adv. Math.  290 (2016) 809--859.  
2010 MSC.  
37F10, 30D05. Keywords: 
polynomial semigroups, random complex dynamics, 
randomness-induced phenomena,     
Julia sets, 
devil's coliseum, complex Takagi function.  
The author's webpage: http:www.math.sci.osaka-u.ac.jp/$\sim $sumi/. 
 }}  
\vspace{-7mm} 
\author{Hiroki Sumi\
\\   Department of Mathematics, Graduate School of Science,  
Osaka University\\   
1-1, Machikaneyama,\ Toyonaka,\  Osaka,\ 560-0043,\ 
Japan\\ E-mail:    sumi@math.sci.osaka-u.ac.jp}
\date{}
\maketitle
\vspace{-10mm} 
\begin{abstract}
We investigate the dynamics of $2$-generator semigroups of polynomials with bounded planar postcritical set 
and associated random dynamics on 
the Riemann sphere. Also, we investigate the  space ${\cal B}$ of such semigroups. 
We show that for a parameter $h$ in the intersection of ${\cal B}$, 
the hyperbolicity locus ${\cal H}$ and the closure of the disconnectedness locus 
(the space  of parameters for which the Julia set is disconnected),   
the corresponding semigroup satisfies either the open set condition 
(and  Bowen's formula) or that the Julia sets of the two generators coincide. 
Also, we show that for such a parameter $h$, if the Julia sets of the two generators do not coincide, 
then there exists a neighborhood $U$ of $h$ 
such that for each parameter in $U$, 
the Hausdorff dimension of the Julia set of the corresponding 
semigroup is strictly less than 2. 
Moreover, we show that the intersection of the connectedness locus 
and ${\cal B}\cap {\cal H}$ 
has dense interior. 
By using the results on the semigroups corresponding to these parameters,   
we investigate the associated functions which give the probability of tending to $\infty$ 
(complex analogues of the devil's staircase or Lebesgue's singular functions) 
and complex analogues of the Takagi function. 
 
\end{abstract} 
\vspace{-6mm} 
\section{Introduction}
\label{Introduction} 
\vspace{-2mm} 
Some partial results of this paper have been announced in \cite{Ssugexp} without any proofs.   
We investigate the dynamics 
of polynomial semigroups (i.e., semigroups of non-constant polynomial maps 
where the semigroup operation is functional composition) on the Riemann sphere $\CCI $ and 
random dynamics of polynomials.  
We focus on the  {\bf randomness-induced phenomena} 
 (i.e., phenomena which can hold in random dynamical systems but cannot 
hold in deterministic iteration dynamical systems) 
and the {\bf gradation between chaos and order}  
in random complex dynamical systems  
and we study complex analogues of the devil's staircase, Lebesgue's singular functions
and the Takagi function which are \Hol der continuous on $\CCI $ and vary only on 
{\bf connected} thin fractal Julia sets of the associated polynomial semigroups.  

The first study of random complex dynamics was given by J. E. Fornaess and  N. Sibony (\cite{FS}). 
For research on random dynamics of quadratic polynomials, 
see \cite{Br1, BBR,  GQL}.  
For recent research and motivations on random complex dynamics,  
see the author's works \cite{SdpbpI}--\cite{Scp}, \cite{JS1}.  
  In order to investigate random complex dynamical systems, it is very natural to study the dynamics of 
  associated polynomial semigroups. 
In fact, it is a very powerful tool to investigate random complex dynamics, 
since random complex dynamics and the dynamics of polynomial semigroups are related to each other very deeply.   
The first study of dynamics of polynomial semigroups was 
conducted by
A. Hinkkanen and G. J. Martin (\cite{HM}),
who were interested in the role of the
dynamics of polynomial semigroups while studying
various one-complex-dimensional
moduli spaces for discrete groups,
and
by F. Ren's group (\cite{GR}), 
 who studied 
such semigroups from the perspective of random dynamical systems.
Since the Julia set $J(G)$ of a finitely generated polynomial semigroup 
$G$ generated by $\{ h_{1},\ldots, h_{m}\} $ has 
``backward self-similarity,'' i.e.,  
$J(G)=\cup _{j=1}^{m}h_{j}^{-1}(J(G))$ (see \cite[Lemma 1.1.4]{S1}),  
the study of the dynamics of polynomial semigroups can be regarded as the study of  
``backward iterated function systems,'' and also as a generalization of the study of 
self-similar sets in fractal geometry.  
For recent work on the dynamics of polynomial semigroups, 
see the author's papers \cite{S1}--\cite{Srcddc}, and 
\cite{St3, SS, SU1, SU2, SU4, JS2}. 

To introduce the main idea of this paper,  
we let $G$ be a polynomial semigroup and denote by $F(G)$ the Fatou set of $G$, which is defined to be  
the maximal open subset of $\CCI $ where $G$ is equicontinuous with respect to the spherical distance on $\CCI $.    
We call $J(G):=\CCI \setminus F(G)$ the Julia set of $G.$  
For a polynomial map $g:\CCI \rightarrow \CCI $, we denote by $\mbox{CV}(g)$ the set of critical values of $g:\CCI \rightarrow \CCI .$ 
We set $\mbox{CV}^{\ast }(g):= \mbox{CV}(g)\setminus \{ \infty \} .$ 
For a polynomial semigroup $G$, we set 
$P(G):= \overline{\cup _{g\in G}\mbox{CV}(g)}$ (the closure is taken in $\CCI $) and $P^{\ast }(G):= P(G)\setminus \{ \infty \} .$ 
The set $P(G)$ is called the postcritical set of $G$ and $P^{\ast }(G)$ is called the planar postcritical set of $G.$ 
Note that if $G$ is generated by a family $\Lambda $ of polynomials, i.e., 
$G=\{ g_{1}\circ \cdots \circ g_{n}\mid n\in \NN, \forall g_{j}\in \Lambda \}$, 
 then 
$P(G)=\overline{ \cup _{h\in G\cup \{ Id\} }h(\cup _{g\in \Lambda }\mbox{CV}(g))}.$ 
In particular, $h(P(G))\subset P(G)$ for each $h\in G.$  
For a polynomial semigroup $G$, we set 
$\hat{K}(G):=\{ z\in \CC \mid \cup _{g\in G}\{ g(z)\}  \mbox{ is bounded in }\CC \} .$

Let 
${\cal P}$ be the space of all polynomial maps $g:\CCI \rightarrow \CCI $ with $\deg (g)\geq 2$ endowed with 
the distance $\kappa $ which is defined by 
$\kappa (f,g):=\sup _{z\in \CCI }d(f(z),g(z))$, where $d$ denotes the spherical distance on $\CCI .$ 
Let ${\cal P}_{n}=\{ f\in {\cal P}\mid \deg (f)=n\}.$ We remark that 
$\deg :{\cal P}\rightarrow \NN $ is continuous and 
for each $n\geq 2$,  
 ${\cal P}_{n}$ is a connected component of ${\cal P}$, ${\cal P}_{n}$ is an open and closed subset of ${\cal P}$, 
 and ${\cal P}_{n}\cong (\CC \setminus \{ 0\}) \times \CC ^{n}.$

For each $h=(h_{1},\ldots ,h_{m})\in {\cal P} ^{m}$, 
we denote by $\langle h_{1},\ldots ,h_{m}\rangle $ the polynomial semigroup 
generated by $\{ h_{1},\ldots, h_{m}\} $, i.e., 
$\langle h_{1},\ldots ,h_{m}\rangle =\{ h_{i_{1}}\circ \cdots \circ h_{i_{n}}\in {\cal P}  
\mid n\in \NN, \forall i_{j}\in \{ 1,\ldots ,m\} \} .$ 
Moreover, 
we set $F(h_{1},\ldots h_{m}):=F(\langle h_{1},\ldots ,h_{m}\rangle )$, 
$J(h_{1},\ldots ,h_{m}):=J(\langle h_{1},\ldots ,h_{m}\rangle )$, 
$P(h_{1},\ldots ,h_{m}):= P(\langle h_{1},\ldots ,h_{m}\rangle )$, 
$P^{\ast }(h_{1},\ldots ,h_{m}):= P^{\ast }(\langle h_{1},\ldots ,h_{m}\rangle )$,
and $\hat{K}(h_{1},\ldots ,h_{m}):=\hat{K}(\langle h_{1}\ldots ,h_{m}\rangle ).$ 

We say that a polynomial semigroup $G$ is {\bf postcritically bounded} if 
$P^{\ast }(G)$ is bounded in $\CC .$ We say that a polynomial semigroup $G$ is {\bf hyperbolic} if 
$P(G)\subset F(G).$ 
We are interested in the parameter space of $2$-generator postcritically bounded polynomial semigroups. 
\vspace{-1mm} 
\begin{df}We use the following notations.   
\vspace{-2mm} 
\begin{itemize}
\item 
We set 
${\cal B}:= \{ (h_{1},h_{2})\in 
{\cal P}^{2}\mid 
P^{\ast }(h_{1},h_{2}) \mbox{ is bounded in } 
\CC \}$.  
\vspace{-2mm} 
\item We set 
${\cal C}:=  
\{ (h_{1},h_{2})\in 
{\cal P}^{2}\mid 
J(h_{1},h_{2})\mbox{ is connected}\} $.  
\vspace{-2mm} 
\item We set 
${\cal D}:=  
\{ (h_{1},h_{2})\in 
{\cal P}^{2}\mid 
J(h_{1},h_{2})\mbox{ is disconnected}\} $.  
\vspace{-2mm} 
\item We set ${\cal H}:=
\{ (h_{1},h_{2})\in 
{\cal P}^{2}\mid 
\langle h_{1},h_{2}\rangle \mbox{ is hyperbolic} \} $.  
\vspace{-2mm} 
\item We set 
${\cal I} :=\{ 
 (h_{1},h_{2})\in {\cal P}^{2}\mid 
 J(h_{1})\cap J(h_{2})\neq \emptyset \} $.  
\vspace{-2mm} 
\item We set 
${\cal Q} := 
\{ (h_{1},h_{2})\in {\cal P}^{2}\mid 
J(h_{1})=J(h_{2}), \mbox{ and } J(h_{1}) \mbox{ and }
J(h_{2}) \mbox{ are quasicircles}\} $.   

\end{itemize}
\end{df}
It is well-known that for an element $f\in {\cal P}$, 
$J(f)$ is connected if and only if $P^{\ast }(f)$ is bounded in $\CC .$ However, 
we have ${\cal B}\cap {\cal D}\neq \emptyset $ (e.g. $(z^{3}, 2z^{3})\in {\cal B}\cap {\cal D}$). 
Moreover, we have the following. 
\vspace{-0.5mm} 
\begin{lem}[Lemmas 5.4, 5.1 in \cite{SdpbpIII}]
\label{l:hbdopen}
The sets 
${\cal H},{\cal B}\cap {\cal H}, {\cal D}\cap {\cal B}\cap 
{\cal H}$ 
are non-empty and open in the product space 
${\cal P}^{2}={\cal P}\times {\cal P}$.   
\end{lem}
\vspace{-0.9mm} 
To present the first main result, we need some notations. 
For each $z\in \CCI $, we denote by $T\CCI _{z}$ the complex tangent space of
$\CCI $ at $z.$ 
For a holomorphic map $\varphi :V\rightarrow \CCI $ defined on 
a domain $V$, we denote by $D\varphi _{z}:T\CCI _{z}\rightarrow T\CCI _{g(z)}$ the  
differential map of $\varphi $ at $z.$  
We denote by $\| D\varphi _{z}\| _{s}$ the norm of the derivative 
of $\varphi $ at $z$ with respect to the spherical metric on $\CCI .$ 
 \vspace{-0.5mm} 
\begin{df}
Let $h=(h_{1},\ldots ,h_{m})\in {\cal P}^{m}$ be an element. 
Let $\Sigma _{m}:= \{ 1,\ldots ,m\} ^{\NN }$ endowed with 
product topology. This is a compact metric space. 
Let $\sigma :\Sigma _{m}\rightarrow \Sigma _{m}$ be the shift map,  
i.e., $\sigma (w_{1},w_{2},\ldots )=(w_{2},w_{3},\ldots ).$ 
We define the map $\tilde{h}:\Sigma _{m}\times \CCI \rightarrow \Sigma _{m}\times \CCI $ 
by $\tilde{h}(w, y)=(\sigma (w),h_{w_{1}}(y))$, 
where $w=(w_{1},w_{2},\ldots )\in \Sigma _{m}$ and $y\in \CCI .$ 
The map $\tilde{h}:\Sigma _{m}\times \CCI \rightarrow \Sigma _{m}\times \CCI $ 
is called the {\bf skew product map associated with $h=(h_{1},\ldots ,h_{m}).$} 
For each $w\in \Sigma _{m}$, 
we denote by $F_{w}(\tilde{h})$ the maximal open subset of 
$\CCI $ where the family $\{ h_{w_{n}}\circ \cdots \circ h_{w_{1}}\} _{n\in \NN }$ of polynomial maps is normal. 
We set $J_{w}(\tilde{h}):=\CCI \setminus F_{w}(\tilde{h}).$ 
We set $J(\tilde{h}):= \overline{\bigcup _{w\in \Sigma _{m}}\{ w\} \times J_{w}(\tilde{h})}$, 
where the closure is taken in the product space $\Sigma _{m}\times \CCI .$ 
We set $F(\tilde{h}):=(\Sigma _{m}\times \CCI )\setminus J(\tilde{h})$. 
Let $\pi : \Sigma _{m}\times \CCI \rightarrow \Sigma _{m}$ and 
$\pi _{\CCI }: \Sigma _{m}\times \CCI \rightarrow \CCI $ be the canonical projections. 
%
For each $n\in \NN $ and each $(w, y)\in \Sigma _{m}\times \CCI $, 
we set $D\tilde{h}^{n}_{w,y}:=D(h_{w_{n}}\circ \cdots \circ h_{w_{1}})_{y}.$ 
For all $w =(w _{1}, w _{2},\ldots )\in \Sigma _{m}$ and $n\in \NN $, we denote by 
$w _{n}$ the $n$-th coordinate of $w.$  
\end{df} 

\vspace{-6mm} 
\begin{df}
Let $h=(h_{1},h_{2})\in {\cal P}^{2}. $ 
For each $z\in \CCI $ and each $t\geq 0$, 
we set 
\vspace{-2.905mm} 
$$Z_{(h_{1},h_{2})}(z,t):= \sum _{n=1}^{\infty }\sum _{(\omega _{1},\ldots ,\omega _{n})\in \{ 1,2\} ^{n}}
\sum _{y\in (h_{\omega _{n}}\circ \cdots \circ h_{\omega _{1}})^{-1}(z)}
\| D(h_{\omega _{n}}\circ \cdots \circ h_{\omega _{1}})_{y}\| _{s}^{-t},$$  
counting multiplicities. Here, we set $0^{-t}:= \infty . $
We set $Z_{(h_{1},h_{2})}(z):= \inf \{ t\geq 0\mid Z_{(h_{1},h_{2})}(z,t)<\infty \} $, where we set 
$\inf \emptyset := \infty .$ 
For each $h=(h_{1},h_{2})\in {\cal H}$, 
we denote by $\delta _{(h_{1},h_{2})}$ the unique zero of the pressure function 
$P(t)=P(\tilde{h}|_{J(\tilde{h})}, -t\log \varphi ), t\geq 0$, where 
$P(\cdot ,\cdot )$ denotes the topological pressure, and $\varphi (\omega ,y):= -\log \| D(h_{\omega _{1}})_{y}\| _{s}$ 
(for the existence and uniqueness of the zero, see \cite{S6}).   
\end{df}
\vspace{-1mm} 
Note that $(h_{1},h_{2})\mapsto \delta _{(h_{1},h_{2})}$ is real-analytic and plurisubharmonic 
on ${\cal H}$ (see \cite{SU1}). 
If an element $h=(h_{1},h_{2})\in {\cal H}$ satisfies the {\bf open set condition}, i.e., 
there exists a non-empty open subset $U$ of $\CCI $ such that 
$\cup _{j=1}^{2}h_{j}^{-1}(U)\subset U$ and $\cap _{j=1}^{2}h_{j}^{-1}(U)=\emptyset $, 
then the dynamics of $\langle h_{1}, h_{2}\rangle $ has  many interesting properties, e.g., 
$\dim _{H}(J(h_{1},h_{2}))=\delta _{(h_{1},h_{2})}$ (this is called Bowen's formula) where $\dim _{H}$ denotes the Hausdorff dimension 
with respect to the Euclidean distance (see \cite{S6}). 
Thus it is  very interesting  to consider when $(h_{1},h_{2})\in {\cal H}$ satisfies the open set condition.

We are interested in the semigroup and random dynamical system generated by an element $(h_{1},h_{2})$ of a small neighborhood of 
$(\partial {\cal D})\cap {\cal B}\cap {\cal H}.$ 
Note that by Lemma~\ref{l:hbdopen}, we have 
$(\partial {\cal D})\cap {\cal B}\cap {\cal H}=(\partial {\cal C})\cap {\cal B}\cap {\cal H}\subset {\cal C}\cap {\cal B}\cap {\cal H}.$  
Under the above notations, we have the following result.  
\vspace{-0.1mm} 
\begin{thm}
\label{spacemainth} 
All of the following statements \ref{spacemainthosc}--\ref{spacemainth4-2} hold.  
\begin{enumerate}
\vspace{-1mm} 
\item \label{spacemainthosc}
Let $(h_{1},h_{2})\in 
(\overline{{\cal D}}\cap {\cal B}\cap {\cal H})\setminus {\cal Q}$.  
Then $(h_{1},h_{2})$ satisfies the open set condition.
\vspace{-1.0mm} 
\item \label{spacemainth1}
Let $(h_{1},h_{2})\in {\cal D}\cap {\cal B}$  and let  
$G=\langle h_{1},h_{2}\rangle $. Then,   
$h_{1}^{-1}(J(G))\cap h_{2}^{-1}(J(G))=\emptyset $, 
$J(G)=\amalg _{w \in \{ 1,2\} ^{\NN }}J_{w }(\tilde{h})$ 
{\em (}$\amalg $ denotes the disjoint union{\em )}, 
each $J_{w }(\tilde{h})$ is connected, the map $w \mapsto J_{w}(\tilde{h})$ is continuous on $\{ 1,2\} ^{\NN }$ 
with respect to the Hausdorff metric, and $(h_{1},h_{2}) $ satisfies the open set condition. 
\item \label{spacemainth3}
Let $(h_{1},h_{2})\in 
\overline{{\cal D}}\cap {\cal B}\cap {\cal H}$ and let $G=\langle h_{1},h_{2}\rangle .$     
Then, $J(G)$ is porous and  
$$\dim _{H}(J(G))\leq \overline{\dim} _{B}(J(G))
< 2.$$ \vspace{-0.5mm} 
Here, a compact subset $X$ of $\CCI $ is said to be porous if 
there exists a constant $0<k<1$ such that for each $x\in X$ and for each 
$0<r<\mbox{{\em diam}}\,\CCI$, there exists a ball in $\{ y\in \CCI \mid d(y,x)<r\}\setminus X$ with radius at least $kr.$ 
Moreover, 
$\overline{\dim }_{B}$ denotes the upper Box dimension with respect to the Euclidean distance.  
\item \label{spacemainth4}
Let $(h_{1},h_{2})\in 
(\overline{{\cal D}}\cap {\cal B}\cap {\cal H})\setminus {\cal Q}$. 
Let $d_{j}:=\deg (h_{j})$ for each $j.$ 
Then,   
for each $z\in 
\CCI \setminus P(h_{1},h_{2})$, 
\vspace{-1mm} 
$$1<\frac{\log (d_{1}+d_{2})}{\sum _{j=1}^{2}\frac{d_{j}}{d_{1}+d_{2}}\log d_{j}}\leq \dim _{H}(J(h_{1},h_{2}))= 
\overline{\dim }_{B}(J(h_{1},h_{2}))=\delta _{(h_{1},h_{2})}=
Z_{(h_{1},h_{2})}(z).$$ 
Moreover, in addition to the assumption, 
if $\frac{\log (d_{1}+d_{2})}{\sum _{j=1}^{2}\frac{d_{j}}{d_{1}+d_{2}}\log d_{j}}= \dim _{H}(J(h_{1},h_{2}))$, 
then $(h_{1},h_{2})\in {\cal D}\cap {\cal B}\cap {\cal H}$ and 
there exist a transformation $\varphi (z)=\alpha z+\beta $ with $\alpha ,\beta \in \CC, \alpha \neq 0$,  
two complex numbers $a_{1}, a_{2}$ and an integer $d\geq 3$  
such that  
$d=d_{1}=d_{2}$ and such that for each $j=1,2, \varphi \circ h_{j}\circ \varphi ^{-1}(z)=a_{j}z^{d}.$  
\item \label{spacemainth4-2}
Let $(h_{1},h_{2})\in 
(\overline{{\cal D}}\cap {\cal B}\cap {\cal H})\setminus {\cal Q}$. 
Then there exist an    
$\epsilon >0$ and a neighborhood $V$ of  $(h_{1},h_{2})$ in 
${\cal B}\cap {\cal H}$ such that   
for each $(g_{1},g_{2})\in V$ and for each 
$z\in \CCI \setminus P(g_{1},g_{2})$, 
$$\dim _{H}(J( g_{1},g_{2}))\leq 
\overline{\dim }_{B}(J(g_{1},g_{2}))\leq 
\delta _{(g_{1},g_{2})}=Z_{(g_{1},g_{2})}(z)\leq 2-\epsilon <2.$$   
\end{enumerate}   
\end{thm}
We remark that $(z^{2},\frac{1}{2}z^{2})\in {\cal B}\cap {\cal H}$ and $J(z^{2},\frac{1}{2}z^{2})=\{ z\mid 1\leq |z|\leq 2\} $,  
whose interior is not empty, and the author showed that there exists an open neighborhood $U$ of 
$(z^{2},\frac{1}{2}z^{2})$ in ${\cal B}\cap {\cal H}$ such that for a.e. 
$(h_{1},h_{2})\in U$ with respect to the Lebesgue measure on $({\cal P}_{2})^{2}$,  
$2$-dimensional Lebesgue measure of $J(h_{1},h_{2})$ is positive 
(\cite[Theorem 1.6]{SU4}). 
We also remark that for each $(d_{1},d_{2})\in \NN \times \NN $ with $d_{1}\geq 2$, 
$d_{2}\geq 2$, $(d_{1},d_{2})\neq (2,2)$, we have 
$(z^{d_{1}},z^{d_{2}})\in $ 
$(\partial C)\cap {\cal B}\cap {\cal H}\cap {\cal Q}$, 
since $(z^{d_{1}},(1+\frac{1}{n})z^{d_{2}})\in {\cal D}\cap {\cal B}\cap {\cal H}$.  

We now present results on the topology of connectedness locus in the parameter space. 
\begin{thm} 
\label{t:spacetopology}
All of the following statements \ref{spacemainth2}--\ref{spaceexample} hold. 
\begin{enumerate}
\item \label{spacemainth2}
$\overline{\mbox{{\em int}}({\cal C})}\cap {\cal B}\cap {\cal H}
={\cal C}\cap {\cal B}\cap {\cal H}$.  
Here, we denote by $\mbox{{\em int}}({\cal C})$ the set of interior points of 
${\cal C}$ with respect to the topology in ${\cal P}^{2}.$ 
In particular, 
any element in $(\partial {\cal D})\cap {\cal B}\cap {\cal H}$ can be approximated by 
a sequence in $(\mbox{{\em int}}({\cal C}))\cap {\cal B}\cap {\cal H}.$

\item \label{spacemainth5}
$ {\cal D}\cap {\cal Q}=\emptyset $ and ${\cal D}\cap {\cal B}\cap {\cal I}=\emptyset .$    
For each connected component ${\cal V}$ of ${\cal P}^{2}$, 
   ${\cal Q}\cap {\cal V}$ is included in a proper holomorphic subvariety of  
${\cal V}$.   
\item \label{spacemainth7}
$((\partial {\cal C})\cap {\cal B}\cap {\cal H})\setminus {\cal Q}$ 
is an open and dense subset of $(\partial {\cal C})\cap {\cal B}\cap {\cal H}$ which is endowed with the relative topology 
from ${\cal P}^{2}.$   
\item \label{spaceexample}
Suppose that $h_{1}\in {\cal P}$, $\langle h_{1}\rangle $ is postcritically bounded, and $h_{1}$ is hyperbolic. 
Moreover, 
let $d\in \NN ,d\geq 2$ with   
$(\deg (h_{1}),d)\neq (2,2)$. Then, 
there exists an element $h_{2}\in {\cal P}$ 
such that 
 $$(h_{1},h_{2})\in 
((\partial {\cal C})\cap {\cal B}\cap {\cal H})\setminus {\cal I}\subset 
((\partial {\cal C})\cap {\cal B}\cap {\cal H})\setminus {\cal Q} \mbox{ and } 
\deg (h_{2})=d.$$   
\end{enumerate} 
\end{thm}
We next present the main results on random complex dynamical systems associated with 
elements $(h_{1},h_{2})\in {\cal B}.$ 
Let $\tau $ be a Borel probability measure on ${\cal P}$ with compact support. 
We consider the independent and identically distributed (i.i.d.) random dynamics on $\CCI $ such that 
at every step we choose a map $h\in {\cal P}$ according to $\tau .$ 
Thus this determines a time-discrete Markov process with time-homogeneous transition probabilities 
on the state space 
$\CCI $ such that for each $x\in \CCI $ and 
each Borel measurable subset $A$ of $\CCI $, 
the transition probability 
$p(x,A)$ from $x$ to $A$ is defined as $p(x,A)=\tau (\{ g\in {\cal P} \mid g(x)\in A\} ).$ 
For a metric space $X$, let ${\frak M}_{1}(X)$ be the space of all 
Borel probability measures on $X$ endowed with the topology 
induced by weak convergence (thus $\mu _{n}\rightarrow \mu $ in ${\frak M}_{1}(X)$ if and only if 
$\int \varphi d\mu _{n}\rightarrow \int \varphi d\mu $ for each bounded continuous function $\varphi :X\rightarrow \RR $). 
Note that if $X$ is a compact metric space, then ${\frak M}_{1}(X)$ is compact and metrizable. 
For each $\tau \in {\frak M}_{1}(X)$, we denote by supp$\, \tau $ the topological support of $\tau .$  
Let ${\frak M}_{1,c}(X)$ be the space of all Borel probability measures $\tau $ on $X$ such that supp$\,\tau $ is 
compact.     
It is very interesting and important to consider the following function of probability of tending to $\infty .$ 

\begin{df} For each   
$h=(h_{1},h_{2})\in {\cal P}^{2}$, each    
$z\in \CCI ,$ and each $ p\in (0,1)$, we use the following notations.   
We denote by 
$T(h_{1},h_{2},p,z)$ the probability of tending to $\infty \in \CCI $ 
regarding the random dynamics on $\CCI $ such that at every step we choose 
$h_{1}$ with probability $p$ and we choose $h_{2}$ with probability $1-p.$  
More precisely,  
setting $\tau _{h_{1},h_{2},p}:= p\delta _{h_{1}}+(1-p)\delta _{h_{2}}\in {\frak M}_{1}(\{ h_{1},h_{2}\} )$  
($\delta _{h_{i}}$ denotes the Dirac measure concentrated at $h_{i}$), 
we set 
$$T(h_{1},h_{2},p,z):= \tilde{\tau }_{h_{1},h_{2},p}(\{ \gamma \in \{ h_{1},h_{2}\} ^{\NN }\mid 
\gamma _{n}\circ \cdots \circ \gamma _{1}(z)\rightarrow 
\infty \mbox{ as } n\rightarrow \infty\} ),$$ 
where $\tilde{\tau }_{h_{1},h_{2},p}:= 
\otimes _{n=1}^{\infty }\tau _{h_{1},h_{2},p}\in {\frak M}_{1}(\{ h_{1},h_{2}\} ^{\NN }).$    
(Note:  
$z\mapsto T(h_{1},h_{2},p,z)$ is locally constant on 
$F(h_{1},h_{2})$. See \cite[Lemma 3.24]{Splms10}.)  
\end{df}

We now present results on the functions of probability of tending to $\infty .$ 
\begin{thm}
\label{t:spaceth1h2} 
Statements \ref{spacemainth8} and \ref{spacemainth9} hold. 
\begin{enumerate} 
\item \label{spacemainth8}
For each $(h_{1},h_{2})\in 
({\cal D}\cap {\cal B})\cup 
((\partial  {\cal C})\cap {\cal B}\cap {\cal H})$ and for each 
$0<p<1$,  \\  
$J(h_{1},h_{2})=
\{ z_{0}\in \CCI \mid \mbox{for each neighborhood }U\mbox{ of }z_{0}, \ 
  z\mapsto T(h_{1},h_{2},p,z) \mbox{ is not constant on }U \} $.    
\item \label{spacemainth9}
Let $(h_{1},h_{2})\in 
(\partial {\cal C})\cap {\cal B}\cap {\cal H}$ and let 
$0<p<1$. Then, the function  
$z\mapsto T(h_{1},h_{2},p,z)$ is continuous on  $\CCI $ if and only if 
$J(h_{1})\cap J(h_{2})=\emptyset $.    

\end{enumerate} 
\end{thm}

We next give the definition of mean stable systems, minimal sets and transition operator in order to present 
the results on the associated random dynamical systems and 
further results on the functions of probability of tending to $\infty .$  
\begin{df}
\label{d:as}
Let  $\Gamma $ be a non-empty compact subset of ${\cal P}$.  
Let $G$ be the polynomial semigroup generated by $\Gamma $, i.e., 
$G=\{ h_{1}\circ \cdots \circ h_{n}\in {\cal P} \mid n\in \NN , \forall h_{j}\in \Gamma \} .$  
We say that $\Gamma$ is {\bf mean stable} if 
there exist non-empty open subsets $U,V$ of $F(G)$ and a number $n\in \NN $ 
such that all of the following  (1)--(3) hold: 
{(1)} $\overline{V}\subset U$ and $\overline{U}\subset F(G)$;  
{(2)} For each $\gamma =(\gamma _{1},\gamma _{2},\ldots )\in \Gamma ^{\NN }$, 
$\gamma _{n}\circ \cdots \circ \gamma _{1}(\overline{U})\subset V$;  
{(3)} For each point $z\in \CCI $, there exists an element 
$g\in G$ such that  
$g(z)\in U.$ 

Furthermore, for a $\tau \in {\frak M}_{1,c}({\cal P})$, 
we say that $\tau $ is mean stable if $\mbox{supp}\,\tau$ is mean stable. 
\end{df}
The author showed that if $\tau \in {\frak M}_{1,c}({\cal P})$ is mean stable, then 
the associated random dynamical system has many interesting properties, e.g. the chaos 
of the averaged system disappears at every point of $\CCI $ due to the cooperation of many kinds of maps in the system, 
even though the iteration of each map of the system has a chaotic part 
(\cite{Splms10,Scp}). 
Also, the author showed that the set of mean stable compact subsets $\Gamma $ of ${\cal P}$ is 
open and dense in the space of all non-empty compact subsets of ${\cal P}$ with respect to the Hausdorff metric
(see \cite{Scp}). Those results are called the {\bf cooperation principle}. 
Note that for every $f\in {\cal P}$, $\{ f\}$ is not mean stable and 
$\langle f\rangle $ is chaotic in the Julia set $J(f)\neq \emptyset $. 
Thus the cooperation principle is a randomness-induced phenomenon 
which cannot hold in the deterministic iteration dynamics of an $f\in {\cal P}.$ 
Note that randomness-induced phenomena have been a central interest 
in the study of random dynamics. Many physicists have observed various kinds 
of randomness-induced phenomena (sometimes physicists call them ``noise-induced phenomena'') 
by numerical experiments (e.g., \cite{MT}).   
In this paper, we consider how large the set ${\cal MS}:=\{ (h_{1},h_{2})\in {\cal P}^{2}
\mid \{ h_{1},h_{2}\} \mbox{ is mean stable}\} $ (resp. ${\cal MS}\cap {\cal B}$) 
is in ${\cal P}^{2}$ (resp. ${\cal B}$).  
\begin{df}
For a metric space $X$, we denote by $\Cpt(X)$ the space of all 
non-empty compact subsets of $X$ endowed with the Hausdorff metric.  
 For a polynomial semigroup $G$, we say that a non-empty compact subset $L$ of $\CCI $ is a minimal set for $(G,\CCI )$ 
if $L$ is minimal in 
$\{ C\in \Cpt(\CCI ) \mid  \forall g\in G, g(C)\subset C\} $ 
with respect to inclusion.  
Moreover, we set $\Min (G,\CCI ):= \{ L \in \Cpt(\CCI )\mid L \mbox{ is a minimal set for } (G,\CCI )\} .$ 

\end{df}
\begin{df}
For a compact metric space $X$, we denote by $C(X)$ the Banach space of 
all complex-valued continuous functions on $X$ endowed with supremum norm $\| \cdot \| .$ 
Let $(h_{1},h_{2})\in {\cal P}^{2}$ and let $p\in (0,1)$. 
We denote by $M_{h_{1},h_{2},p}:C(\CCI )\rightarrow C(\CCI )$ 
the operator defined by 
$M_{h_{1},h_{2},p}(\varphi )(z)=p\varphi (h_{1}(z))+(1-p)\varphi (h_{2}(z))$ for each 
$\varphi \in C(\CCI ), z\in \CCI .$ 
This $M_{h_{1},h_{2},p}$ is called the transition operator with respect to 
the random dynamical system associated with $\tau _{h_{1},h_{2},p}=p\delta _{h_{1}}+(1-p)\delta _{h_{2}}.$  
We denote by 
$M_{h_{1},h_{2},p}^{\ast }:{\frak M}_{1}(\CCI )\rightarrow {\frak M}_{1}(\CCI )$ the 
dual map of $M_{h_{1},h_{2},p}$, i.e., 
$\int \varphi (z) d(M_{h_{1},h_{2},p}^{\ast }(\mu ))(z) =\int M_{h_{1},h_{2},p}(\varphi )(z) d\mu (z)$ 
for each $\varphi \in C(\CCI )$ and $\mu \in {\frak M}_{1}(\CCI ).$ 
\end{df}
\begin{df}
For each $\alpha \in (0,1)$, 
we set $C^{\alpha }(\CCI ):=\{ \varphi \in C(\CCI )\mid  
\sup _{x,y\in \CCI ,x\neq y}\frac{|\varphi (x)-\varphi (y)|}{d(x,y)^{\alpha }}<\infty \} .$ 
Moreover, for each $\varphi \in C^{\alpha }(\CCI )$, 
we set $\| \varphi \| _{\alpha }:= \sup _{z\in \CCI }|\varphi (z)|
+\sup _{x,y\in \CCI ,x\neq y}\frac{|\varphi (x)-\varphi (y)|}{d(x,y)^{\alpha }}.$ 
Note that $C^{\alpha }(\CCI )$ is a Banach space with this norm $\| \cdot \| _{\alpha }.$ 
\end{df}
\begin{df}
\label{d:finftyg}
For a polynomial semigroup $G$ with $\infty \in F(G)$, we denote by $F_{\infty }(G)$ the 
connected component of $F(G)$ containing $\infty .$ Also, for an element $g\in {\cal P}$, 
we set $F_{\infty }(g):= F_{\infty }(\langle g\rangle ).$ 
\end{df}
\begin{rem}
It is easy to see that if $G$ is generated by a compact subset of ${\cal P}$, then 
$\infty \in F(G).$ 
\end{rem}

We now present the results on the random dynamical systems generated by elements 
in ${\cal B}$ and 
further results on the functions of probability of tending to $\infty .$  
\begin{thm}
\label{t:timportant}
Statements \ref{spacemainth10} and \ref{spacemainth11} hold. 
\begin{enumerate} 
\item \label{spacemainth10}
Let $(h_{1},h_{2})\in ({\cal D}\cap {\cal B}\cap {\cal H})\cup 
(((\partial {\cal C})\cap {\cal B}\cap {\cal H})
\setminus {\cal I})$. Let $p\in (0,1).$ 
Then, there exist an open neighborhood $V$ of 
$(h_{1},h_{2})$ in ${\cal B}\cap {\cal H}$,  an open neighborhood $W$ of $p$ in $(0,1)$ 
and a constant $\alpha \in (0,1)$  
such that   
all of the following {\em (i)}--{\em (viii)} hold. 
\begin{itemize}
\item[{\em (i)}]
For each $(g_{1},g_{2},q)\in V\times (0,1)$,  
$\{ g_{1},g_{2}\} $ is mean stable and $\tau _{g_{1},g_{2},q}$ is mean stable. 
\item[{\em (ii)}]
For each $(g_{1},g_{2},q)\in V\times W$, 
$z\mapsto T(g_{1},g_{2},q,z)$ is $\alpha $-H\"{o}lder continuous on $\CCI $.    
\item[{\em (iii)}] 
Let $(g_{1},g_{2})\in V$.  
Then there exists a unique minimal set $L_{g_{1},g_{2}}$ for $(\langle g_{1},g_{2}\rangle , \hat{K}(\langle g_{1},g_{2}\rangle ))$ and 
the set of minimal sets for $(\langle g_{1},g_{2}\rangle ,\CCI )$ is $\{ \{ \infty \}, L_{g_{1},g_{2}}\} .$  
Moreover, $\{ \infty \} \cup L_{g_{1},g_{2}}\subset F(g_{1},g_{2})$ and 
$L_{g_{1},g_{2}}\subset \mbox{int}(\hat{K}(g_{1},g_{2})).$ 
\item[{\em (iv)}] 
For each $(g_{1},g_{2},q,z)\in V\times W\times \CCI $ 
there exists a Borel subset ${\cal B}_{g_{1},g_{2},q,z}$ of $\{ g_{1},g_{2}\} ^{\NN }$ 
with $\tilde{\tau }_{g_{1},g_{2},p}({\cal B}_{g_{1},g_{2},q,z})=1$ such that for each 
$\gamma =(\gamma _{1},\gamma _{2},\ldots )\in {\cal B}_{g_{1},g_{2},q,z}$, we have 
$d(\gamma _{n}\cdots \gamma _{1}(z), \{ \infty \} \cup L_{g_{1},g_{2}})\rightarrow 0$ as $n\rightarrow \infty .$  
\item[{\em (v)}] 
For each $(g_{1},g_{2},q)\in V\times W$, 
there exists a unique $M_{g_{1},g_{2},q}^{\ast }$-invariant Borel probability measure $\nu =\nu _{g_{1},g_{2},q}$ on 
$\hat{K}(g_{1},g_{2})$ 
such that for each $\varphi \in C(\CCI )$, 
$$M_{g_{1},g_{2},q}^{n}(\varphi )(z)\rightarrow T(g_{1},g_{2},q,z)\cdot \varphi (\infty ) +(1-T(g_{1},g_{2},q,z))\cdot \int \varphi d\nu$$  
uniformly on $\CCI $ as $n\rightarrow \infty .$  
\item[{\em (vi)}] 
The map $(g_{1},g_{2},q)\in V\times W\mapsto T(g_{1},g_{2},q,\cdot )\in C(\CCI )$ is continuous on $V\times W.$ 
The map $(g_{1},g_{2},q)\in V\times W\mapsto \nu _{g_{1},g_{2},q}\in {\frak M}_{1}(\CCI )$ is continuous on $V\times W.$ 
\item[{\em (vii)}] 
For each $(g_{1},g_{2})\in V$, there exist an open neighborhood $W_{g_{1},g_{2}}$ of $p$ in $W$ and a constant $\beta \in (0,1)$ 
such that for each $q\in W_{g_{1},g_{2}}$, we have $T(g_{1},g_{2},q,\cdot )\in C^{\beta }(\CCI )$ and 
the map $q\mapsto T(g_{1},g_{2},q,\cdot )\in (C^{\beta }(\CCI ), \| \cdot \| _{\beta })$ is real-analytic on $W_{g_{1},g_{2}}.$  
\item[{\em (viii)}] 
Let $(g_{1},g_{2})\in V.$ Let $G=\langle g_{1},g_{2}\rangle.$ Then  
for each $z\in \CCI $, the function  
$q\mapsto T(g_{1},g_{2},q,z) $ is real-analytic on   
$(0,1)$. Moreover,   
for each $n\in \NN \cup \{ 0\} $, the function  
$(q,z)\mapsto (\partial ^{n}T/\partial q^{n})
(g_{1},g_{2},q,z)$ is continuous on $(0,1)\times \CCI $.   
Moreover, for each $q\in (0,1)$, 
the function $z\mapsto T(g_{1},g_{2},q,z)$ is characterized by the unique element $\varphi \in C(\CCI )$
such that 
$M_{g_{1},g_{2},q}(\varphi )=\varphi ,\varphi |_{\hat{K}(G)}\equiv 0,  $ 
$ \varphi | _{ F_{\infty }(G)}\equiv 1$. 
Furthermore, inductively, 
for any $n\in \NN \cup \{ 0\} $ and for any $q\in (0,1),$  
the function $z\mapsto (\partial ^{n+1}T/\partial q^{n+1})(g_{1},g_{2},q,z)$
is characterized by the unique element $\varphi \in C(\CCI )$ such that  
\begin{eqnarray*}\varphi (z)& \equiv &M_{g_{1},g_{2},q}(\varphi )(z)+
(n+1)\left(\frac{\partial ^{n}T}{\partial q^{n}}(g_{1},g_{2},q,g_{1}(z))-
\frac{\partial ^{n}T}{\partial q^{n}}(g_{1},g_{2},q,g_{2}(z))\right),\\   
& & \varphi |_{\hat{K}(G)\cup F_{\infty }(G)} \equiv 0.
\end{eqnarray*}   
Moreover, for any $n\in \NN \cup \{ 0\} $, the function $z\mapsto (\partial ^{n}T/\partial q^{n})(g_{1},g_{2},q,z)$ is H\"{o}lder continuous on $\CCI $ and locally constant on $F(G).$  
Moreover, for each $q\in (0,1), $ there exists a $\zeta \in (0,1)$ such that 
setting $\psi _{n,g_{1},g_{2},q}(z)=
(\partial ^{n}T/\partial q^{n})(g_{1},g_{2},q,z)$ for each $z\in \CCI $ and $n\in \NN \cup \{ 0\}$, we have that 
$\psi _{n+1,g_{1},g_{2},q}=\sum _{j=0}^{\infty }M_{g_{1},g_{2},q}^{j}
((n+1)(\psi _{n,g_{1},g_{2},q}\circ g_{1}-\psi _{n,g_{1},g_{2},q}\circ 
g_{2}))$ in $(C^{\zeta }(\CCI ),\| \cdot \| _{\zeta })$ for each $n\in \NN \cup \{ 0\} .$   
\end{itemize}

\item \label{spacemainth11}
Let $(h_{1},h_{2})\in ({\cal D}\cap {\cal B})$ and let 
$G=\langle h_{1},h_{2}\rangle $. Then,   
for each $p\in (0,1)$, the function $T(h_{1},h_{2},p,\cdot )$ is H\"{o}lder continuous on $\CCI $ and 
for each $z\in \CCI $, the function  
$p\mapsto T(h_{1},h_{2},p,z) $ is real-analytic on   
$(0,1)$. Moreover,   
for each $n\in \NN \cup \{ 0\} $, the function  
$(p,z)\mapsto (\partial ^{n}T/\partial p^{n})
(h_{1},h_{2},p,z)$ is continuous on $(0,1)\times \CCI $.   
For each $p\in (0,1)$, the function $z\mapsto T(h_{1},h_{2},p,z)$ is characterized by the unique element $\varphi \in C(\CCI )$
such that 
$M_{h_{1},h_{2},p}(\varphi )=\varphi ,\varphi |_{\hat{K}(G)}\equiv 0,  
\varphi |_{F_{\infty }(G)}\equiv 1$. 
Furthermore, inductively, 
for any $n\in \NN \cup \{ 0\} $, 
the function $z\mapsto (\partial ^{n+1}T/\partial p^{n+1})(h_{1},h_{2},p,z)$
is characterized by the unique element $\varphi \in C(\CCI )$ such that  
\\ 
$\varphi (z)\equiv M_{h_{1},h_{2},p}(\varphi )(z)+
(n+1)\left((\partial ^{n}T/\partial p^{n})(h_{1},h_{2},p,h_{1}(z))-
(\partial ^{n}T/\partial p^{n})(h_{1},h_{2},p,h_{2}(z))\right),\\   
\varphi |_{\hat{K}(G)\cup F_{\infty }(G)}\equiv 0$.  

Moreover, setting $\psi _{n,h_{1},h_{2},p}(z)=(\partial ^{n}T/\partial p^{n})(h_{1},h_{2},p,z)$ for each 
$p\in (0,1), z\in \CCI, n\in \NN \cup \{ 0\} $, we have that 
$\psi _{n+1,h_{1},h_{2},p}=\sum _{j=0}^{\infty }M_{h_{1},h_{2},p}^{j}
((n+1)(\psi _{n,h_{1},h_{2},p}\circ h_{1}-\psi _{n,h_{1},h_{2},p}\circ h_{2}))$ 
in $(C(\CCI ), \| \cdot \| _{\infty })$ for each $n\in \NN \cup \{ 0\} .$   
\end{enumerate} 
\end{thm}
In order to present results on the pointwise H\"{o}lder exponents of 
the functions of probability of tending to $\infty $, we need the following definitions. 

\begin{df}
\label{d:maxrelent}
Let $h=(h_{1},h_{2})\in {\cal P}^{2}$. 
Let $p\in (0,1).$ We set $p_{1}=p, p_{2}=1-p.$  
Let $\eta _{p} =\otimes _{n=1}^{\infty }(\sum _{j=1}^{2}p_{j}\delta _{j})\in {\frak M}_{1}(\Sigma _{2})$ 
be the Bernoulli measure on $\Sigma _{2}$ with respect to the weight $(p_{1},p_{2}).$ 
We denote by $\pi :\Sigma _{2}\times \CCI \rightarrow \Sigma _{2}$ the canonical projection onto $\Sigma _{2}.$  
Also, we denote by 
$\pi _{\CCI }: \Sigma _{2}\times \CCI \rightarrow \CCI $ the canonical projection onto $\CCI .$  
It is known that there exists a unique $\tilde{h}$-invariant ergodic Borel probability measure 
$\tilde{\lambda }_{h_{1},h_{2},p} $ on $\Sigma _{2}\times \CCI $ 
such that $\pi _{\ast }(\tilde{\lambda }_{h_{1},h_{2},p})=\eta _{p}$ and   
$h_{\tilde{\lambda }_{h_{1},h_{2},p}}(\tilde{h}|\sigma )=\max _{\rho \in {\frak E}_{1}(\Sigma _{m}\times \CCI ): 
\tilde{h}_{\ast }(\rho )=\rho, \pi _{\ast }(\rho )=\tilde{\tau} }  
h_{\rho }(\tilde{h}|\sigma )=\sum _{j=1}^{2}p_{j}\log (\deg (h_{j}))$, 
where $h_{\rho }(\tilde{h}|\sigma )$ denotes the relative metric entropy 
of $(\tilde{h},\rho )$ with respect to $(\sigma, \eta _{p})$, and 
${\frak E}_{1}(\cdot )$ denotes the space of ergodic measures (see \cite{S3}).  
This $\tilde{\lambda }_{h_{1},h_{2},p} $ is called the {\bf maximal relative entropy measure} for $\tilde{h}$ with respect to 
$(\sigma ,\eta _{p}).$ Also, we set 
$\lambda _{h_{1},h_{2},p}:= (\pi _{\CCI })_{\ast }(\tilde{\lambda }_{h_{1},h_{2},p}).$ 
This is a Borel probability measure on $J(h_{1},h_{2}).$    
\end{df}
\noindent \begin{df}\ 
\begin{itemize}  
\item 
Let 
$\tilde{h}:\Sigma _{2}\times 
\CCI \rightarrow \Sigma _{2}\times \CCI $ be the skew product 
associated with 
$h=(h_{1},h_{2})\in {\cal H}$. Let $p\in (0,1)$. 
We set $p_{1}=p, p_{2}=1-p.$ 
Let  
$\rho $ be an $\tilde{h}$-invariant Borel probability measure on 
$J(\tilde{h})$ (i.e., $\rho (A)=\rho (\tilde{h}^{-1}(A))$ for each Borel measurable subset $A$ of 
$J(\tilde{h})$).  
Moreover, we set 
$$u(h_{1},h_{2},p,\rho ):= 
\frac{-\int _{\Sigma _{2}\times \CCI }\log p_{w_{1}}\ d\rho (w,x)}
{\int _{\Sigma _{2}\times \CCI }\log \| D(h_{w_{1}})_{x}\| _{s} \ d\rho (w,x)}\in (0,\infty ).$$ 
\item 
Let $V$ be an open subset of 
$\CC $. For any function $\varphi :V \rightarrow \RR $ and any point 
$y\in V $, if $\varphi $ is bounded around $y$, we set \\ 
$\mbox{H\"{o}l}(\varphi ,y):= 
\sup \{ \beta \geq 0 \mid \limsup _{z\rightarrow y,z\neq y}
(|\varphi (z)-\varphi (y)|/|z-y|^{\beta })<\infty \} \in [0,\infty ]$, and this is called  
the {\bf pointwise H\"{o}lder exponent of $\varphi $ at $y$}.  
({\bf Note:} If $\mbox{H\"{o}l}(\varphi ,y)<1$, then 
$\varphi $ is not differentiable at $y$. If $\mbox{H\"{o}l}(\varphi ,y)>1$, then 
$\varphi $ is differentiable at $y$ and the derivative is equal to $0$.) 
\end{itemize}
\end{df}

We now present the results on the pointwise H\"{o}lder exponents 
of the functions of probability of tending to $\infty .$ 
\begin{thm}
\label{t:spacend}
Statements~\ref{nondiff} and \ref{nonorable} hold. 
\begin{enumerate} 
\item \label{nondiff}({\bf Non-differentiability})
Let $(h_{1},h_{2})\in \overline{{\cal D}}\cap {\cal B}\cap {\cal H}$, 
$G=\langle h_{1},h_{2}\rangle $, and 
$0<p<1$.  
Then,  
 supp $\lambda  _{h_{1},h_{2},p}=J(G)$, $\lambda _{h_{1},h_{2},p}$ is non-atomic, 
and for almost every point $z_{0}\in J(G)$ with respect to  
 $\lambda _{h_{1},h_{2},p}$, 
\begin{equation} 
\label{eq:holth1}
\mbox{{\em H\"{o}l}}(T(h_{1},h_{2},p,\cdot ),z_{0})\leq 
u(h_{1},h_{2},p,\tilde{\lambda }_{h_{1},h_{2},p})
=-\frac{p\log p+(1-p)\log (1-p)}
{p\log (\deg (h_{1}))+(1-p)\log (\deg (h_{2}))}
<1
\end{equation}  
and $T(h_{1},h_{2},p,\cdot )$ is not differentiable at $z_{0}$.   
In particular, there exists an uncountable dense subset $A$ of $J(G)$ such that 
at every point of $A$, the function 
$T(h_{1},h_{2},p,\cdot )$ is not differentiable. 
Moreover, if 
$(h_{1},h_{2})\in {\cal D}\cap {\cal B}\cap {\cal H}$,  
then $\mbox{{\em H\"{o}l}}(T(h_{1},h_{2},p,\cdot ),z_{0})= 
u(h_{1},h_{2},p,\tilde{\lambda }_{h_{1},h_{2},p})$ for almost every point $z_{0}\in J(G) $ 
with respect to $\lambda _{h_{1},h_{2},p}.$  
Moreover, if $(h_{1},h_{2})\in (\overline{{\cal D}}\cap {\cal B}\cap {\cal H})\setminus {\cal Q}$, 
then for almost every point $z_{0}\in J(G)$ with respect to $\lambda _{h_{1},h_{2},p}$, 
the function $T(h_{1},h_{2},p,\cdot )$ is continuous at $z_{0}.$ 
\item \label{nonorable}
Let $(h_{1},h_{2})\in (\overline{{\cal D}}\cap {\cal B}\cap {\cal H})\setminus {\cal Q}$, 
$G=\langle h_{1},h_{2}\rangle $, 
and $0<p<1$. Let $v =\dim _{H}(J(G))$ and let  
$H^{v }$ be the $v$-dimensional Hausdorff measure.  
Then we have the all of the following. 
\begin{itemize}
\item[{\em (i)}] 
$0<H^{v}(J(G))<\infty $ and there exists a unique Borel probability measure 
$\nu $ on $J(\tilde{h})$ such that for each $\varphi \in C(J(\tilde{h}))$, 
$\int _{J(\tilde{h})}\sum _{(\alpha ,y)\in \tilde{h}^{-1}(\gamma ,x)}
\frac{\varphi (\alpha ,y)}{\| D(\alpha _{1})_{y}\| _{s}^{v}}\ d\nu (\gamma ,x)=
\int _{J(\tilde{h})} \varphi (\gamma, x) d\nu (\gamma ,x).$ Moreover, 
there exists a unique element $\psi \in C(J(\tilde{h}))$ with $\psi (\gamma ,x)>0\ (\forall (\gamma ,x))$ such that for each $(\gamma, x)\in J(\tilde{h})$, 
we have 
$\sum _{(\alpha ,y)\in \tilde{h}^{-1}(\gamma ,x)}
\frac{\psi (\alpha ,y)}{\| D(\alpha _{1})_{y}\| _{s}^{v}}=
\psi (\gamma, x)$. Also, $\eta := \psi \cdot \nu $ is an $\tilde{h}$-invariant ergodic Borel probability measure 
on $J(\tilde{h}).$ Further, $(\pi _{\CCI })_{\ast }(\nu )=\frac{H^{v}|_{J(G)}}{H^{v}(J(G))}.$   
%
\item[{\em (ii)}] 
For almost every  $z_{0}\in J(G)$ with respect to 
$H^{v}$,  the function $T(h_{1},h_{2},p,\cdot ):\CCI \rightarrow [0,1] $ is continuous at $z_{0}$ and 
$\mbox{{\em H\"{o}l}}(T(h_{1},h_{2},p,\cdot ),z_{0})\leq u(h_{1},h_{2},p,\eta )$. 
If $(h_{1},h_{2})\in {\cal D}\cap {\cal B}\cap {\cal H}$, 
then for almost every  $z_{0}\in J(G)$ with respect to 
$H^{v}$,  
$\mbox{{\em H\"{o}l}}(T(h_{1},h_{2},p,\cdot ),z_{0})=u(h_{1},h_{2},p,\eta )$.
\end{itemize} 
\end{enumerate} 
\end{thm}

\begin{thm}
\label{t:j1orj2}
Let $h=(h_{1},h_{2})\in (\overline{{\cal D}}\cap {\cal B}\cap {\cal H})\setminus {\cal I}.$ 
Let $G=\langle h_{1},h_{2}\rangle .$ 
Let $p\in (0,1).$ Let $p_{1}=p, p_{2}=1-p.$ 
Then 
there exists an open neighborhood $V$ of $(h_{1},h_{2})$ in ${\cal B}\cap {\cal H}$ 
and a number $i\in \{ 1,2\} $ such that for each $(g_{1},g_{2})\in V$, 
denoting by $\mu _{j}$ the maximal entropy measure for $g_{j}:\CCI \rightarrow \CCI $ for each $j=1,2,$  
all of the following hold. 
\begin{itemize}
\item[{\em (i)}]
 For each $j=1,2$, for $\mu _{j}$-a.e. $z_{0}\in J(g_{j})$,  
$\emHol(T(g_{1},g_{2},p,\cdot ),z_{0})\leq -\frac{\log p_{j}}{\log \deg (g_{j})}.$ 
\item[{\em (ii)}]
For $\mu _{i}$-a.e. $z_{0}\in J(g_{i})$, 
$\emHol(T(g_{1},g_{2},p,\cdot ),z_{0})\leq -\frac{\log p_{i}}{\log \deg (g_{i})}<1.$  
\item[{\em (iii)}] 
For each $\alpha \in (-\frac{\log p_{i}}{\log \deg (g_{i})}, 1)$ 
and for each $\varphi \in C^{\alpha}(\CCI )$ such that $\varphi (\infty )=1$ and $\varphi |_{\hat{K}(G)}\equiv 0$,  
we have $\| M_{g_{1},g_{2},p}^{n}(\varphi )\| _{\alpha }\rightarrow \infty $ as $n\rightarrow \infty .$ 
\end{itemize} 
\end{thm}

\begin{rem} 
Let $(h_{1},h_{2})\in {\cal D}\cap {\cal B}\cap {\cal H}$ be an element. 
By statements \ref{nondiff}, \ref{nonorable} in Theorem~\ref{t:spacend},  
it follows that if 
$p$ is close enough to $0$ or $1$, then 
(1) for almost every  $z_{0}\in J(h_{1},h_{2})$ with respect to 
$\lambda _{h_{1},h_{2},p}$, 
$z\mapsto T(h_{1},h_{2},p,z)$ is not differentiable at $z_{0}$,  
but (2) for almost every $z_{0}\in J(h_{1},h_{2})$ 
with respect to 
$H^{v}$, 
$z\mapsto T(h_{1},h_{2},p,z)$ is differentiable at $z_{0}$ and the derivative is equal to $0$.  
\end{rem}

\begin{rem}
\label{r:gradation} 
Theorems~\ref{t:timportant}, \ref{t:spacend}, \ref{t:j1orj2}  and \cite[Theorem 3.30]{Scp} 
imply that if $(h_{1},h_{2})\in (\overline{{\cal D}}\cap {\cal B}\cap {\cal H})\setminus {\cal I}$
 then there exists a neighborhood ${\cal A}_{0}$ of $(h_{1},h_{2})$ in ${\cal P}^{2}$ 
 such that for each $(g_{1},g_{2})\in {\cal A}_{0}$ and each $p\in (0,1)$, 
 the associated random dynamical system does not have chaos in the $C^{0}$ sense, 
but still has a kind of chaos in the $C^{\alpha }$ sense for some $0<\alpha <1$ 
(see also the similar results Theorems~\ref{t:spgenex}, \ref{t:oscanal}, \ref{t:oschg2} in 
which we do not assume hyperbolicity). 
More precisely, for such an element $(g_{1},g_{2},p)$, 
there exists a number $\alpha _{0}\in (0,1)$ such that 
for each $\alpha \in (0,\alpha _{0})$, 
the system behaves well on the Banach space $C^{\alpha }(\CCI )$ endowed with $\alpha$-H\"{o}lder 
norm $\| \cdot \| _{\alpha }$ (see \cite[Theorem 1.10]{Scp}),
 but for each $\alpha \in (\alpha _{0}, 1)$, 
the system behaves chaotically  
on the Banach space $C^{\alpha }(\CCI )$ (and on the Banach space $C^{1}(\CCI )$ as well).  
In this way, regarding the random dynamical systems, 
we have a kind of {\bf gradation between chaos and order}. Note that in \cite{Splms10,Scp},
this phenomenon was found for the systems with disconnected Julia sets, but in this paper, 
we show that this phenomenon can hold for plenty of systems with connected Julia sets. 
  For the related results in which we do not assume hyperbolicity, 
see Theorems~\ref{t:spgenex}, \ref{t:oscanal}, \ref{t:oschg2} and Remark~\ref{r:oschg3}. 
In \cite{JS1}, we perform a multifractal analysis of the pointwise \Hol der exponents of 
$T(h_{1},h_{2},p,\cdot )$ (and more general limit state functions of random complex dynamical systems). 

\end{rem}

\begin{rem}
From the point of view of \cite[Introduction]{Splms10} and \cite[Remark 1.14]{Scp}, we can say that 
the function $z\mapsto T(h_{1},h_{2},p,z)$ is a complex analogue 
of the devil's staircase or Lebesgue's singular functions and it is called a 
``{\bf devil's coliseum}''  
(see Figures~\ref{fig:0.41dc5grey} and \ref{fig:0.41dc5udgrey}, 
for the definition of the devil's staircase and Lebesgue's singular functions, see \cite{YHK}), 
and the function $z\mapsto (\partial T/\partial p)(h_{1},h_{2},p,z)$ 
(see Figure~\ref{fig:0.41ct2grey}) is a complex analogue of the Takagi function 
(or the Blancmange function) 
(for the definition of the Takagi function, see \cite{YHK}).  
These notions have been introduced in \cite{Splms10, Scp}, though in those papers 
we deal with the case having disconnected Julia sets. In this paper, we also deal with the case with 
{\bf connected} Julia sets which are thin fractal sets.   
\end{rem}
\begin{figure}[htbp]
\caption{The Julia set of $G=\langle h_{1}, h_{2}\rangle $, where   
$(h_{1},h_{2})\in ((\partial {\cal C})\cap {\cal B}\cap {\cal H})\setminus {\cal I}$ 
and $\deg (h_{1})=\deg (h_{2})=4.$ 
The Julia set $J(G)$ is {\bf connected}, $(h_{1},h_{2})$ satisfies the open set condition and 
$\frac{3}{2}<\dim _{H}(J(G))<2.$ }
\ \ \ \ \ \ \ \ \ \ \ \ \ \ \ \ \ \ \ \ \ \ \ \ \ \ \ \ \ \ \ 
\ \ \ \ \ \ \ \ \ \ \ \ \ \ \  \  \ \  \ \ \  \ \ 
\includegraphics[width=2.6cm,width=2.6cm]{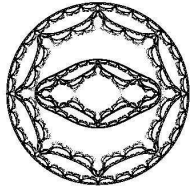}
\label{fig:0.41}
\end{figure} 

\begin{figure}[htbp]
\caption{The graph of $z\mapsto T(h_{1},h_{2},1/2,z)$, where   
$(h_{1},h_{2})$ is as in Figure~\ref{fig:0.41}.  
A devil's coliseum (a complex analogue of the devil's staircase or Lebesgue's singular functions).
The function is H\"{o}lder continuous on $\CCI $ and the set of varying points is equal to the 
{\bf connected} Julia set in Figure~\ref{fig:0.41}.}
\ \ \ \ \ \ \ \ \ \ \ \ \ \ \ \ \ \ \ \ \ \ \ \ \ \ \ \ \ \ \ 
\ \ \ \ \ \ \ \ \ \ \ \ 
\includegraphics[width=5.5cm,width=5.5cm]{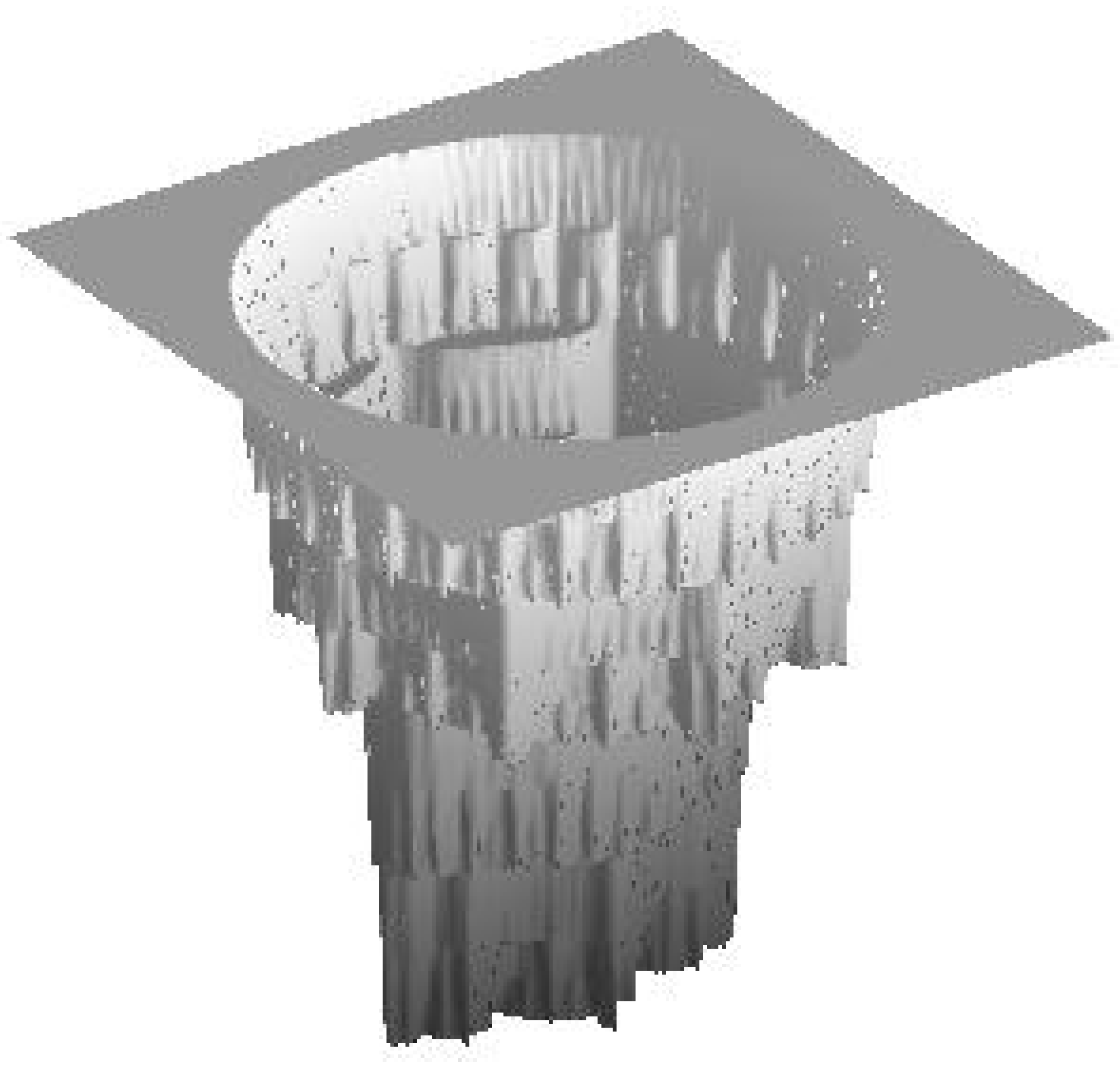}
\label{fig:0.41dc5grey}
\caption{Figure~\ref{fig:0.41dc5grey} upside down.}
\ \ \ \ \ \ \ \ \ \ \ \ \ \ \ \ \ \ \ \ \ \ \ \ \ \ \ \ \ \ \ 
\ \ \ \ \ \ \ \ \  
\includegraphics[width=6.0cm,width=6.0cm]{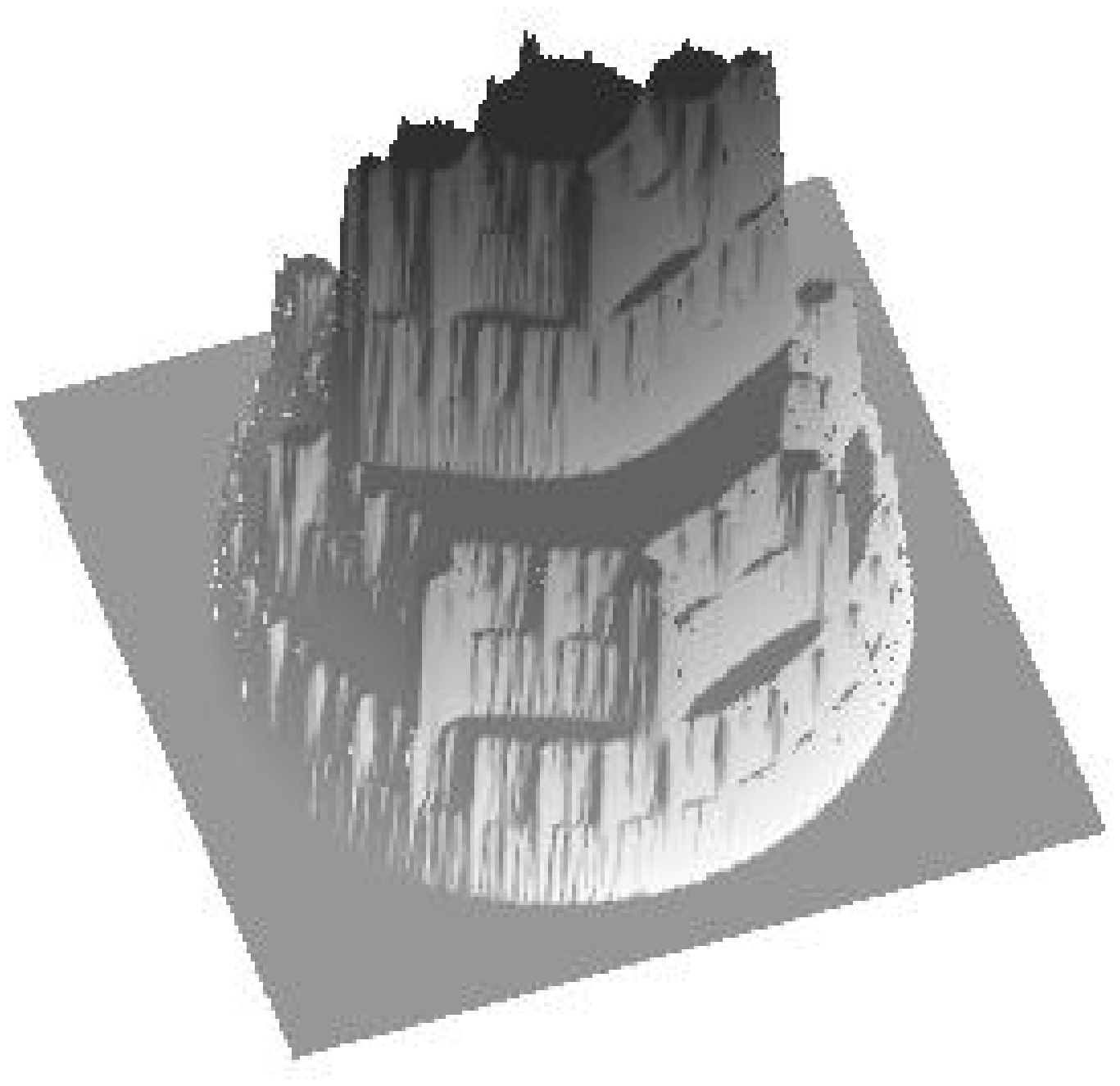}
\label{fig:0.41dc5udgrey}
\caption{The graph of $z\mapsto (\partial T/\partial p)(h_{1},h_{2},1/2,z)$,  
where   
$(h_{1},h_{2})$ is as in Figure~\ref{fig:0.41}. A complex analogue of 
the Takagi function (the Blancmange function).  
This function is H\"{o}lder continuous on $\CCI $ and the set of varying points is included in 
the Julia set in Figure~\ref{fig:0.41}.} 
\ \ \ \ \ \ \ \ \ \ \ \ \ \ \ \ \ \ \ \ \ \ \ \ \ \ \ \ \ \ \ 
\ \ \ \ \ \ \ \ \ \ \    
\includegraphics[width=5.5cm,width=5.5cm]{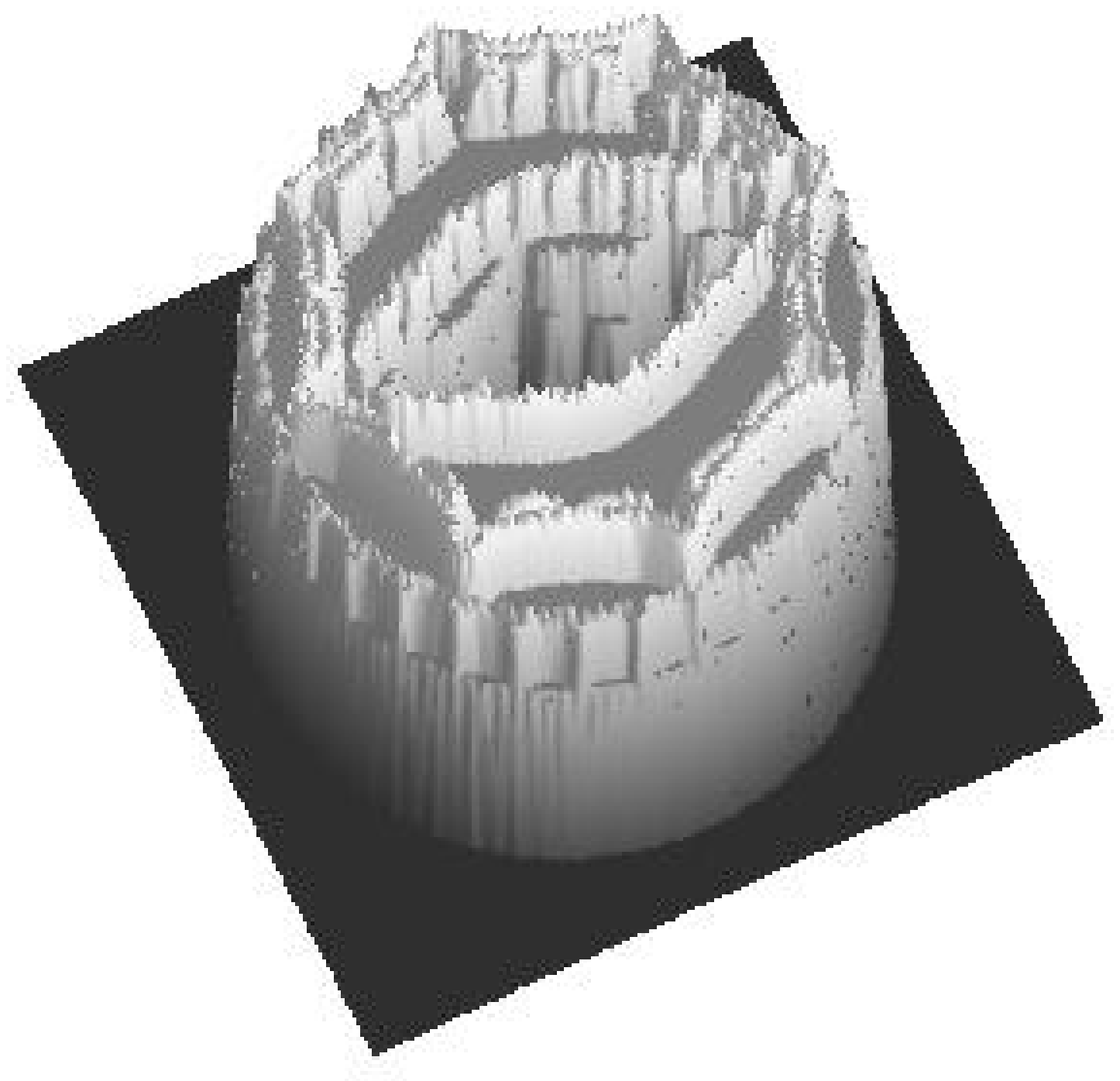}
\label{fig:0.41ct2grey}
\end{figure}
\begin{rem}
\label{r:summary}
This paper is the first one in which the parameter space of polynomial semigroups is investigated. 
We focus on the space of parameters for which the semigroup is postcritically bounded. 
In particular, we study the disconnectedness locus and the connectedness locus in the above space. 
We combine all ideas on postcritically bounded polynomial semigroups (\cite{SdpbpI, SS}), 
interaction cohomology (\cite{S15}), skew products and potential theory (\cite{S3, S4, Splms10}), 
(semi-)hyperbolic semigroups and thermodynamic formalisms 
(\cite{S6, SUb,SU2}),  ergodic theory, perturbation theory for linear operators, and random complex dynamics (\cite{Splms10, Scp}).    
In the proofs of the results, we use the idea of the nerve of backward images of the Julia set under the elements of a polynomial 
semigroup and associated cohomology (interaction cohomology) from \cite{S15} and we combine this with 
potential theory. Also, we use the results on the dynamics of postcritically bounded 
polynomial semigroups from \cite{SdpbpI} and the results on hyperbolic semigroups from 
\cite{S6}.  
From these, we obtain that any element 
$h=(h_{1},h_{2})\in (\overline{{\cal D}}\cap {\cal B}\cap {\cal H})\setminus {\cal Q}$ satisfies the open set condition 
 and Bowen's formula (Theorem~\ref{spacemainth}-\ref{spacemainthosc}, \ref{spacemainth4}). This is crucial to obtain other results in 
this paper. 
Combining this with some geometric observations by using Green's functions and a result on the porosity of the Julia sets from \cite{S7}, we obtain $\dim _{H}(J(h_{1},h_{2}))<2$ (Theorem~\ref{spacemainth}-\ref{spacemainth3}) 
and Theorem~\ref{spacemainth}-\ref{spacemainth4}.  
Moreover, we obtain that there exists a neighborhood ${\cal A}_{0}$ of 
$(\overline{\cal D}\cap {\cal B}\cap {\cal H})\setminus {\cal I}$ such that 
for each $(h_{1},h_{2})\in {\cal A}_{0}$, the set $\{ h_{1},h_{2}\} $ is 
mean stable (Theorem~\ref{t:timportant}-\ref{spacemainth10}-(i)). 

Also, it is important and interesting to study the topology of the connectedness locus and its boundary. 
Combining $\overline{\mbox{int}({\cal C})}\cap {\cal B}\cap {\cal H}={\cal C}\cap {\cal B}\cap {\cal H}$ 
(Theorem~\ref{t:spacetopology}-\ref{spacemainth2}), that 
$((\partial {\cal C})\cap {\cal B}\cap {\cal H})\setminus {\cal Q}$ is dense in 
$(\partial {\cal C})\cap {\cal B}\cap {\cal H}$ (Theorem~\ref{t:spacetopology}-\ref{spacemainth7}) and 
Theorem~\ref{spacemainth}-\ref{spacemainth4}, we see that  
in any neighborhood of any point of $(\partial {\cal C})\cap {\cal B}\cap {\cal H}$, 
there exists an open subset ${\cal A}_{1}$ of $\mbox{int}({\cal C})\cap {\cal B}\cap {\cal H}$ 
such that for each $g=(g_{1},g_{2})\in {\cal A}_{1}$, 
we have that $\dim _{H}(J(g_{1},g_{2}))<2$. 
Combining the results on semigroups and some results on random complex dynamics from 
\cite{Splms10, Scp}, and developing many new ideas,  
we show Theorems~\ref{t:timportant}, \ref{t:spacend}, \ref{t:j1orj2}. 
Note that Theorem~\ref{t:timportant} (vii)(viii) (complex analogues of the Takagi function) 
are obtained by using some results from \cite{Scp}, which were shown by using the perturbation 
theory for linear operators. 
Combining the above results on semigroups and Theorem~\ref{t:timportant}, we obtain that   
 in any neighborhood of any point of $((\partial {\cal C})\cap {\cal B}\cap {\cal H})\setminus {\cal I}$, 
there exists an open subset ${\cal A}_{2}$ of $\mbox{int}({\cal C})\cap {\cal B}\cap {\cal H}$ 
such that for each $g=(g_{1},g_{2})\in {\cal A}_{2}$, 
$\{ g_{1},g_{2}\} $ is mean stable and 
 the function $T(g_{1},g_{2},p,\cdot )$ of probability of tending to $\infty $ 
is H\"{o}lder continuous on $\CCI $ and varies only in a thin connected fractal set $J(g_{1},g_{2})$ 
whose Hausdorff dimension is strictly less than two. 
Moreover, if $(g_{1},g_{2})\in ((\partial {\cal C})\cap {\cal B}\cap {\cal H})\setminus {\cal I}$, then 
$T(g_{1},g_{2},p,\cdot )$ is H\"{o}lder continuous on $\CCI $ 
and varies precisely on the connected Julia set $J(g_{1},g_{2})$ 
which is a thin fractal set (Theorems~\ref{t:timportant}, \ref{t:spaceth1h2}, \ref{spacemainth}).  
Thus, we can say that even in the parameter region ${\cal C}\cap {\cal B}\cap {\cal H}$, 
there are plenty of examples of ``{\bf complex singular functions}'' (complex analogues of the devil's staircase 
or Lebesgue singular functions). Note that in \cite{Splms10, Scp}, 
we have obtained many results on the functions of probability of 
tending to $\infty $ when the associated Julia sets are ``disconnected''.  
This paper is the first one in which we show the existence of plenty of examples 
such that  
the function of probability of tending to $\infty $ is continuous and ``singular'' when the associated 
Julia set is connected. (In fact, if the overlap $h_{1}^{-1}(J(h_{1},h_{2}))\cap h_{2}^{-1}(J(h_{1},h_{2}))$
 is not empty, then the analysis of the random dynamical systems generated by $\{ h_{1},h_{2}\} $ and $T(h_{1},h_{2},p,\cdot )$ is  difficult.) 
Also, the details of such functions are given in Theorem~\ref{t:timportant}.    
\end{rem}
In section~\ref{Proofs}, we give the proofs of the main results.  
Also, in section~\ref{Proofs}, we show some further results 
in which we do not assume hyperbolicity (Theorems~\ref{t:shshdhl2}, \ref{t:spgenex},  
\ref{t:oscanal}, \ref{t:oscholgen}, \ref{t:oschg2}).  Moreover, 
we show a result on the Fatou components (Theorem~\ref{t:infcomp}) 
and some results on semi-hyperbolicity 
(Theorems~\ref{t:shshdhl2}, \ref{t:spgenex}). 

\ 

\noindent {\bf Acknowledgement.} The author thanks Rich Stankewitz for valuable comments. 
This research was partially supported by 
JSPS KAKENHI 24540211, 15K04899. 
\section{Proofs of the main results}
\label{Proofs} 
In this section, we give the proofs of the main results. Also, we show some further results 
in which we do not assume hyperbolicity. 

To recall some known facts on polynomial semigroups, 
let $G$ be a polynomial semigroup with $G\subset {\cal P}.$ 
Then, for each $g\in G$, $g(F(G))\subset F(G), g^{-1}(J(G))\subset J(G).$ 
Moreover, $J(G)$ is a perfect set and 
$J(G)$ is equal to the closure of the set of repelling cycles of elements of $G$. 
In particular, $J(G)=\overline{\cup _{g\in G}J(g)}$. 
We set $E(G):=\{ z\in \CCI \mid \sharp \bigcup _{g\in G}g^{-1}(\{ z\} )<\infty \} .$ 
Then $\sharp E(G)\leq 2$ and 
for each $z\in J(G)\setminus E(G)$, $J(G)=\overline{\bigcup _{g\in G}g^{-1}(\{ z\} ) }.$ 
Also, $J(G)$ is the smallest set in 
$\{ \emptyset \neq K\subset \CCI \mid K \mbox{ is compact}, \forall g\in G, g(K)\subset K\} $ 
with respect to the inclusion. 
If $G$ is generated by a compact family $\Lambda $ of ${\cal P}$, then 
$J(G)=\bigcup _{h\in \Lambda }h^{-1}(J(G))$ (this is called the backward self-similarity). 
For more details on these  
properties of polynomial semigroups, 
see \cite{HM, St3, GR, S3}.  
The article \cite{St3} by R. Stankewitz is a very nice introductory one for 
basic facts on polynomial semigroups.  
For the properties of the dynamics of postcritically bounded polynomial semigroups, 
see \cite{SdpbpI, SS}. 

 For fundamental tools and lemmas  of random complex dynamics, see \cite{Splms10,Scp}.  
%
%
\vspace{-2mm} 
\subsection{Proofs of Theorems~\ref{spacemainth} and \ref{t:spacetopology}}
\label{Pfspacemainth}
\vspace{-2mm} 
In this subsection, we prove Theorems~\ref{spacemainth} and \ref{t:spacetopology}. 
Also, we show some results in which we do not assume hyperbolicity 
(Lemma~\ref{l:oscinte}, Theorems~\ref{t:shshdhl2}, \ref{t:spgenex}).  
We need some definitions. 
\begin{df}[\cite{SdpbpI}] 
For any connected sets $K_{1}$ and 
$K_{2}$ in $\CC ,\ $  we write $K_{1}\leq _{s}K_{2}$ to indicate that 
$K_{1}=K_{2}$, or $K_{1}$ is included in 
a bounded component of $\CC \setminus K_{2}.$ Furthermore, 
$K_{1}<_{s}K_{2}$ indicates $K_{1}\leq _{s}K_{2}$ 
and $K_{1}\neq K_{2}.$ 
Moreover, $K_{2}\geq _{s}K_{1}$ indicates $K_{1}\leq _{s}K_{2}$,  
and $K_{2}>_{s}K_{1}$ indicates $K_{1}<_{s}K_{2}.$  
Note that 
$ \leq _{s}$ is a partial order in 
the space of all non-empty compact connected 
sets in $\CC .$ This $\leq _{s}$ is called 
the {\bf surrounding order.}

\end{df}
\begin{df}
For a topological space $X$, we denote by Con$(X)$ the set of all connected 
components of $X.$ 
\end{df}
\begin{df}
We denote by ${\cal G}$ the set of all postcritically bounded 
polynomial semigroups $G$ 
with $G\subset {\cal P}.$ 
We denote by ${\cal G}_{dis}$ the set of all 
$G\in {\cal G}$ with disconnected Julia set. 

\end{df} 
\begin{rem}
\label{r:icjto}
Let $G\in {\cal G}_{dis}.$ 
In \cite{SdpbpI}, it was shown that 
$J(G)\subset \CC $, $(\mbox{Con}(J(G)), \leq _{s})$ is totally ordered, 
there exists a unique maximal element $J_{\max }=J_{\max}(G)\in (\mbox{Con}(J(G)),\leq _{s})$, there exists 
a unique minimal element $J_{\min }=J_{\min }(G)\in (\mbox{Con}(J(G)),\leq _{s})$, 
and each element of $\mbox{Con}(F(G))$ 
is either simply connected or doubly connected.  
Moreover, in \cite{SdpbpI}, it was shown that
${\cal A}\neq \emptyset $, where  
${\cal A}$ denotes the set of all doubly connected components of $F(G)$  
(more precisely, for each $J,J'\in \mbox{Con}(J(G))$ with 
$J<_{s}J'$, there exists an $A\in {\cal A}$ with $J<_{s}A<_{s}J'$), 
$\bigcup _{A\in {\cal A}}A\subset \CC $, and $({\cal A},\leq _{s})$ is totally 
ordered. Note that each $A\in {\cal A}$ is bounded and multiply connected, while for a single $f\in {\cal P}$, 
we have no bounded multiply connected component of $F(f).$   
\end{rem}
\begin{df}
Let $\gamma =(\gamma _{1},\gamma _{2},\ldots )\in {\cal P}^{\NN}$ be a sequence of 
polynomials. For each $m,n\in \NN $ with $n\leq m$, we set $\g _{m,n}:=\g _{m}\circ \cdots \circ \g _{n}.$ 
We denote by $F_{\gamma }$ the set of points $z\in \CCI $ for which 
there exists a neighborhood $U$ of $z$ such that 
$\{ \gamma _{n,1}\} _{n\in \NN }$ is 
normal. The set $F_{\gamma }$ is called the Fatou set of the sequence 
$\gamma $ of polynomials. 
Moreover, we set $J_{\gamma }:= \CCI \setminus F_{\gamma }.$ 
The set $J_{\gamma }$ is called the Julia set of the sequence $\gamma $ 
of polynomials. 
\end{df}
\begin{lem}
\label{l:bdjconti}
Let $(h_{1},h_{2})\in {\cal B}\cap {\cal D}.$ 
Let $\G =\{ h_{1},h_{2}\} $ and $G=\langle h_{1},h_{2}\rangle .$  
Then, $h_{1}^{-1}(J(G))\cap h_{2}^{-1}(J(G))=\emptyset $, 
$J(G)=\amalg _{\g \in \GN }J_{\g }$ (disjoint union), 
and the map $\g \mapsto J_{\g }$ is continuous on $\GN $ with respect to the 
Hausdorff metric.  Moreover, for each $\g \in \GN $, $J_{\g }$ is connected. 
\end{lem}
\begin{proof}
By \cite[Proposition 2.24]{SdpbpI}, we may assume that $J(h_{1})\subset J_{\min }(G)$ and $J(h_{2})\subset J_{\max }(G).$ 
Then by \cite[Theorem 2.20-5]{SdpbpI}, 
$\emptyset \neq \mbox{int}(\hat{K}(G))\subset \mbox{int}(K(h_{1})).$ 
Here, for any subset $A$ of $\CCI $, int$(A)$ denotes the set of interior points 
of  $A$ with respect to the topology in $\CCI .$ 
By \cite[Theorem 2.20-5]{SdpbpI} again, 
$h_{2}(J(h_{1}))\subset h_{2}(J_{\min }(G))\subset \mbox{int}(\hat{K}(G))\subset \mbox{int}(K(h_{1})).$ 
Therefore
$h_{2}(\mbox{int}(K(h_{1})))\subset \mbox{int}(K(h_{1})).$ 
Thus $\mbox{int}(K(h_{1}))\subset F(G).$ 

By \cite[Theorems 1.7, 1.5]{S15}, 
$h_{1}^{-1}(J(G))\cap h_{2}^{-1}(J(G))=\emptyset .$ 
Since $J(G)=h_{1}^{-1}(J(G))\cup h_{2}^{-1}(J(G))$ (\cite[Lemma 1.1.4]{S1}), 
it follows that $J(G)=\amalg _{\gamma \in \GN} J(G)_{\gamma }$, 
where $J(G)_{\gamma }:= \cap _{n=1}^{\infty }(\gamma _{1}^{-1}\cdots \gamma _{n}^{-1}(J(G))).$ 
Since $J_{\gamma }= 
\gamma _{1}^{-1}\cdots \gamma _{n}^{-1}
(J_{(\gamma _{n+1},\gamma _{n+2},\ldots )})\subset 
\gamma _{1}^{-1}\cdots \gamma _{n}^{-1}(J(G))$ for each 
$\gamma $ and each $n\in \NN $, it follows that 
for each $\g \in \GN $, $J_{\g }\subset J(G)_{\g }.$ 
 
Let $\g \in \GN .$ In order to prove $J_{\g }=J(G)_{\g }$, 
suppose that there exists a point $y_{0}\in J(G)_{\g }\setminus J_{\g }.$ 
We now consider the following two cases. 
Case 1: $\sharp \{ n\in \NN \mid \g _{n}\neq h_{1}\} =\infty .$ 
Case 2: $\sharp \{ n\in \NN \mid \g _{n}\neq h_{1}\} <\infty .$ 

 Suppose that we have Case 1. 
Then there exist an open neighborhood $U$ of $y_{0}$ in $\CCI $, 
a strictly increasing sequence $\{ n_{j}\} _{j=1}^{\infty }$ of positive integers, 
and a map $\varphi : U\rightarrow \CCI $, such that  
$\g _{n_{j}+1}=h_{2}$ for each $j\in \NN $, and such that 
$\g _{n_{j},1}\rightarrow \varphi $ uniformly on $U$ as $j\rightarrow \infty .$ 
Since $\gamma _{n_{j},1}(y_{0})\in J(G)$ for each $j$, and since int$(\hat{K}(G))\subset F(G)$,  
\cite[Lemma 5.6]{SdpbpII} 
implies that  
$\varphi $ is constant. 
By \cite[Lemma 3.13]{SdpbpIII}, it follows that 
$d(\g _{n_{j},1}(y_{0}), P^{\ast }(G))\rightarrow 0$ as $j\rightarrow \infty .$ 
Moreover, since $\g _{n_{j}+1}=h_{2}$, 
we obtain $\g _{n_{j},1}(y_{0})\in h_{2}^{-1}(J(G))$ for each $j.$ 
Furthermore, by \cite[Theorem 2.20-2,5]{SdpbpI}, $h_{2}^{-1}(J(G))\subset \CCI \setminus P^{\ast }(G).$ 
This is a contradiction. Hence, we cannot have Case 1. 

 Suppose we have Case 2. 
Let $r\in \NN $ be a number such that 
for each $s\in \NN $ with $s\geq r$, $\g _{s}=h_{1}.$ 
Then $h_{1}^{n}(\g _{r,1}(y_{0}))\in J(G)$ for each $n\geq 0.$ 
Since $y_{0}\not\in J_{\g }$, 
we have $\g _{r,1}(y_{0})\not\in J(h_{1}).$ 
Moreover, since $\g_{r,1}(y_{0})\in J(G)$ and int$(\hat{K}(h_{1}))\subset F(G)$, 
it follows that $\g _{r,1}(y_{0})$ belongs to $F_{\infty }(h_{1})$. 
It implies that  $h_{1}^{n}(\g _{r,1}(y_{0}))\rightarrow \infty $ as $n\rightarrow \infty $. 
However, this contradicts that $h_{1}^{n}(\g _{r,1}(y_{0}))\in J(G)$ for each $n\geq 0.$ 
Therefore, we cannot have Case 2. 

Thus, for each $\g \in \G $, $J_{\g }=J(G)_{\g }.$ 
Combining this with \cite[Lemma 3.5]{SdpbpIII} and 
\cite[Proposition 2.2(3)]{S7}, 
we obtain that the map $\g \mapsto J_{\g }$ is continuous on $\GN $ with respect to the Hausdorff metric.  

 Finally, by \cite[Lemma 3.8]{SdpbpIII}, $J_{\g }$ is connected for each $\g \in \GN .$ 
Thus we have proved our lemma. 
\end{proof}

\begin{prop}
\label{p:dosc}
Let $(h_{1},h_{2})\in {\cal B}\cap {\cal D}.$ Then 
either {\em (1)} $K(h_{1})\subset \mbox{{\em int}}(K(h_{2}))$ 
or {\em (2)} 
$K(h_{2})\subset \mbox{{\em int}}(K(h_{1}))$ holds. 
If {\em (1)} holds, then setting $U:=(\mbox{{\em int}}(K(h_{2})))\setminus K(h_{1})$,
we have that $(h_{1},h_{2})$ satisfies the open set condition with $U$, 
that $K(h_{1})\subset h_{1}^{-1}(K(h_{2}))\subset h_{2}^{-1}(K(h_{1}))\subset K(h_{2})$,  
that $h_{2}$ is hyperbolic,  
and  that $J(h_{2})$ is a quasicircle. 
\end{prop}
\begin{proof}

Let $\G :=\{ h_{1},h_{2}\} .$ 
Let $G=\langle h_{1},h_{2}\rangle .$ 
By Lemma~\ref{l:bdjconti}, we have $J(G)=\amalg _{\g \in \GN }J_{\g }.$ 
Combining this with \cite[Theorem 2.7]{SdpbpI}, we obtain that 
 either $J(h_{1})<_{s}J(h_{2})$ or $J(h_{2})<_{s}J(h_{1})$ holds.  
 We now assume that $J(h_{1})<_{s}J(h_{2})$ (which is equivalent to  $K(h_{1})\subset \mbox{int}(K(h_{2}))$).  
Then, by \cite[Proposition 2.24]{SdpbpI}, 
 $J(h_{1})\subset J_{\min }(G)$ and $J(h_{2})\subset J_{\max }(G).$ 
By Lemma~\ref{l:bdjconti} again, 
it follows that 
 $J(h_{1})=J_{\min }(G)$ and $J(h_{2})=J_{\max}(G).$ Let $A=K(h_{2})\setminus \mbox{int}(K(h_{1})).$ 
 We now prove the following claim. \\ 
Claim 1. $h_{1}^{-1}(A)\cup h_{2}^{-1}(A)\subset A.$ 

 To prove this claim, let $\alpha =(h_{2},h_{1},h_{1},\ldots )\in \GN .$ 
Then $J_{\alpha }=h_{2}^{-1}(J(h_{1})).$ 
Since $J(h_{1})=J_{\min }(G)$, 
we obtain that 
$J(h_{1})<_{s}J_{\alpha }=h_{2}^{-1}(J(h_{1})).$ 
Therefore $h_{2}^{-1}(A)\subset A.$ Similarly, 
letting $\beta =(h_{1},h_{2},h_{2},\ldots )\in \GN $, 
we have $J_{\beta }=h_{1}^{-1}(J(h_{2}))<_{s}J(h_{2})$ and 
$h_{1}^{-1}(A)\subset A.$ Thus we have proved Claim 1. 

 We have that $h_{1}^{-1}(A) $ and $h_{2}^{-1}(A)$ are connected compact sets. 
We prove the following claim.

\noindent Claim 2. $J_{\beta }=h_{1}^{-1}(J(h_{2}))<_{s}J_{\alpha } =h_{2}^{-1}(J(h_{1})).$ In particular, 
$h_{1}^{-1}(A)<_{s}h_{2}^{-1}(A).$ 

 To prove this claim, suppose that $J_{\beta }<_{s}J_{\alpha }$ does not hold. 
 Then by \cite[Theorem 2.7]{SdpbpI}, 
 $J_{\alpha }<_{s}J_{\beta }.$ 
 This implies that $A=h_{1}^{-1}(A)\cup h_{2}^{-1}(A).$ 
 By \cite[Corollary 3.2]{HM}, we have $J(G)\subset A.$ 
 Since $J(G)$ is disconnected (assumption) and since $A$ is connected, 
 $F(G)\cap A\neq \emptyset .$ 
 Let $y\in F(G)\cap A.$ Since $A=h_{1}^{-1}(A)\cup h_{2}^{-1}(A)$, 
there exists an element $\gamma \in \GN $ such that 
for each $n\in \NN $, $\g _{n,1}(y)\in A.$ 
Since $y\in A\cap F(G)$ and $\bigcup _{g\in G}g(F(G))\subset F(G)$, 
  $\gamma _{n,1}(y)\in F_{\infty }(h_{1})\cap A$ for each $n\in \NN .$ 
Therefore there exists a strictly increasing sequence $\{ n_{j}\} _{j=1}^{\infty }$ in $\NN $ 
such that 
for each $j$, $\g _{n_{j}+1}=h_{2}.$ 
Since $y\in F_{\g }$, we may assume that there exists an open neighborhood $U$ of $y$ in $\CCI $ 
and a holomorphic map $\varphi :U\rightarrow \CCI $ such that 
$\g _{n_{j},1}\rightarrow \varphi $ uniformly on $U$ as $j\rightarrow \infty .$ 
Since $\g _{n_{j},1}(y)\in F_{\infty }(h_{1})\cap A\subset (\CCI \setminus \hat{K}(G))\cap A$ for each $j$, 
\cite[Lemma 5.6]{SdpbpII} implies 
that 
there exists a constant $c\in \CC $ such that $\varphi =c$ on $U.$ 
By \cite[Lemma 3.13]{SdpbpIII}, 
it follows that $c\in P^{\ast }(G).$ 
Moreover, $\gamma _{n_{j}+1,1}\rightarrow h_{2}(c)\in P^{\ast }(G)$ as $j\rightarrow \infty $ 
uniformly on $U$.  
Since $P^{\ast }(G)\subset K(h_{1})$ and 
since $\g _{n_{j}+1,1}(y)\in F_{\infty }(h_{1})$ for each $j$, it follows that 
$d(\g _{n_{j}+1,1}(y),J(h_{1}))\rightarrow 0$ as $j\rightarrow \infty .$ 
Combining this with that $\g _{n_{j}+1}=h_{2}$ for each $j$, we obtain that 
$d(\g _{n_{j},1}(y), h_{2}^{-1}(J(h_{1})))\rightarrow \infty .$ 
Since $J(h_{1})<_{s}h_{2}^{-1}(J(h_{1}))$, 
it follows that $c\in F_{\infty }(h_{1}).$ However, this is a contradiction, 
Hence, $J_{\beta }<_{s}J_{\alpha }.$ Thus we have proved Claim 2. 

From Claims 1 and 2, we obtain that 
$(h_{1},h_{2})$ satisfies the open set condition with $U$ and 
$K(h_{1})\subset h_{1}^{-1}(K(h_{2}))\subset h_{2}^{-1}(K(h_{1}))\subset K(h_{2})$. 

By \cite[Theorem 2.20-4]{SdpbpI}, $h_{2}$ is hyperbolic and $J(h_{2})$ is a quasicircle. 

Thus we have proved our proposition. 
\end{proof}
\begin{lem}
\label{l:pfmainth1}
Statement~\ref{spacemainth1} in Theorem~\ref{spacemainth} holds.  
\end{lem}
\begin{proof}
Statement~\ref{spacemainth1} in Theorem~\ref{spacemainth} follows from Lemma~\ref{l:bdjconti} and Proposition~\ref{p:dosc}. 
\end{proof}
The following proposition is the key to proving many results in this paper. 
\begin{prop}
\label{p:bdyosc}
Let $(h_{1},h_{2})\in \overline{{\cal D}}\cap ({\cal B}\cap {\cal H}).$ 
Then, either {\em (1)} $K(h_{1})\subset K(h_{2})$ or {\em (2)} 
$K(h_{2})\subset K(h_{1})$ holds. 
If {\em (1)} holds then, setting $U_{h_{1},h_{2}}:=(\mbox{{\em int}}(K(h_{2})))\setminus K(h_{1})$, 
we have all of the following. 
\begin{itemize}
\item[{\em (a)}] 
$K(h_{1})\subset h_{1}^{-1}(K(h_{2}))\subset h_{2}^{-1}(K(h_{1}))\subset K(h_{2}).$ 

\item[{\em (b)}] 
$h_{1}^{-1}(U_{h_{1},h_{2}})\amalg h_{2}^{-1}(U_{h_{1},h_{2}})\subset U_{h_{1},h_{2}}.$ 

\item[{\em (c)}] 
$J(h_{2})$ is a quasicircle. 
\end{itemize}

\end{prop}
\begin{proof}
Let $(h_{1},h_{2})\in \overline{{\cal D}}\cap ({\cal B}\cap {\cal H}).$ 
Then there exists a sequence $\{ (h_{1,n}, h_{2,n})\} _{n\in \NN }$ 
in ${\cal D}\cap {\cal H}\cap {\cal B} $ such that 
$(h_{1,n},h_{2,n})\rightarrow (h_{1},h_{2})$ as $n\rightarrow \infty .$ 
By Proposition~\ref{p:dosc}, we may assume that for each $n\in \NN $, 
$K(h_{1,n})\subset \mbox{int}(K(h_{2,n}))$, 
\begin{equation}
\label{eq:kh1n}
K(h_{1,n})\subset h_{1,n}^{-1}(K(h_{2,n}))\subset h_{2,n}^{-1}(K(h_{1,n}))\subset K(h_{2,n}),
\end{equation}   
 and 
$h_{1,n}^{-1}(U_{n})\amalg h_{2,n}^{-1}(U_{n})\subset U_{n}$, 
where $U_{n}:= (\mbox{int}(K(h_{2,n})))\setminus K(h_{1,n})$. 
 Suppose that $h_{1}^{-1}(K(h_{2}))\cap (\CCI \setminus h_{2}^{-1}(K(h_{1})))\neq \emptyset .$ 
Then, 
\begin{equation}
\label{eq:bdyh1k}
\partial (h_{1}^{-1}(K(h_{2})))\setminus h_{2}^{-1}(K(h_{1}))\neq \emptyset .
\end{equation} 
For, let $z\in h_{1}^{-1}(K(h_{2}))\cap (\CCI \setminus h_{2}^{-1}(K(h_{1})))\neq \emptyset $ be a point. 
If $z\in \mbox{int}(h_{1}^{-1}(K(h_{2})))$, 
then since $\CCI \setminus h_{2}^{-1}(K(h_{1}))$ is connected (this is because 
$(h_{1},h_{2})\in {\cal B}$),  
there exists a curve $\gamma :[0,1] \rightarrow \CCI \setminus h_{2}^{-1}(K(h_{1}))$ 
such that $\gamma (0)=z\in \mbox{int}(h_{1}^{-1}(K(h_{2})))$ and 
$\gamma (1)=\infty \in \CCI \setminus h_{1}^{-1}(K(h_{2})).$ Therefore there exists a 
$t\in [0,1]$ with $\gamma (t)\in \partial (h_{1}^{-1}(K(h_{2}))).$ 
Thus (\ref{eq:bdyh1k}) holds. 
Let $V$ be an open disc neighborhood of $\infty $ such that for each 
$n\in \NN $, 
$V\subset F(h_{1,n},h_{2,n})\cap F(h_{1},h_{2}).$ 
By (\ref{eq:bdyh1k}), there exist a point $w\in h_{1}^{-1}(J(h_{2}))$ and a positive integer 
$l$ 
such that $h_{1}^{l}h_{2}(w)\in V.$ 
Since $J(h_{2,n})\rightarrow J(h_{2})$ as $n\rightarrow \infty $ with respect to the 
Hausdorff metric, for a sufficiently large $n\in \NN $, there exists a point 
$w_{0}\in h_{1,n}^{-1}(J(h_{2,n}))$ such that $h_{1,n}^{l}h_{2,n}(w_{0})\in V$. 
Therefore $w_{0}\in \CC \setminus h_{2,n}^{-1}(K(h_{1,n})).$ 
Hence, $h_{1,n}^{-1}(J(h_{2,n}))\cap (\CC \setminus h_{2,n}^{-1}(K(h_{1,n})))\neq \emptyset $. 
However, this contradicts (\ref{eq:kh1n}). Thus, we should have that 
$h_{1}^{-1}(K(h_{2}))\subset h_{2}^{-1}(K(h_{1})).$ 
Similarly, by using the fact $K(h_{1,n})\subset \mbox{int}(K(h_{2,n}))$ for each $n$, 
 we obtain $K(h_{1})\subset K(h_{2}).$ 
Therefore, 
$K(h_{1})\subset h_{1}^{-1}(K(h_{1}))\subset h_{1}^{-1}(K(h_{2}))\subset h_{2}^{-1}(K(h_{1}))
\subset h_{2}^{-1}(K(h_{2}))=K(h_{2}).$ Thus statement (a) holds.  

 We next prove statement (b). Let $U=(\mbox{int}(K(h_{2})))\setminus K(h_{1}).$ 
By statement (a), we have 
\begin{align}
\label{eq:h1-uint} 
h_{1}^{-1}(U)  = & h_{1}^{-1}((\mbox{int}(K(h_{2})))\setminus K(h_{1}))
              =  (\mbox{int}(h_{1}^{-1}(K(h_{2}))))\setminus h_{1}^{-1}(K(h_{1})) 
              =  \mbox{int}(h_{1}^{-1}(K(h_{2}))))\setminus K(h_{1})\\  
              \subset & (\mbox{int}(K(h_{2})))\setminus K(h_{1})=U.
\end{align} 
Similarly, by statement (a), we have 
\begin{align}
\label{eq:h2-uint}
h_{2}^{-1}(U) = & h_{2}^{-1}((\mbox{int}(K(h_{2})))\setminus K(h_{1}))
              = h_{2}^{-1}(\mbox{int}(K(h_{2})))\setminus h_{2}^{-1}(K(h_{1})) 
              =  (\mbox{int}(K(h_{2})))\setminus h_{2}^{-1}(K(h_{1}))\\ 
             \subset & (\mbox{int}(K(h_{2})))\setminus K(h_{1})=U.    
\end{align}
Moreover, since we have $h_{1}^{-1}(K(h_{2}))\subset h_{2}^{-1}(K(h_{1}))$ (statement (a)), 
 
\begin{equation}
\label{eq:inth1kh2}
(\mbox{int}(h_{1}^{-1}(K(h_{2}))))\cap ((\mbox{int}(K(h_{2})))\setminus h_{2}^{-1}(K(h_{1})))=\emptyset .
\end{equation}
Combining (\ref{eq:h1-uint}), (\ref{eq:h2-uint}), and (\ref{eq:inth1kh2}), 
we obtain $h_{1}^{-1}(U)\cap h_{2}^{-1}(U)=\emptyset .$ Therefore statement (b) holds. 
Since $J(h_{2,n})$ is a quasicircle for each $n$, since $h_{2}$ is hyperbolic, and since 
$h_{2,n}\rightarrow h_{2}$, we obtain that $J(h_{2})$ is a quasicircle. Hence statement (c) holds. 

 Thus we have proved Proposition~\ref{p:bdyosc}.
\end{proof}
\begin{lem}
\label{l:pfmainthosc}
Statement~\ref{spacemainthosc} in Theorem~\ref{spacemainth} holds.  
\end{lem}
\begin{proof}
Let $(h_{1},h_{2})\in (\overline{{\cal D}}\cap {\cal B}\cap {\cal H})\setminus {\cal Q}.$ 
Then by Proposition~\ref{p:bdyosc}, either $K(h_{1})\subset K(h_{2})$ or $K(h_{2})\subset K(h_{1}).$ 
We may assume $K(h_{1})\subset K(h_{2}).$ 
By Proposition~\ref{p:bdyosc} again, we obtain that setting 
$U:= (\mbox{int}(K(h_{2})))\setminus K(h_{1})$, 
$h_{1}^{-1}(U)\amalg h_{2}^{-1}(U)\subset U.$ Moreover, 
$J(h_{2})$ is a quasicircle. 
Since $(h_{1},h_{2})\notin {\cal Q}$, 
it follows that $K(h_{1})\subsetneqq K(h_{2}).$ Therefore, $U\neq \emptyset $. 
Thus, $(h_{1},h_{2})$ satisfies the open set condition with $U.$ 
Hence statement~\ref{spacemainthosc} in Theorem~\ref{spacemainth} holds. 
\end{proof} 
\begin{lem}
\label{l:h1h2bdhb}
Let $(h_{1},h_{2})\in (\partial {\cal D})\cap {\cal B}\cap {\cal H}.$ 
Then $h_{1}^{-1}(J(h_{2}))\cap h_{2}^{-1}(J(h_{1}))\neq \emptyset .$ 
\end{lem}
\begin{proof}
Suppose $h_{1}^{-1}(J(h_{2}))\cap h_{2}^{-1}(J(h_{1}))= \emptyset .$
By Proposition~\ref{p:bdyosc}, we may assume that 
$K(h_{1})\subset K(h_{2}).$ 
Combining Proposition~\ref{p:bdyosc} and  
$h_{1}^{-1}(J(h_{2}))\cap h_{2}^{-1}(J(h_{1}))= \emptyset $, 
we obtain that 
\begin{equation}
\label{eq:h1kh2sint}
h_{1}^{-1}(K(h_{2}))\subset \mbox{int}(h_{2}^{-1}(K(h_{1}))). 
\end{equation}
Let $U=(\mbox{int}(K(h_{2})))\setminus K(h_{1}).$ 
By Proposition~\ref{p:bdyosc}-(b) and \cite[Corollary 3.2]{HM}, 
we have $J(G)\subset \overline{U}.$ 
By (\ref{eq:h1-uint}) and (\ref{eq:h2-uint}), 
we have that 
\begin{equation}
\label{eq:h1cush1k}
h_{1}^{-1}(\overline{U})\subset h_{1}^{-1}(K(h_{2})) \mbox{ and } 
h_{2}^{-1}(\overline{U})\subset 
K(h_{2})\setminus (\mbox{int}(h_{2}^{-1}(K(h_{1})))).
\end{equation}
Since we are assuming  $h_{1}^{-1}(J(h_{2}))\cap h_{2}^{-1}(J(h_{1}))= \emptyset $, 
from (\ref{eq:h1kh2sint}) and (\ref{eq:h1cush1k}) it follows that 
$h_{1}^{-1}(J(G))\cap h_{2}^{-1}(J(G))\subset h_{1}^{-1}(\overline{U})\cap h_{2}^{-1}(\overline{U})=\emptyset .$ 
Since $J(G)=h_{1}^{-1}(J(G))\cup h_{2}^{-1}(J(G))$ (\cite[Lemma 1.1.4]{S1}), 
we obtain that $J(G)$ is disconnected. 
It implies $(h_{1},h_{2})\in {\cal D}\cap {\cal B}\cap {\cal H}$. 
However, since ${\cal D}\cap {\cal B}\cap {\cal H}$ is an open subset of ${\cal P}^{2}$ 
(Lemma~\ref{l:hbdopen}), it contradicts $(h_{1},h_{2})\in (\partial {\cal D})\cap {\cal B}\cap {\cal H}.$ 
Therefore, we must have that 
$h_{1}^{-1}(J(h_{2}))\cap h_{2}^{-1}(J(h_{1}))\neq  \emptyset .$ 

Thus we have proved our lemma. 
\end{proof} 
\begin{lem}
\label{l:h1h2hbz0}
Let $(h_{1},h_{2})\in {\cal B}\cap {\cal H}.$ 
Suppose $h_{1}^{-1}(J(h_{2}))\cap h_{2}^{-1}(J(h_{1}))\neq \emptyset .$ 
Let $z_{0}\in h_{1}^{-1}(J(h_{2}))\cap h_{2}^{-1}(J(h_{1})).$ 
For each $b\in \CC $, let 
$S_{b}(z)=z+b$ and let $h_{3,b}=S_{b}\circ h_{2}\circ S_{b}^{-1}.$ 
Then, for each $\epsilon>0$, 
there exists a number $c\in \{ b\in \CC \mid |b|<\epsilon \} $ such that 
$h_{1}^{-1}(J(h_{3,c}))\setminus h_{3,c}^{-1}(K(h_{1}))\neq \emptyset $ 
and $h_{1}^{-1}(K(h_{3,c}))\neq h_{3,c}^{-1}(K(h_{1})).$ 
\end{lem}
\begin{proof}
Let $w_{0}:=h_{1}(z_{0})\in J(h_{2}).$ Then, 
\begin{equation}
\label{eq:t0w0tb}
S_{0}(w_{0})=w_{0}, \ S_{b}(w_{0})\in J(h_{3,b}) \mbox{ for each }b\in \CC , \mbox{ and } 
b\mapsto S_{b}(w_{0}) \mbox{ is holomorphic on } \CC . 
\end{equation}
Let $V$ be an unbounded  simply connected subdomain of $\CC $ with $0\in V$  
such that the set $V_{0}:=\{ S_{b}(w_{0})\mid b\in V\} $ does not contain any critical value of $h_{1}.$ 
Let $\zeta :V_{0}\rightarrow \CC $ be a well-defined inverse branch of $h_{1}$ such that 
$\zeta (w_{0})=z_{0}.$ Let $\alpha :V\rightarrow \CC $ be the map defined by 
$\alpha (b)=(h_{3,b}\circ \zeta \circ S_{b})(w_{0}), b\in V.$    
Let $l:= \deg (h_{1}), n:= \deg (h_{2})$ and 
let $a_{1}, a_{2}\in \CC \setminus \{ 0\} $ be numbers such that 
$h_{1}(z)=a_{1}z^{l}+\cdots $ and $h_{2}(z)=a_{2}z^{n}+\cdots .$ 
Then, 
$$| \zeta (S_{b}(w_{0}))|=|\zeta (w_{0}+b)|\sim 
\left|\frac{w_{0}+b}{a_{1}}\right|^{\frac{1}{l}}=\left(\frac{1}{|a_{1}|}\right)^{\frac{1}{l}}\cdot |b|^{\frac{1}{l}}(1+o(1)), 
\mbox{ as } b\rightarrow \infty \mbox{ in }  V.$$
Hence $|\zeta (S_{b}(w_{0}))-b|=|b|(1+o(1)|$ as $ b\rightarrow \infty $ in  $V.$ 
Therefore $|h_{2}(\zeta (S_{b}(w_{0}))-b)|=|a_{2}||b|^{n}(1+o(1))$ as $b \rightarrow \infty $ in $V.$ 
Thus $$|\alpha (b)|=|b+h_{2}(\zeta (S_{b}(w_{0}))-b)|=|a_{2}||b|^{n}(1+o(1)) 
\mbox{ as } b\rightarrow \infty  \mbox{ in }V.$$ 
Hence $\alpha :V\rightarrow \CC $ is not constant on $V.$ Since $\alpha (0)=h_{2}(z_{0})\in J(h_{1})$,  
it follows that for each $\epsilon >0$ there exists a number $c\in \{ b\in \CC \mid |b|<\epsilon \} $ 
such that $\alpha (c)\in \CC \setminus K(h_{1}).$  Combining this with $S_{c}(w_{0})\in J(h_{3,c})$, 
we obtain that $(h_{3,c}^{-1}(\CC \setminus K(h_{1})))\cap (h_{1}^{-1}(J(h_{3,c})))\neq \emptyset .$ 
Therefore $(h_{1}^{-1}(J(h_{3,c})))\setminus (h_{3,c}^{-1}(K(h_{1})))\neq \emptyset .$ 

Thus we have proved our lemma. 
\end{proof}

\begin{lem}
\label{l:pfmainth2}
Statement~\ref{spacemainth2} in Theorem~\ref{t:spacetopology} holds. 
\end{lem}
\begin{proof}
It suffices to prove that $(\partial {\cal C})\cap {\cal B}\cap {\cal H}\subset 
\overline{\mbox{int}({\cal C})}\cap {\cal B}\cap {\cal H}.$ 
In order to prove it, let 
$(h_{1},h_{2})\in (\partial {\cal C})\cap {\cal B}\cap {\cal H}.$ 
Then $(h_{1},h_{2})\in (\partial {\cal D})\cap {\cal B}\cap {\cal H}.$ 
By Proposition~\ref{p:bdyosc}, we may assume $K(h_{1})\subset K(h_{2}).$ 
Let $W$ be a small neighborhood of $h_{2}$ in ${\cal P}$ with $\{ h_{1}\} \times W\subset {\cal B}\cap {\cal H}.$  
By Lemmas~\ref{l:h1h2bdhb} and \ref{l:h1h2hbz0}, 
there exists an element $h_{3}\in W$ such that 
$(h_{1}^{-1}(J(h_{3})))\setminus (h_{3}^{-1}(K(h_{1})))\neq \emptyset .$ 
Then there exists an open neighborhood $B$ of $(h_{1},h_{3})$ in ${\cal B}\cap {\cal H}$ such that 
\begin{equation}
\label{eq:g1-1jg2}
(g_{1}^{-1}(J(g_{2})))\setminus (g_{2}^{-1}(K(g_{1})))\neq \emptyset \mbox{ for each }(g_{1},g_{2})\in B. 
\end{equation}
We consider the following two cases. Case 1. $K(h_{1})\subsetneqq K(h_{2}).$ 
Case 2. $K(h_{1})=K(h_{2}).$

We now suppose that we have case 1. Then, letting $W$ and $B$ so small, 
we obtain that 
\begin{equation}
\label{eq:kg2kg1}
K(g_{2})\setminus K(g_{1})\neq \emptyset \mbox{ for each }(g_{1},g_{2})\in B.
\end{equation}
Then, for each $(g_{1},g_{2})\in B$, $J(g_{1},g_{2})$ is connected. 
Indeed, suppose that there exists an element $(\beta _{1},\beta _{2})\in B$ such that 
$J(\beta _{1},\beta _{2})$ is disconnected. Then $(\beta _{1}, \beta _{2})\in {\cal D}\cap {\cal B}.$ 
By Proposition~\ref{p:dosc} and (\ref{eq:kg2kg1}), 
we obtain $K(\beta _{1})\subset \mbox{int}(K(\beta _{2})).$ 
By Proposition~\ref{p:dosc} again, we get that 
$\beta _{1}^{-1}(K(\beta _{2}))\subset \beta _{2}^{-1}(K(\beta _{1})).$ However, this contradicts   
(\ref{eq:g1-1jg2}). Therefore, for each $(g_{1},g_{2})\in B$, $J(g_{1},g_{2})$ is connected. 
Thus, $B\subset \mbox{int}({\cal C})\cap {\cal B}\cap {\cal H}.$ Since $W$ was an arbitrary small neighborhood of $h_{2}$, 
it follows that $(h_{1},h_{2})\in \overline{\mbox{int}({\cal C})}\cap {\cal B}\cap {\cal H}.$ 

 We now suppose that we have case 2. 
For each $b\in \CC $, let $S_{b}(z)=z+b$ and $h_{3,b}(z)=S_{b}\circ h_{2}\circ S_{b}^{-1}.$ 
Let $A:= K(h_{1})=K(h_{2}).$ 
Let  $x_{1},x_{2}\in A$ be two points such that 
$|x_{1}-x_{2}|=\mbox{diam}_{E}(A)$, where $\mbox{diam}_{E}(A_{0}):=
\sup _{\alpha, \beta\in A_{0}}|\alpha-\beta|$ 
for a subset $A_{0}$ of $\CC .$  
Then $x_{1},x_{2}\in J(h_{1})=J(h_{2}).$ Moreover, we have 
\begin{equation}
\label{eq:tbx1tbx2}
|S_{b}(x_{1})-S_{b}(x_{2})|=\mbox{diam}_{E}(K(h_{3,b}))=\mbox{diam}_{E}(K(h_{1})).
\end{equation}
Let $\epsilon >0$ be any small number. 
Since $\mbox{int}(K(h_{1}))=\mbox{int}(K(h_{2}))\neq \emptyset $ (see Proposition~\ref{p:bdyosc}-(c)), 
there exists an element $c\in \{ b\in \CC \mid |b|<\epsilon \} $ such that 
\begin{equation}
\label{eq:tbx1int}
S_{c}(x_{1})\in \mbox{int}(K(h_{1})).
\end{equation}
By (\ref{eq:tbx1tbx2}) and (\ref{eq:tbx1int}), 
we obtain 
\begin{equation}
\label{eq:tbx2ck}
S_{c}(x_{2})\in \CC \setminus K(h_{1}).
\end{equation}
By (\ref{eq:tbx1int}) and (\ref{eq:tbx2ck}), 
it follows that 
$$ J(h_{3,c})\cap \mbox{int}(K(h_{1}))\neq \emptyset \mbox{ and } 
J(h_{3,c})\cap (\CC \setminus K(h_{1}))\neq \emptyset . 
 $$
 Therefore there exists a neighborhood $B$ of $(h_{1},h_{3,c})$ in ${\cal B}\cap {\cal H}$ 
 such that for each $(g_{1},g_{2})\in B$, 
 $$ J(g_{2})\cap \mbox{int}(K(g_{1}))\neq \emptyset \mbox{ and } 
J(g_{2})\cap (\CC \setminus K(g_{1}))\neq \emptyset . 
 $$
 Since $J(g_{1})$ is connected, we obtain that for each 
$(g_{1},g_{2})\in B$, $J(g_{1})\cap J(g_{2})\neq \emptyset .$ Thus 
$J(g_{1})\cup J(g_{2})$ is included in a connected component of $J(g_{1},g_{2}).$ 
By \cite[Theorem 2.1]{SdpbpI}, 
it follows that for each $(g_{1},g_{2})\in B$, $J(g_{1},g_{2})$ is connected. 
Hence $B\subset \mbox{int}({\cal C}).$ 
From these arguments, we obtain that $(h_{1},h_{2})\in \overline{\mbox{int}({\cal C})}
\cap {\cal B}\cap {\cal H}.$ 

 Therefore $(\partial {\cal C})\cap {\cal B}\cap {\cal H}\subset 
\overline{\mbox{int}({\cal C})}\cap {\cal B}\cap {\cal H}.$ 
Thus statement \ref{spacemainth2} in Theorem~\ref{t:spacetopology} holds. 
\end{proof}
\begin{lem}
\label{l:d1d2not22}
There exists a neighborhood $U$ of $\overline{{\cal B}\cap {\cal D}}$ in ${\cal P}^{2}$ 
such that for each $(g_{1},g_{2})\in U$, $(\deg (g_{1}),\deg (g_{2}))\neq (2,2).$ 
Moreover, for each connected component $A$ of ${\cal P}^{2}$ with 
$A\cap \overline{{\cal B}\cap {\cal D}}\neq \emptyset $, 
$(\deg (g_{1}),\deg (g_{2}))\neq (2,2))$ for each $(g_{1},g_{2})\in A.$ 
\end{lem}
\begin{proof}
By \cite[Theorem 2.15]{SdpbpI}, for each $(g_{1},g_{2})\in {\cal B}\cap {\cal D}$ 
we have $(\deg (g_{1}),\deg (g_{2}))\neq (2,2).$ 
Since the function $\deg : {\cal P}\rightarrow \NN  \subset \RR $ 
is continuous, the statement of our lemma holds.  
\end{proof} 

\begin{df}
Let $D$ be a domain in $\CCI $ with $\infty \in D.$ We denote by 
$\varphi (D,z)$ Green's function on $D$ with pole at $\infty .$ 
\end{df}

\begin{lem}
\label{l:dh1-1log}
Let $(h_{1},h_{2})\in (\overline{{\cal D}}\cap {\cal B}\cap {\cal H})\setminus {\cal Q}.$ 
For each $i=1,2,$ let $d_{i}:=\deg (h_{i})$ and we denote by $c_{i}$ the leading coefficient of 
$h_{i}.$ Then 
$\frac{1}{d_{1}-1}\log |c_{1}|\neq \frac{1}{d_{2}-1}\log |c_{2}|.$ Moreover, 
$h_{1}^{-1}(K(h_{2}))\neq h_{2}^{-1}(K(h_{1})).$ 
\end{lem}
\begin{proof}
Let $(h_{1},h_{2})\in (\overline{{\cal D}}\cap {\cal B}\cap {\cal H})\setminus {\cal Q}.$ 
By Proposition~\ref{p:bdyosc}, we may assume that $K(h_{1})\subset K(h_{2}).$ 
Then by Proposition~\ref{p:bdyosc} again, we have 
$K(h_{1})\subset h_{1}^{-1}(K(h_{2}))\subset h_{2}^{-1}(K(h_{1}))\subset K(h_{2}).$ 
For each $i=1,2,$ let $A_{i}:=\CCI \setminus K(h_{i}).$ 
Since $K(h_{1})\subset K(h_{2})$, we have that 
\begin{equation}
\label{eq:grh2}
\varphi (A_{2},z)-\varphi (A_{1},z)\leq 0 \mbox{ on } \CC \setminus K(h_{2}).
\end{equation}
Therefore by letting $z\rightarrow \infty $ in (\ref{eq:grh2}), we obtain that 
\begin{equation}
\label{eq:dh2-1logc2}
\frac{1}{d_{2}-1}\log |c_{2}|-\frac{1}{d_{1}-1}\log |c_{1}|\leq 0. 
\end{equation}
Since the function $\varphi (A_{2},z)-\varphi (A_{1},z)$ is harmonic and 
bounded in $\CC \setminus K(h_{2})$, the maximum principle implies that 
the equality holds in (\ref{eq:dh2-1logc2}) if and only if 
$\varphi (A_{2},z)-\varphi (A_{1},z)\equiv 0 \mbox{ on } \CC \setminus K(h_{2})$, which is equivalent to  
$J(h_{1})=J(h_{2}).$ Since $K(h_{1})\subset K(h_{2})$, Proposition~\ref{p:bdyosc}-(c) implies that 
$J(h_{2})$ is a quasicircle. Therefore, $J(h_{1})=J(h_{2})$ is equivalent to $(h_{1},h_{2})\in {\cal Q}.$ 
Since we are assuming $(h_{1},h_{2})\notin {\cal Q}$, 
it follows that 
\begin{equation}
\label{eq:1d2-1}
\frac{1}{d_{2}-1}\log |c_{2}|-\frac{1}{d_{1}-1}\log |c_{1}|\neq 0.
\end{equation}
It is easy to see that $\varphi (\CCI \setminus h_{1}^{-1}(K(h_{2})), z)=\frac{1}{d_{1}}\varphi (A_{2},h_{1}(z)) $ 
for each $z\in \CCI \setminus h_{1}^{-1}(K(h_{2}))$ and 
$\varphi (\CCI \setminus h_{2}^{-1}(K(h_{1})),z)=\frac{1}{d_{2}}\varphi (A_{1},h_{2}(z))$ 
for each $z\in \CCI \setminus h_{2}^{-1}(K(h_{1})).$ 
Since $h_{1}^{-1}(K(h_{2}))\subset h_{2}^{-1}(K(h_{1}))$, we obtain that 
$\varphi (\CCI \setminus h_{2}^{-1}(K(h_{1})),z)-\varphi (\CCI \setminus h_{1}^{-1}(K(h_{2})), z)\leq 0$ for each 
$z\in \CC \setminus h_{2}^{-1}(K(h_{1})).$ 
Therefore 
\begin{equation}
\label{eq:dh1gr}
\frac{1}{d_{2}}\varphi (A_{1},h_{2}(z))- 
\frac{1}{d_{1}}\varphi (A_{2},h_{1}(z))\leq 0 
\mbox{ for each }z\in \CC \setminus h_{2}^{-1}(K(h_{1})).
\end{equation}
Since $\varphi (A_{i},z)=\log |z|+\frac{1}{d_{i}-1}\log |c_{i}|+O(\frac{1}{|z|})$ as $z \rightarrow \infty $ 
for each $i=1,2$, by letting $z \rightarrow \infty $ in (\ref{eq:dh1gr})   
we obtain that 
\begin{equation}
\label{eq:1d1log}
\frac{1}{d_{2}}(\log |c_{2}|+\frac{1}{d_{1}-1}\log |c_{1}|)
-\frac{1}{d_{1}}(\log |c_{1}|+\frac{1}{d_{2}-1}\log |c_{2}|)
\leq 0.
\end{equation}
The function $\frac{1}{d_{2}}\varphi (A_{1}, h_{2}(z))-\frac{1}{d_{1}}\varphi (A_{2},h_{1}(z))$ 
is harmonic and bounded in $\CC \setminus h_{2}^{-1}(K(h_{1}))$.  
Therefore, the maximum principle implies that 
 the equality holds in (\ref{eq:1d1log}) if and only if 
$$\frac{1}{d_{2}}\varphi (A_{1}, h_{2}(z))-\frac{1}{d_{1}}\varphi (A_{2},h_{1}(z))\equiv 0  
\mbox{ on }\CC \setminus h_{2}^{-1}(K(h_{1})),$$ 
which is equivalent to that $h_{1}^{-1}(K(h_{2}))=h_{2}^{-1}(K(h_{1})).$ 
Moreover,  (\ref{eq:1d1log}) is equivalent to 
\begin{equation}
\label{eq:d1d1-1}
(d_{2}(d_{1}-1)(d_{2}-1)-d_{1}(d_{2}-1))\log |c_{1}|\geq (d_{1}(d_{2}-1)(d_{1}-1)-d_{2}(d_{1}-1))\log |c_{2}|.
\end{equation}
It is easy to see that (\ref{eq:d1d1-1}) is equivalent to 
\begin{equation}
\label{eq:d1d2-d1-d2}
(d_{1}d_{2}-d_{1}-d_{2})((d_{2}-1)\log |c_{1}|-(d_{1}-1)\log |c_{2}|)\geq 0. 
\end{equation}
Since $(d_{1},d_{2})\neq (2,2)$ (Lemma~\ref{l:d1d2not22}), 
(\ref{eq:d1d2-d1-d2}) is equivalent to 
\begin{equation}
\label{eq:d2-1log}
(d_{2}-1)\log |c_{1}|\geq (d_{1}-1)\log |c_{2}|.
\end{equation}
Therefore, (\ref{eq:1d1log}) is equivalent to (\ref{eq:d2-1log}), and 
the equality holds in (\ref{eq:1d1log}) if and only if the equality holds in (\ref{eq:d2-1log}). 
By (\ref{eq:1d2-1}), it follows that $h_{1}^{-1}(K(h_{2}))\neq h_{2}^{-1}(K(h_{1})).$ 

Thus we have proved our lemma.  
\end{proof} 
\begin{lem}
\label{l:pfmainth3}
Statement~\ref{spacemainth3} in Theorem~\ref{spacemainth} holds. 
\end{lem}
\begin{proof}
It suffices to prove that if $(h_{1},h_{2})\in (\overline{{\cal D}}\cap {\cal B}\cap {\cal H})\setminus {\cal Q}$, 
then $J(h_{1},h_{2})$ is porous and $\dim _{H}(J(h_{1},h_{2}))\leq \dim _{B}(J(h_{1},h_{2}))<2.$ 
In order to prove it, let $(h_{1},h_{2})\in (\overline{{\cal D}}\cap {\cal B}\cap {\cal H})\setminus {\cal Q}$. 
By Proposition~\ref{p:bdyosc}, we may assume that 
$K(h_{1})\subset K(h_{2}).$ By Proposition~\ref{p:bdyosc} again, 
we obtain that $K(h_{1})\subset h_{1}^{-1}(K(h_{2}))\subset h_{2}^{-1}(K(h_{1}))\subset K(h_{2})$ and   
setting $U:=(\mbox{int}(K(h_{2})))\setminus K(h_{1})$, 
$(h_{1},h_{2})$ satisfies the open set condition with $U.$ 
Moreover, 
$h_{1}^{-1}(\overline{U})\cup h_{2}^{-1}(\overline{U})\subset \overline{U}.$ 
We now show the following claim. \\ 
Claim. $h_{1}^{-1}(\overline{U})\cup h_{2}^{-1}(\overline{U})\neq \overline{U}.$ 

To prove this claim, suppose that $h_{1}^{-1}(\overline{U})\cup h_{2}^{-1}(\overline{U})= \overline{U}.$ 
 Since $\overline{U}\subset K(h_{2})\setminus \mbox{int}(K(h_{1}))$, 
 we obtain that 
\begin{equation}
\label{eq:h1-1overu}
\overline{U}=h_{1}^{-1}(\overline{U})\cup h_{2}^{-1}(\overline{U})\subset 
 \left(h_{1}^{-1}(K(h_{2}))\setminus \mbox{int}(K(h_{1}))\right) 
\cup \left(K(h_{2})\setminus \mbox{int}(h_{2}^{-1}(K(h_{1})))\right).
\end{equation}   
Moreover, by Lemma~\ref{l:dh1-1log}, 
$h_{1}^{-1}(K(h_{2}))\subsetneqq h_{2}^{-1}(K(h_{1})).$ 
Since $\overline{\mbox{int}(h_{2}^{-1}(K(h_{1})))}=h_{2}^{-1}(K(h_{1}))$, 
we obtain that 
$(\mbox{int}(h_{2}^{-1}(K(h_{1}))))\setminus h_{1}^{-1}(K(h_{2}))\neq \emptyset .$ 
Therefore 
$$\emptyset \neq (\mbox{int}(h_{2}^{-1}(K(h_{1}))))\setminus h_{1}^{-1}(K(h_{2}))
\subset (\mbox{int}(K(h_{2})))\setminus K(h_{1})=U.$$  
However, this contradicts (\ref{eq:h1-1overu}). Thus we have proved our claim. 

 Let $G=\langle h_{1},h_{2}\rangle .$ 
 By \cite[Corollary 3.2]{HM}, $J(G)\subset \overline{U}.$ 
 Moreover, by \cite[Lemma 1.1.4]{S1}), $J(G)=h_{1}^{-1}(J(G))\cup h_{2}^{-1}(J(G)).$  
Therefore, Claim above implies that 
$J(G) =h_{1}^{-1}(J(G))\cup h_{2}^{-1}(J(G))\subset h_{1}^{-1}(\overline{U})\subset h_{2}^{-1}(\overline{U})
\subsetneqq \overline{U}.$ Hence, $J(G)\neq \overline{U}.$ 
Combining this with \cite[Theorem 1.25]{S7}, we see that 
$J(G)$ is porous and $\dim _{H}(J(G))\leq \overline{\dim }_{B}(J(G))<2.$ 
Thus statement~\ref{spacemainth3} in Theorem~\ref{spacemainth} holds. 
\end{proof}
\begin{lem}
\label{l:pfmainth4}
Statement~\ref{spacemainth4} in Theorem~\ref{spacemainth} holds. 
\end{lem}
\begin{proof}
Let $(h_{1},h_{2})\in (\overline{{\cal D}}\cap {\cal B}\cap {\cal H})\setminus {\cal Q}.$ 
By Lemma~\ref{l:pfmainthosc}, $(h_{1},h_{2})$ satisfies the open set condition. 
Thus, by \cite[Theorem 1.1, 1.2]{S6}, for 
each $z\in \CCI \setminus P(\langle h_{1},h_{2}\rangle )$, 
$\dim _{H}(J(h_{1},h_{2}))=\overline{\dim}_{B}(J(h_{1},h_{2}))=\delta _{(h_{1},h_{2})}=Z_{(h_{1},h_{2})}(z).$ 
Moreover, by \cite[Theorem 3.15]{SUb}, we have that 
$\frac{\log (d_{1}+d_{2})}{\sum _{j=1}^{2}\frac{d_{j}}{d_{1}+d_{2}}\log d_{j}}\leq \dim _{H}(J(h_{1},h_{2})).$ 

Suppose that  $\frac{\log (d_{1}+d_{2})}{\sum _{j=1}^{2}\frac{d_{j}}{d_{1}+d_{2}}\log d_{j}}= \dim _{H}(J(h_{1},h_{2}))$.  
Then \cite[Theorem 3.15]{SUb} implies that 
there exist a transformation $\varphi (z)=\alpha z+\beta $ with $\alpha ,\beta \in \CC, \alpha \neq 0$,  
two complex numbers $a_{1}, a_{2}$ and a positive integer $d$  
such that  
we have $d=d_{1}=d_{2}$ and for each $j=1,2, \varphi \circ h_{j}\circ \varphi ^{-1}(z)=a_{j}z^{d}.$  
By Lemma~\ref{l:d1d2not22}, we obtain that $d\geq 3.$ From these it is easy to see that $(h_{1},h_{2})\in {\cal D}
\cap {\cal B}\cap {\cal H}.$   
\end{proof}

\begin{lem}
Statement~\ref{spacemainth4-2} in Theorem~\ref{spacemainth} holds. 
\end{lem}
\begin{proof}
There exists a neighborhood $W$ of $(h_{1},h_{2})$ in ${\cal B}\cap {\cal H}.$ 
By \cite[Theorem 1.1]{S6}, for each 
$(g_{1},g_{2})\in W$ and for each $z\in \CCI \setminus P(\langle g_{1},g_{2}\rangle )$, we have  
$\dim _{H}(J(g_{1},g_{2}))\leq \overline{\dim }_{B}(J(g_{1},g_{2}))\leq \delta _{(g_{1},g_{2})}=Z_{(g_{1},g_{2})}(z).$ 
Furthermore, by \cite{SU1}, the map  $(g_{1},g_{2})\mapsto \delta _{(g_{1},g_{2})}$ is continuous in $W.$ 
Combining these arguments with Lemmas~\ref{l:pfmainth3}, \ref{l:pfmainth4}, we see that 
Statement~\ref{spacemainth4-2} in Theorem~\ref{spacemainth} holds. 
\end{proof}

\begin{lem}
\label{l:l1subvar}
Let $L_{1}:= \{ (h_{1},h_{2})\in {\cal P}^{2}\mid J(h_{1})=J(h_{2})\} .$ 
Then, for each connected component ${\cal V}$ of ${\cal P}^{2}$, $L_{1}\cap {\cal V}$ is included in a proper holomorphic subvariety of ${\cal V}.$ 

\end{lem}
\begin{proof}
Let $(h_{1},h_{2})\in {\cal P} ^{2}.$ 
Let $n:=\deg (h_{1}), m:= \deg (h_{2}).$ 
We write $h_{1}(z)=a_{n}(h_{1})z^{n}+a_{n-1}(h_{1})z^{n-1}+\cdots +a_{0}(h_{1})$ 
and $h_{2}(z)=b_{m}(h_{2})z^{m}+b_{m-1}(h_{2})z^{m-1}+\cdots +b_{0}(h_{2}).$ 
Let $\zeta (h_{1}):= -\frac{a_{n-1}(h_{1})}{n\cdot a_{n}(h_{1})}.$
Let $$\Sigma (h_{1}):= \{ \alpha (z)= a(z-\zeta (h_{1}))+\zeta (h_{1})\mid a\in \CC ,|a|=1, \alpha (J(h_{1}))=J(h_{1})\} .$$  
By \cite[Theorems 1,5]{Be90}, $J(h_{1})=J(h_{2})$ if and only if 
there exists an element $\eta \in \Sigma (h_{1})$ such that 
$h_{1}\circ h_{2}=\eta \circ h_{2}\circ h_{1}.$ Moreover, 
either (1) $\sharp \Sigma (h_{1})\leq n$ or (2) $\sharp \Sigma (h_{1})=\infty .$ 
Moreover, (1) holds if and only if 
setting $\rho (z)=(\frac{1}{a_{n}(h_{1})})^{\frac{1}{n-1}}\cdot z$, 
we have $\rho ^{-1}\circ T_{\zeta (h_{1})}^{-1}\circ h_{1}\circ T_{\zeta (h_{1})}\circ \rho (z)=
z^{c}h_{1,0}(z^{d})$, where $c,d\geq 0$ are maximal for this form, 
and $h_{1,0}$ is a non-constant monic polynomial.  
Note that in this case, 
$\Sigma (h_{1})=\{ \alpha (z)=a(z-\zeta (h_{1}))+\zeta (h_{1})\mid \alpha (J(h_{1}))=J(h_{1}), \alpha ^{d}=1\} .$ 
 Furthermore, ``$J(h_{1})=J(h_{2})$ and (2)'' holds if and only if 
$h_{1}(z)=a_{n}(h_{1})(z-\zeta (h_{1}))^{n}+\zeta (h_{1})$ and 
$h_{2}(z)=b_{m}(h_{2})(z-\zeta (h_{1}))^{m}+\zeta (h_{1})$. 
Thus the statement of our lemma holds.  
\end{proof}
\begin{lem}
\label{l:pfmainth5}
Statement~\ref{spacemainth5} in Theorem~\ref{t:spacetopology} holds.
\end{lem}
\begin{proof}
By the definition of ${\cal Q}$, it is easy to see that ${\cal Q}\cap {\cal D}=\emptyset .$ 
Moreover, by Lemma~\ref{l:bdjconti}, for each $(h_{1},h_{2})\in {\cal D}\cap {\cal B}$, 
$h_{1}^{-1}(J(h_{1},h_{2}))\cap h_{2}^{-1}(J(h_{1},h_{2}))=\emptyset .$ 
Thus ${\cal D}\cap {\cal B}\cap {\cal I}=\emptyset .$  
Let $L_{1}:= \{ (h_{1},h_{2})\in {\cal P}^{2}\mid J(h_{1})=J(h_{2})\} .$ 
Let ${\cal V}$ be any connected component of ${\cal P}^{2}.$ 
Since ${\cal Q}\subset L_{1}$, Lemma~\ref{l:l1subvar} implies that 
${\cal Q}\cap {\cal V}$ is included in a proper subvariety 
of ${\cal V}.$ 
Thus statement~\ref{spacemainth5} in Theorem~\ref{t:spacetopology} holds.
\end{proof}
\begin{lem}
\label{l:pfmainth7}
Statement~\ref{spacemainth7} in Theorem~\ref{t:spacetopology} holds. 
\end{lem}
\begin{proof}
It is easy to see that $((\partial {\cal C})\cap {\cal B}\cap {\cal H})\setminus {\cal Q}$ 
is an open subset of $(\partial {\cal C})\cap {\cal B}\cap {\cal H}$. 
In order to show that $((\partial {\cal C})\cap {\cal B}\cap {\cal H})\setminus {\cal Q}$ is dense in 
$(\partial {\cal C})\cap {\cal B}\cap {\cal H}$, 
let 
$(h_{1},h_{2})\in (\partial {\cal C})\cap {\cal B}\cap {\cal H}$. 
Let $W$ be any open polydisc neighborhood of $(h_{1},h_{2})$ in ${\cal B}\cap {\cal H}.$  
By Lemma~\ref{l:pfmainth2} and Lemma~\ref{l:pfmainth5}, 
there exists an element 
$(g_{1},g_{2})\in ((\mbox{int}({\cal C}))\cap W)\setminus {\cal Q}.$ 
By Lemma ~\ref{l:pfmainth5}, $W\setminus {\cal Q}$ is connected. 
Therefore there exists a curve $\gamma $ in $W\setminus {\cal Q}$ which joins $(g_{1},g_{2})$ and 
a point in ${\cal D}.$ Then $\gamma \cap \partial {\cal D}\neq \emptyset .$ 
Therefore $\gamma \cap (((\partial {\cal D})\cap {\cal B}\cap {\cal H})\setminus {\cal Q})\neq \emptyset .$ 
Thus $((\partial {\cal C})\cap {\cal B}\cap {\cal H})\setminus {\cal Q}$ is dense in 
$(\partial {\cal C})\cap {\cal B}\cap {\cal H}$. 
Hence statement~\ref{spacemainth7} in Theorem~\ref{t:spacetopology} holds. 
\end{proof}
\begin{lem}
\label{l:pfexample}
Statement~\ref{spaceexample} in Theorem~\ref{t:spacetopology} holds. 
\end{lem}
\begin{proof}
Let $h_{1}\in {\cal P}$ and suppose $\langle h_{1}\rangle \in {\cal G}$ and $h_{1}$ is hyperbolic. 
Then $\mbox{int}(K(h_{1}))\neq \emptyset .$ 
Let $d_{1}:=\deg (h_{1})$ and let $d\in \NN .$ Suppose $(d_{1},d)\neq (2,2).$ 
Let $b\in \mbox{int}(K(h_{1}))$ be a point. 
Here, if $h_{1}(z)$ is of the form $c_{1}(z-c_{2})^{d_{1}}+c_{2}$, then 
we need the additional condition that $b\neq c_{2}.$   
Let $z_{0}\in J(h_{1})$ be a point such that 
$|z_{0}-b|=\sup _{z_{1}\in K(h_{1})}|z_{1}-b|.$ 
Let $s:= |z_{0}-b|.$ 
We show the following claim. \\ 
{\bf Claim 1}. $\{ z\in \CC \mid |z-b|=s\} \setminus J(h_{1})\neq \emptyset .$ 

In order to prove Claim 1, let $C:= \{ z\in \CC \mid |z-b|=s\} .$ 
By the way of the choice of $b$, we have $C\neq J(h_{1}).$ 
(Indeed, if $J(h_{1})=C$, then  $h_{1}(z)=c_{1}(z-b)^{d_{1}}+b$ for some $c_{1}\in \CC$.) 
Suppose $C\subset J(h_{1}).$ Then $C\subsetneqq J(h_{1}).$ 
By the definition of $s$, we have $J(h_{1})\subset \{ z\in \CC \mid |z-b|\leq s\} .$ 
Therefore $J(h_{1})\setminus C\subset \{ z\in \CC \mid |z-b|<s\} .$ 
Let $w_{1}\in J(h_{1})\setminus C$ be a point. 
Let $W$ be any neighborhood of $w_{1}$ in $\CC .$ Then 
there exists a point $w_{2}\in W\cap (\CCI \setminus K(h_{1}))\cap \{ z\in \CC \mid |z-b|<s\} .$ 
Since $\CCI \setminus K(h_{1})$ is connected, there exist a curve 
$\gamma $ in $\CCI \setminus K(h_{1})$ which joins $w_{2}$ and $\infty .$ 
Then $\emptyset \neq \gamma \cap C\subset \CCI \setminus (K(h_{1}))$, which contradicts 
the assumption that $C\subset J(h_{1}).$ Thus, 
we must have that $C\not\subset J(h_{1}).$ 
Hence claim 1 holds. 

Let $z_{1}\in \{ z\in \CC \mid |z-b|=s\} \setminus J(h_{1}).$ 
Let $\theta \in \RR $ be a number such that 
$e^{i\theta }\frac{1}{s^{d-1}}(z_{0}-b)^{d}+b=z_{1}.$ 
Let $r>0$ be a number such that 
$D(b,r)\subset \mbox{int}(K(h_{1})).$ 
Let $c_{1}\in \CC \setminus \{ 0\} $ be the leading coefficient of $h_{1}.$ 
Then $h_{1}(z)=c_{1}z^{d_{1}}(1+O(\frac{1}{z}))$ as $z\rightarrow \infty .$ 
Let $R\in \RR $ be a number such that 
\begin{equation}
\label{eq:rexp1}
R>\exp\left(\frac{1}{dd_{1}-d-d_{1}}(-d_{1}\log r+d_{1}d\log 2-d\log |c_{1}|)\right).
\end{equation}
Taking $R$ large enough, we may assume that $R$ satisfies the following 
\begin{equation}
\label{eq:dbd1rc1}
D\left(b, \sqrt[d_{1}]{\frac{R}{|c_{1}|}}\cdot \frac{3}{4}\right)\subset h_{1}^{-1}(D(b,R))
\subset D\left(b, \sqrt[d_{1}]{\frac{R}{|c_{1}|}}\cdot \frac{3}{2}\right)\subset \subset D(b,R), 
\end{equation}
where $A\subset \subset B$ means that $\overline{A}$ is a compact subset of int$B.$  
For each $t>0$, let $g_{t}(z)=te^{i\theta }(z-b)^{d}+b.$ 
Then for each $t>0$, we have 
$g_{t}^{-1}(D(b,r))=D\left(b,\sqrt[d]{\frac{r}{t}}\right).$
Let $t_{0}=\frac{1}{R^{d-1}}.$ Taking $R$ so large, we may assume that 
\begin{equation}
\label{eq:t01sd-1}
t_{0}<\frac{1}{s^{d-1}}.
\end{equation}
Since $(d_{1},d)\neq (2,2)$, it is easy to see that
(\ref{eq:rexp1}) is equivalent to   
\begin{equation}
\label{eq:sqrtdfr}
\sqrt[d]{\frac{r}{t_{0}}}>2\sqrt[d_{1}]{\frac{R}{|c_{1}|}}.
\end{equation}
By the definition of $t_{0}$, we have 
\begin{equation}
\label{eq:jgt0}
J(g_{t_{0}})=\left\{ z\in \CC \mid |z-b|=\sqrt[d-1]{\frac{1}{t_{0}}}\right\} =\{ z\in \CC \mid |z-b|=R\} .
\end{equation}
Moreover, taking $R$ so large, we may assume 
\begin{equation}
\label{eq:db12}
D\left( b,\sqrt[d_{1}]{\frac{R}{|c_{1}|}}\cdot \frac{1}{2}\right)\supset K(h_{1}), \ 
D(0,R)\supset K(h_{1}). 
\end{equation}
By (\ref{eq:db12}), (\ref{eq:dbd1rc1}), (\ref{eq:sqrtdfr}), (\ref{eq:jgt0}), 
we obtain 
\begin{align}
\label{eq:kh1sub}
K(h_{1})\subset &  D\left( b,\sqrt[d_{1}]{\frac{R}{|c_{1}|}}\frac{1}{2}\right) \subset \subset 
          D\left( b,\sqrt[d_{1}]{\frac{R}{|c_{1}|}}\frac{3}{4}\right) \subset h_{1}^{-1}(K(g_{t_{0}}))
         \subset D\left( b,\sqrt[d_{1}]{\frac{R}{|c_{1}|}}\cdot \frac{3}{2}\right) \\ 
\label{eq:subsubdb}    \ \ \ \  \subset \subset & D\left( b, \sqrt[d]{\frac{r}{t_{0}}}\right) =g_{t_{0}}^{-1}(D(b,r))
                     \subset g_{t_{0}}^{-1}(K(h_{1}))\subset \subset K(g_{t_{0}}).
\end{align}
Let 
\begin{equation}
\label{eq:t1df}
t_{1}:=\sup \{ t\in [t_{0},\frac{1}{s^{d-1}}] \mid \forall u\in [t_{0},t), 
         K(h_{1})\subset \subset h_{1}^{-1}(K(g_{u}))\subset \subset g_{u}^{-1}(K(h_{1}))
           \subset \subset K(g_{u})\} . 
\end{equation}
Note that by (\ref{eq:kh1sub}), (\ref{eq:subsubdb}) and (\ref{eq:t01sd-1}), 
$t_{1}$ is well-defined. It is easy to see that 
\begin{equation}
\label{eq:kh1h1ik}
K(h_{1})\subset h_{1}^{-1}(K(g_{t_{1}}))\subset g_{t_{1}}^{-1}(K(h_{1}))\subset K(g_{t_{1}}).
\end{equation}
Therefore, for each $t\in [t_{0},t_{1}]$, we have 
$g_{t}(\mbox{CV}^{\ast }(h_{1})\cup \mbox{CV}^{\ast }(g_{t}))\subset g_{t}(K(h_{1}))\subset K(h_{1})$ 
and $h_{1}(\mbox{CV}^{\ast }(h_{1})\cup \mbox{CV}^{\ast }(g_{t}))\subset h_{1}(K(h_{1}))=K(h_{1}).$ 
In particular, 
\begin{equation}
\label{eq:h1gtinb}
(h_{1},g_{t})\in {\cal B} \mbox{ for each } t\in [t_{0},t_{1}].
\end{equation}   
Moreover, for each $t\in [t_{0},t_{1})$, 
$$\left( h_{1}^{-1}(K(g_{t})\setminus \mbox{int}(K(h_{1})))\right) \amalg 
\left( g_{t}^{-1}(K(g_{t})\setminus \mbox{int}(K(h_{1})))\right) \subset 
K(g_{t})\setminus \mbox{int}(K(h_{1})).$$ 
Therefore $J(h_{1},g_{t})\subset K(g_{t})\setminus \mbox{int}(K(h_{1}))$ and 
by \cite[Lemma 1.1.4]{S1} 
$$J(h_{1},g_{t})=h_{1}^{-1}(J(h_{1},g_{t}))\cup g_{t}^{-1}(J(h_{1},g_{t}))
=h_{1}^{-1}(J(h_{1},g_{t}))\amalg g_{t}^{-1}(J(h_{1},g_{t})).$$  
Thus 
\begin{equation}
\label{eq:h1gtd}
(h_{1},g_{t})\in {\cal D} \mbox{ for each }t\in [t_{0},t_{1}).
\end{equation}
In particular, 
\begin{equation}
\label{eq:h1gt1pd}
(h_{1},g_{t_{1}})\in \overline{{\cal D}}\cap {\cal B}. 
\end{equation}
We now show the following claim.\\ 
{\bf Claim 2}. $t_{1}< \frac{1}{s^{d-1}}.$ \\ 
To prove this claim, suppose to the contrary that $t_{1}=\frac{1}{s^{d-1}}.$ 
Then $J(g_{t_{1}})=\{ z\in \CC \mid |z-b|=s\} .$ Since $|z_{0}-b|=s$, 
it follows that $z_{0}\in J(g_{t_{1}})\cap J(h_{1}).$ By (\ref{eq:kh1h1ik}), 
we have $g_{t_{1}}(K(h_{1}))\subset K(h_{1}).$ 
Therefore, $g_{t_{1}}(z_{0})\in J(g_{t_{1}})\cap K(h_{1}).$ 
Moreover, by the way of the choice of $\theta $, we have $g_{t_{1}}(z_{0} )=z_{1}\not\in J(h_{1}).$ 
Hence, we obtain $g_{t_{1}}(z_{0})\in J(g_{t_{1}})\cap \mbox{int}(K(h_{1})).$ 
In particular, $J(g_{t_{1}})\cap \mbox{int}(K(h_{1}))\neq \emptyset .$ 
However, this contradicts (\ref{eq:kh1h1ik}). Thus, we have proved Claim 2. 

 By Claim 2 and that $K(h_{1})\subset \overline{D(b,s)}$, we obtain that 
 \begin{equation}
 \label{eq:jh1jgt1}
 J(h_{1})\cap J(g_{t_{1}})=\emptyset .
 \end{equation}
Combining (\ref{eq:jh1jgt1}) and (\ref{eq:kh1h1ik}), we also obtain that 
\begin{equation}
\label{eq:kh1inth1i}
K(h_{1})\subset \mbox{int}(h_{1}^{-1}(K(g_{t_{1}}))) \mbox { and } g_{t_{1}}^{-1}(K(h_{1}))\subset \mbox{int}(K(g_{t_{1}})). 
\end{equation} 
Moreover, by the definition of $t_{1}$ and (\ref{eq:kh1inth1i}), we obtain 
$h_{1}^{-1}(J(g_{t_{1}}))\cap g_{t_{1}}^{-1}(J(h_{1}))\neq \emptyset .$ 
In particular, $h_{1}^{-1}(J(h_{1},g_{t_{1}}))\cap g_{t_{1}}^{-1}(J(h_{1},g_{t_{1}}))\neq \emptyset .$ 
Combining this with (\ref{eq:h1gtinb}) and \cite[Theorems 1.5, 1.7]{S15}, 
we obtain that $(h_{1},g_{t_{1}})\in {\cal C}.$ 
Combining this with (\ref{eq:h1gt1pd}), we see that 
\begin{equation}
\label{eq:h1gt1dcb}
(h_{1},g_{t_{1}})\in (\partial {\cal D})\cap {\cal B}.
\end{equation}

We next prove the following claim. \\ 
{\bf Claim 3}. $g_{t_{1}}^{-1}(J(h_{1}))\cap J(h_{1})=\emptyset .$ 

 To prove this claim, suppose to the contrary that there exists a point $w\in g_{t_{1}}^{-1}(J(h_{1}))\cap J(h_{1}).$ 
 Then by (\ref{eq:kh1h1ik}), we have 
 $w\in h_{1}^{-1}(K(g_{t_{1}})).$ 
 If we would have that $w\in h_{1}^{-1}(\mbox{int}(K(g_{t_{1}})))$, 
 then (\ref{eq:kh1h1ik}) implies that $w\in g_{t_{1}}^{-1}(\mbox{int}(K(h_{1}))). $ 
 However, this contradicts  $w\in g_{t_{1}}^{-1}(J(h_{1}))$. Therefore, 
 we must have that $w\in h_{1}^{-1}(J(g_{t_{1}})).$ Hence $w\in h_{1}^{-1}(J(g_{t_{1}}))\cap J(h_{1}).$ 
Therefore $h_{1}(w)\in J(g_{t_{1}})\cap J(h_{1})$. However, this contradicts (\ref{eq:jh1jgt1}). 
Thus, we have proved Claim 3. 

 By (\ref{eq:kh1h1ik}) and Claim 3, we obtain that 
 \begin{equation}
 \label{eq:kh1intgt1}
 K(h_{1})\subset \mbox{int}(g_{t_{1}}^{-1}(K(h_{1})). \mbox{ In particular, } 
K(h_{1})\cap g_{t_{1}}^{-1}(J(h_{1}))=\emptyset . 
 \end{equation}
We next prove the following claim. \\ 
{\bf Claim 4}. For each $t\in [t_{0},t_{1}]$, $P^{\ast }(h_{1},g_{t})\subset \mbox{int}(K(h_{1})).$ 
  
 To prove this claim, let $t\in [t_{0},t_{1}].$ 
By (\ref{eq:h1gtinb}), we have that 
 $\mbox{CV}^{\ast }(h_{1})\cup \mbox{CV}^{\ast }(g_{t})\subset K(h_{1}).$ 
 Moreover, by (\ref{eq:kh1h1ik}), $g_{t}(K(h_{1}))\cup h_{1}(K(h_{1}))\subset K(h_{1}).$ 
 Therefore 
\begin{equation}
\label{eq:pasth1}
P^{\ast }(h_{1},g_{t})\subset K(h_{1}).
\end{equation}
Combining this with (\ref{eq:kh1intgt1}), we obtain that 
there exists a constant $\epsilon _{1}>0$ such that 
for each $z\in g_{t}^{-1}(J(h_{1})),$ for each $h\in \langle h_{1},g_{t}\rangle $ and 
for each connected component $V_{1}$ of $h^{-1}(D(z,\epsilon _{1}))$, 
\begin{equation}
\label{eq:gvdz1}
h:V_{1}\rightarrow D(z, \epsilon _{1}) \mbox{ is bijective.} 
\end{equation}
Since $\mbox{CV}^{\ast }(g_{t})=\{ b\} \subset \mbox{int}(K(h_{1}))$, 
there exists a number $\epsilon _{2}>0$  such that for each $z\in J(h_{1})$, 
for each connected component $V_{2}$ of $g_{t}^{-1}(D(z,\epsilon _{2}))$, we have that 
\begin{equation}
\label{eq:diameve1}
\mbox{diam}(V_{2})<\epsilon _{1} \mbox{ and } g_{t}:V_{2}\rightarrow D(z,\epsilon _{2}) 
\mbox{ is bijective}. 
\end{equation} 
Since $h_{1}$ is hyperbolic, there exists a number $\epsilon _{3}>0$ such that 
for each $z\in J(h_{1})$, for each $n\in \NN $ and for each connected component 
$V_{3}$ of $h_{1}^{-n}(D(z,\epsilon _{3}))$, we have that 
\begin{equation}
\label{eq:diameve2}
\mbox{diam}(V_{3})<\epsilon _{2} \mbox{ and } h_{1}^{n}: V_{3}\rightarrow D(z,\epsilon _{3}) \mbox{ is bijective}. 
\end{equation} 
By (\ref{eq:gvdz1}), (\ref{eq:diameve1}) and (\ref{eq:diameve2}), it follows that 
for each $z\in J(h_{1})$,  for each $g\in \langle h_{1},g_{t}\rangle $, and for each 
connected component $W$ of $g^{-1}(D(z,\epsilon _{3}))$, we have 
that $g:W\rightarrow D(z,\epsilon _{3})$ is bijective. 
Thus $J(h_{1})\cap P^{\ast }(h_{1},g_{t})=\emptyset .$ 
Combining this with (\ref{eq:pasth1}), we obtain that 
$P^{\ast }(h_{1},g_{t})\subset \mbox{int}(K(h_{1})).$ Thus we have proved Claim 4. 

By (\ref{eq:t1df}) and (\ref{eq:kh1h1ik}), for each $t\in [t_{0},t_{1}]$, 
$g_{t}(K(h_{1}))\subset K(h_{1})$ and $h_{1}(K(h_{1}))\subset K(h_{1}).$ 
Thus 
\begin{equation}
\label{eq:intkh1sf} 
\mbox{int}(K(h_{1}))\subset F(h_{1},g_{t}) \mbox{ for each } t\in [t_{0},t_{1}]. 
\end{equation}
Combining this with Claim 4, it follows that for each $t\in [t_{0},t_{1}]$, 
$P(h_{1},g_{t})\subset F(h_{1},g_{t}).$ Thus 
\begin{equation}
\label{eq:h1gth}
(h_{1},g_{t})\in {\cal H} \mbox{ for each } t\in [t_{0},t_{1}]. 
\end{equation}
By (\ref{eq:h1gt1dcb}), (\ref{eq:jh1jgt1}) and (\ref{eq:h1gth}), it follows that 
$(h_{1},g_{t_{1}})\in ((\partial {\cal D})\cap {\cal B}\cap {\cal H})\setminus {\cal I}
\subset ((\partial {\cal D})\cap {\cal B}\cap {\cal H})\setminus {\cal Q}. $
Moreover, $\deg (g_{t_{1}})=d.$ Therefore, statement \ref{spaceexample} 
in Theorem~\ref{t:spacetopology} holds. Thus we have proved Lemma~\ref{l:pfexample}. 
\end{proof}
We show some results which are related to statement~\ref{spaceexample} in Theorem~\ref{t:spacetopology}.
In order to do so, we need the following lemma. 
\begin{lem}
\label{l:oscinte}
Let $(h_{1},h_{2})\in {\cal B}$ with $(\deg (h_{1}),\deg (h_{2}))\neq (2,2).$ 
Suppose that $K(h_{1})\subset \mbox{{\em int}}(K(h_{2}))$ and suppose that 
$(h_{1},h_{2})$ satisfies the open set condition with 
$U:= (\mbox{{\em int}}(K(h_{2})))\setminus K(h_{1}).$ 
Then $h_{1}^{-1}(K(h_{2}))\subsetneqq h_{2}^{-1}(K(h_{1}))$ and 
$\mbox{{\em int}}(J(h_{1},h_{2}))=\emptyset .$ 

\end{lem}
\begin{proof}
Since $(h_{1},h_{2})\in {\cal B}$ and $(h_{1},h_{2})$ satisfies the open set condition with $U$, 
we obtain that 
$K(h_{1})\subset h_{2}^{-1}(K(h_{1}))\subset h_{1}^{-1}(K(h_{2}))\subset K(h_{2}).$ 
Combining this, the assumption $(\deg (h_{1}),\deg (h_{2}))\neq (2,2)$, 
and the method in the proof of Lemma~\ref{l:dh1-1log}, we obtain that 
$h_{1}^{-1}(K(h_{2}))\subsetneqq h_{2}^{-1}(K(h_{1})).$ Hence $h_{1}^{-1}(\overline{U})\cup 
h_{2}^{-1}(\overline{U})\subsetneqq \overline{U}.$ 
Let $G=\langle h_{1},h_{2}\rangle .$ 
Since $J(G)\subset \overline{U}$
(see \cite[Corollary 3.2]{HM}) and $h_{1}^{-1}(J(G))\cup h_{2}^{-1}(J(G))=J(G)$
 (see \cite[Lemma 1.1.4]{S1}), it follows that $J(G)\neq \overline{U}.$ 
Combining this with \cite[Proposition 4.3]{S4}, we obtain that $\mbox{int}(J(G))=\emptyset .$ 
\end{proof}
\begin{df}
A polynomial semigroup $G$ is said to be semi-hyperbolic if 
there exists an $N\in \NN $ and 
a $\delta >0$ such that 
for each $z\in J(G)$, for each $g\in G $ and for each connected component 
$V$ of $g^{-1}(B(z,\delta ))$, 
we have $\deg (g:V\rightarrow B(z,\delta ))\leq N.$  
\end{df}
\begin{thm}
\label{t:shshdhl2}
Under the assumptions of Lemma~\ref{l:oscinte}, 
suppose that 
$h_{1}$ is 
semi-hyperbolic (i.e. $\langle h_{1}\rangle $ is semi-hyperbolic). 
Then $\langle h_{1},h_{2}\rangle $ is semi-hyperbolic, $J(h_{1},h_{2})$ is porous and  
$\dim _{H}(J(h_{1},h_{2}))=\overline{\dim }_{B}(J(h_{1},h_{2}))=\inf \{ Z_{(h_{1},h_{2})}(z)\mid z\in \CCI \} <2 .$  
\end{thm}
\begin{proof}
Since $K(h_{1})\subset \mbox{int}(K(h_{2}))$, we have 
$K(h_{1})\subset \mbox{int}(h_{1}^{-1}(K(h_{2}))).$ 
By Lemma~\ref{l:oscinte}, it follows that $K(h_{1})\subset \mbox{int}(h_{2}^{-1}(K(h_{1}))).$ 
By using the similar method to that in the proof of Lemma~\ref{l:pfexample}, 
it is easy to see that 
$\langle h_{1},h_{2}\rangle $ is semi-hyperbolic. 
Let $A$ be the connected component of $\mbox{int}(K(h_{2}))$ with $K(h_{1})\subset A.$ 
Since $K(h_{1})\subset \mbox{int}(K(h_{2}))$, we have 
$h_{2}^{-1}(K(h_{1}))\subset \mbox{int}(K(h_{2})).$ 
Since $(h_{1},h_{2})$ satisfies the open set condition with $U$, we have 
$h_{2}^{-1}(K(h_{1}))\supset K(h_{1}).$  Moreover, since $(h_{1},h_{2})\in {\cal B}$, 
we have that $h_{2}^{-1}(K(h_{1}))$ is connected. 
It follows that $h_{2}^{-1}(K(h_{1}))\subset A.$ Thus 
$h_{2}^{-1}(A)\subset A.$ It implies that $\mbox{int}(K(h_{2}))$ is 
connected. Since $h_{2}(K(h_{1}))\subset K(h_{1})\subset \mbox{int}(K(h_{2}))$, 
it follows that $J(h_{2})$ is a quasicircle. 
Since $J(h_{2})$ is a quasicircle and 
since $F_{\infty }(h_{1})$ is a John domain (this is because $h_{1}$ is semi-hyperbolic, 
see \cite{CJY}), it follows that there exists a constant $\alpha \in (0,1)$ 
such that for each $r\in (0,1]$ and for each $x\in \overline{U}$, 
we have $l_{2}(U\cap D(x,r))\geq \alpha l_{2}(D(x,r))$, where $l_{2}$ denotes the 
$2$-dimensional Lebesgue measure on $\CC .$  
Since $(h_{1},h_{2})$ satisfies the open set condition with $U:=(\mbox{int}(K(h_{2})))\setminus K(h_{1})$,  
 \cite[Theorem 1.11]{SU2} implies that 
$\dim _{H}(J(h_{1},h_{2}))=\inf \{ Z_{(h_{1},h_{2})}(z)\mid z\in \CCI \} .$

Moreover, since $\langle h_{1}.h_{2}\rangle $ is semi-hyperbolic, 
$(h_{1},h_{2})$ satisfies the open set condition with $U$, 
and $J(h_{1},h_{2})\neq \overline{U}$ (see Lemma~\ref{l:oscinte}), 
\cite[Theorem 1.25]{S7} implies that 
$J(h_{1},h_{2})$ is porous and 
$\dim _{H}(J(h_{1},h_{2}))\leq \overline{\dim }_{B}(J(h_{1},h_{2}))<2.$ 
Hence we have proved our theorem. 
\end{proof}

Semi-hyperbolic polynomial semigroups with open set condition have many interesting 
properties. For the results, see \cite{SU2}. 

By using the method in the proof of Lemma~\ref{l:pfexample}, 
we can show the following theorem. 
\begin{thm}
\label{t:spgenex}
Suppose that $h_{1}\in {\cal P}$, $\langle h_{1}\rangle $ is postcritically bounded 
and {\em int}$(K(h_{1}))\neq \emptyset .$ Moreover, let $d\in \NN $, $d\geq 2$ and suppose that 
$(\deg (h_{1}), d)\neq (2,2).$ Then there exists an element $h_{2}\in {\cal P}$ 
such that 
all of the following hold.
\begin{enumerate}
\item \label{spgenex1} $(h_{1},h_{2})\in 
((\partial {\cal D})\cap {\cal C}\cap {\cal B})\setminus {\cal I}$  
and $\deg (h_{2})=d.$ 
\item \label{spgenex2}
We have that $K(h_{1})\subset \mbox{int}(K(h_{2}))$, 
$K(h_{1})\subset \mbox{{\em int}}(h_{1}^{-1}(K(h_{2})))$ and 
$h_{2}^{-1}(K(h_{1}))\subset \mbox{{\em int}}(K(h_{2})).$ 
Moreover, $(h_{1},h_{2})$ satisfies the open set condition with 
$U:=(\mbox{{\em int}}(K(h_{2}))\setminus K(h_{1}).$ 
Furthermore, $\mbox{{\em int}}(J(h_{1},h_{2}))=\emptyset .$
\item \label{spgenex3}
If, in addition to the assumptions of our theorem, 
$h_{1}$ is semi-hyperbolic, 
then $\langle h_{1},h_{2}\rangle $ is semi-hyperbolic, 
$J(h_{1},h_{2})$ is porous and 
 and 
$\dim _{H}(J(h_{1},h_{2}))=\overline{\dim }_{B}(J(h_{1},h_{2}))=\inf \{ Z_{(h_{1},h_{2})}(z)\mid z\in \CCI \}<2$. 
\end{enumerate} 

\end{thm}
\begin{proof}
By using the method in the proof of Lemma~\ref{l:pfexample} and Lemma~\ref{l:oscinte}, 
we can show that there exists an element $h_{2}\in {\cal P}$ which satisfies 
properties~\ref{spgenex1},\ref{spgenex2}. 
We now suppose $h_{1}$ is semi-hyperbolic. 
Then by Theorem~\ref{t:shshdhl2}, we obtain that 
$\langle h_{1},h_{2}\rangle $ is semi-hyperbolic, 
$J(h_{1},h_{2})$ is porous and 
$\dim _{H}(J(h_{1},h_{2}))=\overline{\dim }_{B}(J(h_{1},h_{2}))=\inf \{ Z_{(h_{1},h_{2})}(z)\mid z\in \CCI \}<2$. 
\end{proof}

\vspace{-3mm} 
\subsection{Proof of Theorem~\ref{t:spaceth1h2}}
\vspace{-2mm} 
In this subsection, we prove Theorem~\ref{t:spaceth1h2}. 
We need the following. 
\begin{df}
Let Rat be the space of all non-constant rational maps on $\CCI $, 
endowed with the distance $\kappa $ which is defined by 
$\kappa (f,g):=\sup _{z\in \CCI }d(f(z),g(z))$, where $d$ denotes the spherical distance on $\CCI .$ 
All the notations and definitions in section~\ref{Introduction} are generalized to the settings  
of rational semigroups (i.e., subsemigroups of Rat) 
and random dynamical systems of rational maps. 
\end{df}

\begin{df}
Let $h=(h_{1},\ldots, h_{m})\in (\Rat)^{m}.$ Let $U$ be a non-empty open subset of $\CCI .$ 
We say that $h=(h_{1},\ldots ,h_{m})$ satisfies the open set condition with $U$ 
if $\cup _{j=1}^{m}h_{j}^{-1}(U)\subset U$ and 
$h_{i}^{-1}(U)\cap h_{j}^{-1}(U)=\emptyset $ for each $(i,j)$ with $i\neq j.$ 
We say that $h=(h_{1},\ldots ,h_{m})$ satisfies the open set condition if there exists 
a non-empty open set $U$ such that $h=(h_{1},\ldots ,h_{m})$ satisfies the open set condition with $U.$ 
\end{df}
\begin{df}
For a rational semigroup $G$ and a subset $A$ of $\CCI $, we set 
$G(A):=\cup _{g\in G}g(A)$ and $G^{-1}(A)=\cup _{g\in G}g^{-1}(A).$ 
\end{df}
\begin{df}
We denote by $B(\CCI )$ the set of all bounded Borel measurable complex-valued functions 
on $\CCI .$ 
For each $\tau \in {\frak M}_{1}(\Rat)$, 
we denote by $M_{\tau }: B(\CCI )\rightarrow B(\CCI )$ 
the operator defined by 
$M_{\tau }(\varphi )(z)=\int _{\Rat}\varphi (g(z)) d\tau (g)$ 
for each $\varphi \in B(\CCI )$ and $z\in \CCI .$ 
Note that $M_{\tau }(C(\CCI ))\subset C(\CCI ).$ 
This $M_{\tau }$ is called the transition operator with respect to the random dynamical system
associated with $\tau .$ 
\end{df}

\begin{lem}
\label{l:oscncp}
Let $m\in \NN $ with $m\geq 2.$ 
Let $h_{1},\ldots, h_{m}\in \mbox{{\em Rat}} $ and 
let $G=\langle h_{1},\ldots , h_{m}\rangle .$ 
Let $(p_{1},\ldots ,p_{m})\in (0,1)^{m}$ with 
$\sum _{j=1}^{m}p_{j}=1.$ 
Let $\tau =\sum _{j=1}^{m}p_{j}\delta _{h_{j}}$. 
Let $\varphi \in B(\CCI )$ such that 
$M_{\tau }(\varphi )=a\varphi $ for some $a\in \CC $ with $|a|=1.$ 
Suppose that for each connected component $\Omega $ of $F(G)$, 
$\varphi |_{\Omega }$ is constant. 
 Let $$A=\{ z_{0}\in \CCI \mid  \mbox{ for each neighborhood  } V \mbox{ of } z_{0} \mbox{ in } \CCI, 
\varphi |_{V} \mbox{ is not constant on } V\} . $$ 
Suppose that $(h_{1},\ldots ,h_{m}) $ satisfies the open set condition with 
an open set $U.$ Then either $A=J(G)$ or $A\subset \partial U.$   
\end{lem}
\begin{proof}
From the assumption of our lemma, we have 
$A\subset J(G).$ Moreover, by the open set condition, \cite[Lemma 2.3(f)]{S3} implies that  
$J(G)\subset \overline{U}.$ 
Suppose that there exists a point $z_{0}\in A\cap U.$ 
We now prove the following claim. 

\noindent Claim 1. If $w_{0}\in h_{i}^{-1}(z_{0})$ for some $i$, then 
$w_{0}\in A\cap U.$ 

To prove claim 1, let $h_{i}(w_{0})=z_{0}.$ 
Then there exists an open disk neighborhood $W_{0}$ of $w_{0}$ such that 
$h_{i}(W_{0})\subset U.$ By the open set condition, 
we have that for each $j$ with $j\neq i$, 
$h_{j}(W_{0})\subset (\CCI \setminus \overline{U})\subset F(G).$ 
Combining this with  $A\subset J(G)$, we obtain that 
for each $j$ with $j\neq i$ and for each $w_{1},w_{2}\in W_{0}$, 
$\varphi (h_{j}(w_{1}))=\varphi (h_{j}(w_{2})).$ 
Hence, by $M_{\tau }(\varphi )=a\varphi $, we get that 
for each $w_{1},w_{2}\in W_{0}$, 
$a(\varphi (w_{1})-\varphi (w_{2}))=
M_{\tau }(\varphi )(w_{1})-M_{\tau }(\varphi )(w_{2})
=\sum _{j=1}^{m}p_{j}(\varphi (h_{j}(w_{1}))-\varphi (h_{j}(w_{2})))
=p_{i}(\varphi (h_{j}(w_{1}))-\varphi (h_{j}(w_{2}))).$ 
Since $z_{0}\in A$, it follows that $w_{0}\in A.$ 
Moreover, by the open set condition, $w_{0}\in U.$ 
Therefore, claim 1 holds. 

By claim 1, we obtain that $\overline{G^{-1}(z_{0})}\subset A.$ 

We now prove the following. 

\noindent Claim 2. $z_{0}\not\in E(G).$ 

To prove claim 2, we first observe that  
$\{ (h_{w_{1}}\cdots h_{w_{n}})^{-1}(U)\} _{(w_{1},\ldots ,w_{n})\in \{ 1,\ldots ,m\} ^{n}}$ 
are mutually disjoint because of the open set condition. 
Therefore 
$\sharp \cup _{(w_{1},\ldots ,w_{n})\in \{ 1,\ldots ,m\} ^{n}}(h_{w_{1}}\cdots h_{w_{n}})^{-1}(z_{0})\geq m^{n}.$ 
Thus $\sharp G^{-1}(z_{0})=\infty $ and $z_{0}\not\in E(G).$ Hence we have proved claim 2. 

By the fact $\overline{G^{-1}(z_{0})}\subset A$, claim 2 and 
\cite[Lemma 2.3(e)]{S3}, it follows that $J(G)\subset \overline{G^{-1}(z_{0})}\subset A.$ 
Thus we have proved  Lemma~\ref{l:oscncp}. 
\end{proof} 
\begin{df}
Let $\tau \in {\frak M}_{1}({\cal P})$. 
We set $\tilde{\tau }:=\otimes _{n=1}^{\infty }\tau \in {\frak M}_{1}({\cal P}^{\NN }).$ 
Moreover, for each $z\in \CCI $, we set 
$T_{\infty ,\tau }(z)=\tilde{\tau }(\{ \gamma =(\gamma _{1},\gamma _{2},\ldots )\in {\cal P}^{\NN }
\mid \gamma _{n}\cdots \gamma _{1}(z)\rightarrow \infty \mbox{ as }n\rightarrow \infty \} ).$ 
Furthermore, we denote by $G_{\tau }$ the polynomial semigroup generated by 
supp$\,\tau.$ Namely, 
$G_{\tau }=\{ h_{n}\circ \cdots \circ h_{1}\mid n\in \NN, \forall h_{j}\in \mbox{supp}\,\tau \} .$  
Moreover, we set $\Gamma _{\tau }:=\mbox{supp}\,\tau$ and 
$X_{\tau }:=(\mbox{supp}\,\tau )^{\NN }.$ 
\end{df}

\begin{lem}
\label{l:intT1}
Let $\tau \in {\frak M}_{1}({\cal P}).$ 
Suppose $\mbox{{\em int}}(\hat{K}(G_{\tau }))\neq \emptyset .$ 
Then 
$\mbox{{\em int}}(T_{\infty, \tau }^{-1}(\{ 1\}))\subset F(G_{\tau }).$ 
\end{lem}
\begin{proof}
We first prove the following claim.\\ 
Claim. For each $z_{0}\in T_{\infty ,\tau }^{-1}(\{ 1\} )$, 
there exists no $g\in G_{\tau }$ with $g(z_{0})\in \mbox{int}(\hat{K}(G_{\tau })).$

To prove this claim, let $z_{0}\in T_{\infty ,\tau }^{-1}(\{ 1\} )$ and 
suppose there exists an element $g\in G_{\tau }$ with $g(z_{0})\in \mbox{int}(\hat{K}(G_{\tau })).$ 
Let $h_{1},\ldots ,h_{m}\in \G_{\tau }$ be some elements with 
$g=h_{m}\circ \cdots \circ h_{1}.$ 
Then there exists a neighborhood $W$ of $(h_{1},\ldots ,h_{m})$ in $\G_{\tau }^{m}$ 
such that for each $\omega =(\omega _{1},\ldots ,\omega _{m})\in W$, 
$\omega _{m}\cdots \omega _{1}(z_{0})\in \mbox{int}(\hat{K}(G_{\tau })).$ 
Therefore for each $\g \in X_{\tau }$ with $(\g _{1},\ldots ,\g _{m})\in W$, 
$\{ \g _{n,1}(z_{0})\} _{n\in \NN }$ is bounded. 
Thus 
$T_{\infty ,\tau }(z_{0})\leq 1-\tilde{\tau }(\{ \g \in X_{\tau }\mid (\g _{1},\ldots ,\g _{m})\in W\} )<1.$ 
This is a contradiction. Hence we have proved the claim.

 From this claim, $G_{\tau }(\mbox{int}(T_{\infty ,\tau }^{-1}(\{ 1\} )))\subset \CCI \setminus \mbox{int}(\hat{K}(G_{\tau })).$ 
Therefore $ \mbox{int}(T_{\infty ,\tau }^{-1}(\{ 1\} ))\subset F(G_{\tau }).$ 
Thus we have proved our lemma. 
\end{proof}
\begin{lem}
\label{l:ovintk}
Let $(h_{1},h_{2})\in \overline{{\cal D}}\cap {\cal B}\cap {\cal H}.$ 
Suppose that $K(h_{1})\subsetneqq K(h_{2})$.  
Then $\overline{(\mbox{{\em int}}(K(h_{2})))\setminus K(h_{1})}=K(h_{2})\setminus \mbox{{\em int}}(K(h_{1})).$ 
\end{lem}
\begin{proof}
It is clear that $\overline{(\mbox{int}(K(h_{2})))\setminus K(h_{1})}\subset K(h_{2})\setminus \mbox{int}(K(h_{1})).$ 
Let $A_{1}:=J(h_{2})\cap (\CC \setminus K(h_{1}))$, 
$A_{2}:=(\mbox{int}(K(h_{2})))\cap (\CC \setminus K(h_{1}))$, 
$A_{3}:= J(h_{2})\cap J(h_{1})$ and 
$A_{4}:= (\mbox{int}(K(h_{2})))\cap J(h_{1}).$ 
Then we have 
\begin{eqnarray*}
K(h_{2})\setminus \mbox{int}(K(h_{1})) & = &
(J(h_{2})\cup \mbox{int}(K(h_{2})))\cap (\CC \setminus \mbox{int}(K(h_{1})))\\ 
& = & (J(h_{2})\cup \mbox{int}(K(h_{2})))\cap ((\CC \setminus K(h_{1}))\cup J(h_{1}))\\ 
& = & A_{1}\cup A_{2}\cup A_{3}\cup A_{4}.
\end{eqnarray*} 
We want to see that $A_{i}\subset \overline{(\mbox{int}(K(h_{2})))\setminus K(h_{1})}$ for each $i=1,\ldots ,4.$ 

It is clear that $A_{2}\subset \overline{(\mbox{int}(K(h_{2})))\setminus K(h_{1})}.$ 

Suppose $z_{0}\in A_{1}.$ Let $\epsilon _{0}>0$ be a small number such that 
$B(z_{0},\epsilon _{0})\subset \CC \setminus K(h_{1}).$ 
Then for each $\epsilon \in (0,\epsilon _{0})$ there exists a 
point $z_{1}\in B(z_{0},\epsilon )\cap \mbox{int}(K(h_{2})).$ 
Thus $z_{1}\in (B(z_{0},\epsilon )\cap \mbox{int}(K(h_{2})))\setminus K(h_{1}).$ 
Hence $z_{0}\in \overline{(\mbox{int}(K(h_{2})))\setminus K(h_{1})}.$ 
Therefore $A_{1}\subset \overline{(\mbox{int}(K(h_{2})))\setminus K(h_{1})}.$  

We now let $z_{0}\in A_{4}.$ Then there exists a number $\epsilon _{0}>0$ 
such that $B(z_{0},\epsilon _{0})\subset \mbox{int}(K(h_{2})).$ 
For each $\epsilon \in (0,\epsilon _{0})$, 
there exists an element $z_{1}\in B(z_{0},\epsilon )\cap (\CC \setminus K(h_{1})).$ 
Hence $z_{1}\in B(z_{0},\epsilon )\cap (\CC \setminus K(h_{1}))\cap \mbox{int}(K(h_{2})).$ 
Therefore $z_{0}\in \overline{(\mbox{int}(K(h_{2})))\setminus K(h_{1})}.$ 
Thus $A_{4}\subset \overline{(\mbox{int}(K(h_{2})))\setminus K(h_{1})}.$ 

 We now let $z_{0}\in A_{3}.$ Since we are assuming $K(h_{1})\subsetneqq K(h_{2})$, 
 there exists a point $z_{1}\in J(h_{2})\setminus K(h_{1})$. 
 Then for each $\epsilon >0$ there exists a point $a_{n}\in h_{2}^{-n}(z_{1})$ for some $n\in \NN $ 
 such that $a_{n}\in B(z_{0},\epsilon ).$ There exists a point $w_{1}\in (\mbox{int}(K(h_{2})))\setminus K(h_{1})$ 
 arbitrarily close to $z_{1}.$ Hence there exists a point $b_{n}\in h_{2}^{-n}(w_{1})$ 
 arbitrarily close to $a_{n}.$ 
Since $h_{2}^{-1}((\mbox{int}(K(h_{2})))\setminus K(h_{1})) 
\subset (\mbox{int}(K(h_{2})))\setminus K(h_{1})$ (see Proposition~\ref{p:bdyosc}), 
it follows that $b_{n}\in (\mbox{int}(K(h_{2})))\setminus K(h_{1}).$ 
Therefore we obtain that $z_{0}\in \overline{(\mbox{int}(K(h_{2})))\setminus K(h_{1})}.$ 
Thus $A_{3}\subset \overline{(\mbox{int}(K(h_{2})))\setminus K(h_{1})}.$ 
From these arguments, it follows that 
$\overline{(\mbox{int}(K(h_{2})))\setminus K(h_{1})}=K(h_{2})\setminus \mbox{int}(K(h_{1})).$ 


Thus we have proved Lemma~\ref{l:ovintk}. 
\end{proof} 

The following notion is the key to investigating random complex dynamical systems 
which are associated with  rational semigroups.  
\begin{df}
Let $G$ be a rational semigroup. 
We set $J_{\ker }(G)=\bigcap _{g\in G}g^{-1}(J(G)).$ 
This $J_{\ker }(G)$ is called the {\em kernel Julia set} of $G$. 
Moreover, for a finite subset $\{ h_{1},\ldots ,h_{m}\} $ of Rat, 
we set $J_{\ker }(h_{1},\ldots ,h_{m}):=J_{\ker }(\langle h_{1},\ldots ,h_{m}\rangle ) .$ 
\end{df}
Note that $J_{\ker }(G)$ is the largest forward 
invariant subset of $J(G)$ under the action of $G.$ 
We remark that 
if $G$ is a group or if $G$ is a commutative semigroup, 
then $J_{\ker }(G)=J(G).$ 
However, for a general rational semigroup $G$ generated by a family of 
rational maps $h$ with $\deg (h)\geq 2$, it may happen that 
$\emptyset =J_{\ker }(G)\neq J(G) $.

\begin{prop}
\label{p:jh1jh2jk}
Let $(h_{1},h_{2})\in (\partial {\cal C})\cap {\cal B}\cap {\cal H}.$ 
Let $G=\langle h_{1},h_{2}\rangle .$ 
Then we have all of the following. 
\begin{itemize}
\item[{\em (1)}] 
$J_{\ker }(G)=J(h_{1})\cap J(h_{2}).$ 
\item[{\em (2)}] Either $K(h_{1})\subset K(h_{2})$ or $K(h_{2})\subset K(h_{1}).$ 
\item[{\em (3)}] If $K(h_{1})\subset K(h_{2})$ then 
$F_{\infty }(G)=\CCI \setminus K(h_{2})$ and $\hat{K}(G)=K(h_{1}).$ 
\end{itemize}
\end{prop}
\begin{proof}
If $K(h_{1})=K(h_{2})$, then 
$J(h_{1})=J(h_{2})$ and $J(G)=J(h_{1})=J(h_{2})$. Thus (1)--(3) hold. 

We now suppose $K(h_{1})\neq K(h_{2}).$ 
By Proposition~\ref{p:bdyosc}, either $K(h_{1})\subset K(h_{2})$ or 
$K(h_{2})\subset K(h_{1})$. 
We suppose $K(h_{1})\subset K(h_{2}).$ 
Since we are assuming $K(h_{1})\neq K(h_{2})$, we obtain $K(h_{1})\subsetneqq K(h_{2}).$ 
Therefore $U:=(\mbox{int}{K(h_{2})})\setminus K(h_{1})$ is a non-empty open set. 
Moreover, by Proposition~\ref{p:bdyosc}, we have that 
$(h_{1},h_{2})$ satisfies the open set condition with $U.$ 
Therefore $J(G)\subset \overline{U}.$ Moreover, by Lemma~\ref{l:ovintk}, 
we have $\overline{U}=K(h_{2})\setminus \mbox{int}(K(h_{1})).$ Thus 
we obtain $J(G)\subset K(h_{2})\setminus \mbox{int}(K(h_{1})).$  
Therefore 
\begin{equation}
\label{eq:finftygcci}
F_{\infty }(G)=\CCI \setminus K(h_{2}).
\end{equation}

We now prove the following claim. \\ 
Claim. $\hat{K}(G)=K(h_{1}).$ 

To prove this claim, it is easy to see that $\hat{K}(G)\subset K(h_{1}).$ 
By Proposition~\ref{p:bdyosc}, we have that
\begin{equation}
\label{eq:hj-1intkh2}
h_{j}^{-1}((\mbox{int}(K(h_{2})))\setminus K(h_{1}))\subset 
(\mbox{int}(K(h_{2})))\setminus K(h_{1}) \mbox{ for each } j=1,2.
\end{equation}
Suppose $h_{2}(K(h_{1}))\cap (\CC \setminus K(h_{1}))\neq \emptyset .$ 
Then $h_{2}(\mbox{int}(K(h_{1})))\cap (\CC \setminus K(h_{1}))\neq \emptyset .$ 
Since $h_{2}(\mbox{int}(K(h_{1})))\subset h_{2}(\mbox{int}(K(h_{2})))
\subset \mbox{int}(K(h_{2}))$, 
it follows that $h_{2}^{-1}((\mbox{int}(K(h_{2})))\setminus K(h_{1}))\cap \mbox{int}(K(h_{1}))\neq \emptyset .$ 
However, this contradicts (\ref{eq:hj-1intkh2}). Thus we must have that 
$h_{2}(K(h_{1}))\subset K(h_{1}).$ 
Hence $K(h_{1})\subset \hat{K}(G).$ 
Therefore $K(h_{1})=\hat{K}(G).$ Thus we have proved the above claim. 

By Claim and (\ref{eq:finftygcci}), we obtain 
$J(h_{1})\cap J(h_{2})\subset \hat{K}(G)\cap \overline{F_{\infty }(G)}.$ 
Therefore for each $j=1,2,$ we have 
$$h_{j}(J(h_{1})\cap J(h_{2}))\subset \hat{K}(G)\cap \overline{F_{\infty }(G)} 
=K(h_{1})\cap (\CCI \setminus \mbox{int}(K(h_{2}))).$$ 
Since $K(h_{1})\subset K(h_{2})$, we have 
$\mbox{int}(K(h_{1}))\subset \mbox{int}(K(h_{2})).$ 
Hence $$K(h_{1})\cap (\CCI \setminus \mbox{int}(K(h_{2})))\subset 
J(h_{1})\cap (\CC \setminus \mbox{int}(K(h_{2})))=
J(h_{1})\cap ((\CC \setminus K(h_{2}))\cup J(h_{2}))
=J(h_{1})\cap J(h_{2}).$$
It follows that $h_{j}(J(h_{1})\cap J(h_{2}))\subset J(h_{1})\cap J(h_{2})$ for each 
$j=1,2.$ Therefore $J(h_{1})\cap J(h_{2})\subset J_{\ker }(G).$ 

 We now let $z_{0}\in J_{\ker }(G).$ 
Then  we have $z_{0}\in K(h_{1})\cap K(h_{2})=K(h_{1}).$  
By Proposition~\ref{p:bdyosc}, we have $h_{2}(\mbox{int}(K(h_{1})))\subset \mbox{int}(K(h_{1})).$ 
Therefore $g(\mbox{int}(K(h_{1})))\subset \mbox{int}(K(h_{1}))$ for each $g\in G.$ 
It implies that $\mbox{int}(K(h_{1}))\subset F(G).$ 
Since $z_{0}\in J(G)\cap K(h_{1})$, it follows that $z_{0}\in J(h_{1}).$ 
We now want to show that $z_{0}\in J(h_{2}).$ 
Suppose that $z_{0}\in \mbox{int}(K(h_{2})).$ By Proposition~\ref{p:bdyosc}, 
$J(h_{2})$ is a quasicircle and $h_{2}$ has a unique attracting fixed point $c\in \mbox{int}(K(h_{2})).$ 
 Since $c\in P(G)\subset F(G)$, it follows that there exists a number $n\in \NN $ such that 
 $h_{2}^{n}(z_{0})\in F(G).$ However, this contradicts $z_{0}\in J_{\ker }(G).$ 
 Therefore $z_{0}\in J(h_{2}).$ Thus $z_{0}\in J(h_{1})\cap J(h_{2}).$ 
 Hence $J_{\ker }(G)\subset J(h_{1})\cap J(h_{2}).$ 
 Therefore $J_{\ker }(G)=J(h_{1})\cap J(h_{2}).$ Thus we have proved Proposition~\ref{p:jh1jh2jk}. 
\end{proof}

\begin{lem}
\label{l:pfmainth8}
Statement~\ref{spacemainth8} in Theorem~\ref{t:spaceth1h2} holds. 
\end{lem}
\begin{proof}
Let $(h_{1},h_{2})\in ({\cal D}\cap {\cal B})\cup ((\partial {\cal C})\cap {\cal B}\cap {\cal H})$ 
and let $0<p<1.$ Let 
$\varphi (z) =T(h_{1},h_{2},p,z).$ 
Let $$A=\{ z_{0}\in \CCI \mid \mbox{ for each neighborhood } V \mbox{ of } z_{0}, 
\varphi |_{V} \mbox{ is not constant}\} .$$  
Let $G=\langle h_{1},h_{2}\rangle .$ 
Since $\emptyset \neq P^{\ast }(G)\subset \hat{K}(G)$, 
$\varphi \not \equiv 1.$ Hence $\varphi :\CCI \rightarrow [0,1]$ is not constant.  
If $(h_{1},h_{2})\in {\cal D}\cap {\cal B}$, 
then by Lemma~\ref{l:bdjconti} and  \cite[Lemmas 3.75, 3.73, Theorem 3.22]{Splms10}, it follows 
that $J(G)=A.$ 

 We now suppose $(h_{1},h_{2})\in (\partial {\cal C})\cap {\cal B}\cap {\cal H}.$ 
 We consider the following two cases. Case 1. $(h_{1},h_{2})\in {\cal Q}. $ 
 Case 2. $(h_{1},h_{2})\not\in {\cal Q}.$ 
 If we have case 1, then $\hat{K}(G)=K(h_{1})=K(h_{2})$ and $J(G)=J(h_{1})=J(h_{2})$. 
 Therefore $A=J(G)=J(h_{1})=J(h_{2}).$ Thus, the remaining case is Case 2. 
 Let $(h_{1},h_{2})\in ((\partial {\cal C})\cap {\cal B}\cap {\cal H})\setminus {\cal Q}.$ 
 By Proposition~\ref{p:bdyosc}, we may assume that 
 $K(h_{1})\subset K(h_{2})$. Then by Proposition~\ref{p:bdyosc} again, we have 
\begin{equation} 
\label{eq:kh1sh1-1}
K(h_{1})\subset h_{1}^{-1}(K(h_{2}))\subset h_{2}^{-1}(K(h_{1}))\subset K(h_{2})
\end{equation}
 and  
 $(h_{1},h_{2})$ satisfies the open set condition with 
 $U:=(\mbox{int}(K(h_{2})))\setminus K(h_{1}).$ 
 Moreover, by Lemma~\ref{l:dh1-1log}, we have 
 $h_{1}^{-1}(K(h_{2}))\neq h_{2}^{-1}(K(h_{1})).$ 
 Therefore $h_{1}^{-1}(J(h_{2}))\setminus h_{2}^{-1}(J(h_{1}))\neq \emptyset .$ 
Moreover, by Proposition~\ref{p:jh1jh2jk} (1), that $K(h_{1})\subset K(h_{2})$ 
and that $(h_{1},h_{2})\not\in {\cal Q}$,  we obtain that 
$J(h_{2})\cap J(h_{1})$ is a nowhere dense subset of $J(h_{2}).$ 
It follows that $(h_{1}^{-1}(J(h_{2})))\setminus (h_{2}^{-1}(J(h_{1}))\cup J(h_{1}))\neq \emptyset .$   
 Let $z_{0}\in h_{1}^{-1}(J(h_{2}))\setminus (h_{2}^{-1}(J(h_{1}))\cup J(h_{1}))$ be a point. 
 Since $h_{1}^{-1}(K(h_{2}))\subset h_{2}^{-1}(K(h_{1}))$ and 
 $z_{0}\not\in h_{2}^{-1}(J(h_{1}))$, 
we obtain $z_{0}\in h_{2}^{-1}(\mbox{int}(K(h_{1}))).$ 
Hence 
\begin{equation}
\label{eq:h2z0in}
h_{2}(z_{0})\in \mbox{int}(K(h_{1})).
\end{equation} 
Since 
$K(h_{1})\subset h_{1}^{-1}(K(h_{2}))\subset h_{2}^{-1}(K(h_{1}))\subset K(h_{2})$, 
we obtain $h_{i}(K(h_{1}))\subset K(h_{1})$ for each $i=1,2.$ 
Therefore $\mbox{int}(K(h_{1}))\subset F(G).$ Combining this with (\ref{eq:h2z0in}), 
we obtain 
\begin{equation}
\label{eq:h2z0fg}
h_{2}(z_{0})\in F(G). 
\end{equation}
We now show the following claim. 

\noindent Claim 1. $z_{0}\in U.$ 

To prove this claim, since $z_{0}\in h_{2}^{-1}(J(h_{1}))$, 
the fact $K(h_{1})\subset h_{1}^{-1}(K(h_{2}))\subset h_{2}^{-1}(K(h_{1}))\subset K(h_{2})$ 
implies that $z_{0}\in K(h_{2}).$ 
If $z_{0}\in J(h_{2})$, then (\ref{eq:h2z0in}) implies 
that $h_{2}(z_{0})\in J(h_{2})\cap \mbox{int}(K(h_{1})).$ However, this contradicts 
$K(h_{1})\subset K(h_{2}).$ Therefore, 
we must have that 
\begin{equation}
\label{eq:z0inint}
z_{0}\in \mbox{int}(K(h_{2})).
\end{equation}
If $z_{0}\in K(h_{1})$, then 
the fact $z_{0}\in h_{1}^{-1}(J(h_{2}))\setminus (h_{2}^{-1}(J(h_{1}))\cup J(h_{1}))$ 
implies $z_{0}\not\in J(h_{1}).$ Hence $z_{0}\in \mbox{int}(K(h_{1})).$ 
Therefore $z_{0}\in (\mbox{int}(K(h_{1})))\cap h_{1}^{-1}(J(h_{2})).$ 
However, this contradicts $K(h_{1})\subset h_{1}^{-1}(K(h_{2})).$ 
Thus, we must have that 
\begin{equation}
\label{eq:z0notink}
z_{0}\not\in K(h_{1}). 
\end{equation} 
By (\ref{eq:z0inint})(\ref{eq:z0notink}), we obtain that $z_{0}\in U.$ Thus claim 1 holds. 

Since $z_{0}\in h_{1}^{-1}(J(h_{2}))$, we have 
$h_{1}(z_{0})\in J(h_{2}).$ Moreover, by (\ref{eq:kh1sh1-1}), 
we have $h_{1}(F_{\infty }(h_{2}))\subset F_{\infty }(h_{2})$. 
Therefore for each $g\in G$, we have $g(F_{\infty }(h_{2}))\subset F_{\infty }(h_{2}).$ 
It follows that $F_{\infty }(h_{2})\subset F(G)$. Therefore 
$F_{\infty}(h_{2})=F_{\infty }(G).$ Thus $h_{1}(z_{0})\in \overline{F_{\infty }(G)}\cap J(G).$ 
We now prove the following claim. 

\noindent Claim 2. For each neighborhood $W$ of $h_{1}(z_{0})$ in $\CCI $, the function 
$\varphi |_{W}:W\rightarrow [0,1] $ is not constant. 

To prove claim 2, suppose that there exists a neighborhood $W$ of $h_{1}(z_{0})$ in $\CCI $
and a constant $c\in [0,1]$ such that $\varphi |_{W}\equiv c.$ 
Since $h_{1}(z_{0})\in \overline{F_{\infty} (G)}$, \cite[Lemma 5.24]{Splms10} implies that 
$c=1.$ Thus $h_{1}(z_{0})\in \mbox{int}(\varphi ^{-1}(1)).$ Combining this with Lemma~\ref{l:intT1}, 
we obtain that $h_{1}(z_{0})\in F(G).$ 
However, this contradicts $h_{1}(z_{0})\in J(G).$ Therefore, claim 2 holds. 

By the equation $\varphi (z)=p\varphi (h_{1}(z))+(1-p)\varphi (h_{2}(z))$, 
the fact $\varphi :F(G)\rightarrow [0,1]$ is locally constant (see \cite[Lemma 5.27]{Splms10}), 
(\ref{eq:h2z0fg}), and claim 2, we obtain that $z_{0}\in A.$ From this,  claim 1, Lemma~\ref{l:oscncp} 
and the fact that $(h_{1},h_{2})$ satisfies the open set condition with $U$,  
it follows that $A=J(G).$ 

Thus statement~\ref{spacemainth8} in Theorem~\ref{t:spaceth1h2} holds. 
\end{proof}
We now prove statement~\ref{spacemainth9} in Theorem~\ref{t:spaceth1h2}. 
\begin{lem}
\label{l:pfmainth9}
Statement~\ref{spacemainth9} in Theorem~\ref{t:spaceth1h2} holds. 
\end{lem}
\begin{proof}
Let $(h_{1},h_{2})\in (\partial {\cal C})\cap {\cal B}\cap {\cal H})$ and let $0<p<1.$ 
Let $G=\langle h_{1},h_{2}\rangle .$ 
Let $\varphi (z)=T(h_{1},h_{2},p,z).$ 
Suppose $J(h_{1})\cap J(h_{2})=\emptyset .$ 
Then by Proposition~\ref{p:jh1jh2jk}, we obtain that $J_{\ker }(G)=\emptyset .$ 
By \cite[Theorem 3.22]{Splms10}, it follows that $\varphi :\CCI \rightarrow \RR $ is continuous on $\CCI .$ 

We now suppose that $J(h_{1})\cap J(h_{2})\neq \emptyset .$ 
By Proposition~\ref{p:jh1jh2jk} (3), we have $J(h_{1})\cap J(h_{2})\subset \overline{F_{\infty }(G)}\cap 
\hat{K}(G).$  Since $\varphi |_{F_{\infty }(G)}\equiv 1$ (see \cite[Lemma 5.24]{Splms10}) and 
$\varphi |_{\hat{K}(G)}\equiv 0$, it follows that $\varphi :\CCI \rightarrow \RR $ 
is not continuous at each point in $J(h_{1})\cap J(h_{2}).$ Thus 
statement~\ref{spacemainth9} in Theorem~\ref{t:spaceth1h2} holds.
\end{proof} 
\vspace{-5mm}
\subsection{Proof of Theorem~\ref{t:timportant}}
\vspace{-2mm} 
In this subsection, we prove Theorem~\ref{t:timportant}. 
We also show a result on the Fatou components (Theorem~\ref{t:infcomp}) and 
a result in which we do not assume hyperbolicity (Theorem~\ref{t:oscanal}).  

\begin{df}
For an element $\tau \in {\frak M}_{1}(\Rat)$, 
we denote by $U_{\tau }$ the space of all linear combinations of unitary eigenvectors of 
$M_{\tau }:C(\CCI )\rightarrow C(\CCI )$, where a non-zero element $\varphi \in C(\CCI)$ is said to be 
a unitary eigenvector of $M_{\tau }$ if there exists an element $a\in \CC $ with $|a|=1$ 
such that $M_{\tau }(\varphi )=a\varphi .$ 
\end{df}

\begin{lem}
\label{l:pfmainth10}
Statement~\ref{spacemainth10} in Theorem~\ref{t:timportant} holds. 
\end{lem}
\begin{proof}
Let 
$(h_{1},h_{2})\in ({\cal D}\cap {\cal B}\cap {\cal H})\cup 
(((\partial {\cal C})\cap {\cal B}\cap {\cal H})
\setminus {\cal I})$. Let $p\in (0,1).$ 
By Lemma~\ref{l:pfmainth1} and Proposition~\ref{p:jh1jh2jk}, 
we obtain $J_{\ker }(h_{1},h_{2})=\emptyset .$ From this,  that $\langle h_{1},h_{2}\rangle $ is hyperbolic, and 
\cite[Proposition 3.63]{Splms10}, it follows that there exists an open neighborhood $V$ of $(h_{1},h_{2})$ in 
${\cal B}\cap {\cal H}$ 
such that 
for each $(g_{1},g_{2},q)\in V\times (0,1)$, we have that $\{ g_{1},g_{2}\} $ is mean stable and 
$\tau _{g_{1},g_{2},q}$ is mean stable.  
By \cite[Remark 5.11]{Scp}, taking a small open neighborhood $W$ of $p$, 
and shrinking $V$ if necessary, 
we obtain that there exists a number $\alpha \in (0,1)$ such that 
for each $(g_{1},g_{2},q)\in V\times W$, $T(g_{1},g_{2},q,\cdot )\in C^{\alpha }(\CCI ).$  

We now let $(g_{1},g_{2},q)\in V\times W.$ 
Since $\emptyset \neq \hat{K}(g_{1},g_{2})$ and 
$\langle g_{1},g_{2}\rangle (\hat{K}(g_{1},g_{2}))\subset \hat{K}(g_{1},g_{2})$,  
Zorn's lemma implies that there exists a minimal set $L_{0}$ for $(\langle g_{1},g_{2}\rangle ,\CCI )$ 
with $L_{0}\subset \hat{K}(g_{1},g_{2}).$ 
Let $L$ be a minimal set for $(\langle g_{1},g_{2}\rangle, \CCI )$ with $L\neq \{ \infty \} .$ 
Then $L\subset \hat{K}(g_{1},g_{2}).$ By Proposition~\ref{p:bdyosc}, we may assume that 
$K(h_{1})\subset K(h_{2})$. Then, by Proposition~\ref{p:bdyosc} again, we obtain that 
$J(h_{2})$ is a quasicircle and $h_{2}$ has an attracting fixed point $c$ in $K(h_{2}).$ 
Shrinking $V$ if necessary, we obtain that $J(g_{2})$ is a quasicircle and $g_{2}$ has an attracting fixed point 
$c'$ in $K(g_{2}).$ 
Moreover, since $\{ g_{1},g_{2}\} $ is mean stable, $L\subset F(g_{1},g_{2}).$ 
Therefore $L\subset \mbox{int}(K(g_{2})).$ It follows that for each $z\in L$, 
$g_{2}^{n}(z)\rightarrow c'$ as $n\rightarrow \infty .$ In particular, 
$c'\in L.$ Thus, there exists a unique minimal set $L_{g_{1},g_{2}}$ for 
$(\langle g_{1},g_{2}\rangle , \CCI )$ with $L_{g_{1},g_{2}}\subset \CC .$ 
Hence the set of all minimal sets for $(\langle g_{1},g_{2}\rangle ,\CCI )$ is 
$\{ \{ \infty \} , L_{g_{1},g_{2}}\} .$  
Moreover, from the above argument we have $L_{g_{1},g_{2}}\subset \mbox{int}(\hat{K}(g_{1},g_{2}))\subset F(g_{1},g_{2}).$   
By \cite[Theorem 3.15-7]{Splms10}, item (iv) of statement~\ref{spacemainth10} in Theorem~\ref{t:timportant} follows. 
Since $L_{g_{1},g_{2}}$ contains a fixed point $g_{2}$, \cite[Theorem 3.15-12]{Splms10} implies that 
 the number $r_{L}$ in \cite[Theorem 3.15-8]{Splms10} for $L=L_{g_{1},g_{2}}$ is equal to $1.$ 
Therefore, by \cite[Theorem 3.15-1, 13]{Splms10}, there exist   
two functions $\varphi _{1},\varphi _{2}\in C(\CCI ) $ with $M_{g_{1},g_{2},q}(\varphi _{i})=\varphi _{i}$ 
and a Borel probability measure $\nu =\nu _{g_{1},g_{2},q}$ on $L_{g_{1},g_{2}}$ with $M_{g_{1},g_{2},q}^{\ast }(\nu )=\nu $ 
such that for each $\varphi \in C(\CCI )$, 
$$M_{g_{1},g_{2},q}^{n}(\varphi )\rightarrow \varphi (\infty )\cdot \varphi _{1} +(\int \varphi d\nu )\cdot \varphi _{2}  
\mbox{ in } C(\CCI ) \mbox{ as } n\rightarrow \infty ,$$ 
supp$\,\nu =L_{g_{1},g_{2}}$, 
$\varphi _{1}(\infty )=1, \varphi _{1}|_{L_{g_{1},g_{2}}}\equiv 0$, 
$\varphi _{2}(\infty )=0$, and $\varphi _{2}|_{L_{g_{1},g_{2}}}\equiv 1$.  
Combining these with item (iv), we see that 
$\varphi _{1}(z)=T(g_{1},g_{2},q,z)$ and $\varphi _{2}(z)=1-T(g_{1},g_{2},q,z).$ 
From these arguments item (v) of statement~\ref{spacemainth10} in Theorem~\ref{t:timportant} follows.     
By \cite[Theorem 3.24]{Scp}, shrinking $V$ and $W$ if necessary, item (vi) of  statement~\ref{spacemainth10} in Theorem~\ref{t:timportant} 
holds.  By \cite[Theorem 3.32]{Scp}, item (vii) of  statement~\ref{spacemainth10} in Theorem~\ref{t:timportant} holds. 

 We now prove item (viii). Let $(g_{1},g_{2})\in V$ and let $G=\langle g_{1},g_{2}\rangle .$ 
 Since the statement in item (vii) holds for arbitrary $p\in (0,1)$, 
 we obtain that the function $q\mapsto T(h_{1},h_{2},q,z)$ is real-analytic on $(0,1)$ for any $z\in \CCI $,  
 that the function  
$(q,z)\mapsto (\partial ^{n}T/\partial q^{n})
(g_{1},g_{2},q,z)$ is continuous on $(0,1)\times \CCI $ for any $n\in \NN \cup \{ 0\} ,$ 
and that the function $z\mapsto (\partial ^{n}T/\partial q^{n})(g_{1},g_{2},q,z)$ is H\"{o}lder continuous on $\CCI $ 
for any $n\in \NN \cup \{ 0\} $ and any $q\in (0,1).$ 
Moreover, for any $n\in \NN \cup \{ 0\} $ and any $q\in (0,1)$, 
since $z\mapsto T(g_{1},g_{2},q,z)$ is locally constant on $F(G)$ (see \cite[Lemma 3.24]{Splms10}), 
it follows that  the function $z\mapsto (\partial ^{n}T/\partial q^{n})(g_{1},g_{2},q,z)$ is locally constant on $F(G).$ 
By \cite[Proposition 3.26]{Splms10},  for each $q\in (0,1)$, 
the function $z\mapsto T(g_{1},g_{2},q,z)$ is characterized by the unique element $\varphi \in C(\CCI )$
such that 
$M_{g_{1},g_{2},q}(\varphi )=\varphi ,\varphi |_{\hat{K}(G)}\equiv 0,  $ 
$ \varphi | _{ F_{\infty }(G)}\equiv 1$. 
For each $q\in (0,1)$ and for each $n\in \NN \cup \{ 0\} $, 
we set $\varphi _{n,q}(z)=(\partial ^{n}T/\partial q^{n})(g_{1},g_{2},q,z).$ 
Since $\varphi _{0,q}|_{F_{\infty }(G)}\equiv 1$ and  $\varphi _{0,q}|_{\hat{K}(G)}\equiv 0$, 
we have $\varphi _{n,q}|_{F_{\infty }(G)\cup \hat{K}(G)}\equiv 0$ for each $n\geq 1.$ 
By \cite[Theorem 3.32]{Scp}, the function $\varphi _{1,q}$ is characterized by 
the unique element $\varphi \in C(\CCI )$ such that 
$\varphi _{1,q}(z)=M_{g_{1},g_{2},q}(\varphi _{1,q})(z)+
(\varphi _{0,q}(g_{1}(z))-\varphi _{0,q}(g_{2}(z)))$, 
$\varphi _{1,q}|_{F_{\infty }(G)\cup \hat{K}(G)}\equiv 0.$  
Let $k\geq 0$ and suppose that  $\varphi _{k+1,q}(z)=M_{g_{1},g_{2},q}(\varphi _{k+1,q})(z)+
(k+1)(\varphi _{k,q}(g_{1}(z))-\varphi _{k,q}(g_{2}(z)))$. 
By taking the partial derivatives of both hand sides of this equation with respect to the parameter $q$, 
we obtain that $\varphi _{k+2,q}(z)=M_{g_{1},g_{2},q}(\varphi _{k+2,q})(z)+
(k+2)(\varphi _{k+1,q}(g_{1}(z))-\varphi _{k+1,q}(g_{2}(z)))$. 
Therefore for each $n\in \NN \cup \{ 0\} $, 
we have  $\varphi _{n+1,q}(z)=M_{g_{1},g_{2},q}(\varphi _{n+1,q})(z)+
(n+1)(\varphi _{n,q}(g_{1}(z))-\varphi _{n,q}(g_{2}(z)))$. 
Let $n\in \NN ,q\in (0,1)$ and 
let $\varphi \in C(\CCI ) $ be an element such that 
\begin{equation}
\label{eq:varphizm}
\varphi (z)=M_{g_{1},g_{2},q}(\varphi )(z)+
(n+1)(\varphi _{n,q}(g_{1}(z))-\varphi _{n,q}(g_{2}(z))) \mbox{ and } 
\varphi |_{F_{\infty }(G)\cup \hat{K}(G)}\equiv 0.
\end{equation}
We want to show that $\varphi =\varphi _{n+1,q}.$  
Let $\phi _{n}(z)=(n+1)(\varphi _{n,q}(g_{1}(z))-\varphi _{n,q}(g_{2}(z))).$ 
Then $\phi _{n}\in C^{\gamma }(\CCI )$ for some $\gamma \in (0,1).$ 
Since $\tau _{g_{1},g_{2},q}$ is mean stable, 
there exists a direct decomposition 
$C(\CCI )= U_{\tau _{g_{1},g_{2},q}}\oplus 
\{ \psi \in \CCI \mid M_{g_{1},g_{2},q}^{n}(\psi )\rightarrow 0 \mbox{ as }n\rightarrow \infty \} $ 
(see \cite[Theorem 3.15]{Splms10}).  
Let $\pi _{\tau _{g_{1},g_{2},q}}:C(\CCI )\rightarrow U_{\tau _{g_{1},g_{2},q}}$ be the 
projection map regarding the direct decomposition. 
Moreover, by \cite[Theorem 3.30]{Scp} and its proof, 
there exist constants $\zeta \in (0,\gamma ], \lambda \in (0,1), C>0$ such that 
 for each $\psi \in C^{\zeta }(\CCI )$, 
 $\| M_{g_{1},g_{2},q}^{k}(\psi -\pi _{\tau _{g_{1},g_{2},q}}(\psi ))\| _{\zeta }
 \leq C\lambda ^{k}\| \psi -\pi _{\tau _{g_{1},g_{2},q}}(\psi )\| _{\zeta }$ for each $k\in \NN .$ 
 By the definition of $\phi _{n}$, we have $\phi _{n}|_{\{ \infty \} \cup \hat{K}(G)}\equiv 0.$ 
 Therefore by \cite[Theorem 3.15-2]{Splms10}, we have $\pi _{\tau _{g_{1},g_{2},q}}(\phi _{n})=0.$ 
 It follows that 
 $\| M_{g_{1},g_{2},q}^{k}(\phi _{n})\| _{\zeta }\leq C\lambda ^{k}\| \phi _{n}\| _{\zeta }.$ 
By (\ref{eq:varphizm}), we have 
\begin{equation}
\label{eq:i-mg1}
(I-M_{g_{1},g_{2},q}^{k})(\varphi )=\sum _{j=0}^{k-1}M_{g_{1},g_{2},q}^{k}(\phi _{n}) 
\mbox{ for each }k\in \NN .
\end{equation} 
Moreover, by $\varphi |_{F_{\infty }(G)\cup \hat{K}(G)}\equiv 0$ and \cite[Theorem 3.15-2]{Splms10} we have 
$\pi _{\tau _{g_{1},g_{2},q}}(\varphi )=0 .$ 
Hence $M_{g_{1},g_{2},q}^{k}(\varphi )\rightarrow 0$ in $C(\CCI )$ as $k\rightarrow \infty .$ 
Therefore, letting $k\rightarrow \infty $ in (\ref{eq:i-mg1}), we obtain 
$\varphi =\sum _{j=0}^{\infty }M_{g_{1},g_{2},q}^{j}(\phi _{n})$ in $C(\CCI ).$ 
(In fact, this equation holds even in $C^{\zeta }(\CCI ).$)   
Thus, there exists a unique element $\varphi \in C(\CCI )$ which satisfies 
(\ref{eq:varphizm}). 
Hence we have proved item (viii).  
Thus we have proved statement~\ref{spacemainth10} in Theorem~\ref{t:timportant}. 
\end{proof}
\begin{thm}
\label{t:oscanal}
Let $(h_{1},h_{2})\in {\cal B}.$ Let $p\in (0,1).$ 
Suppose that $K(h_{1})\subset (\mbox{int}(K(h_{2}))).$  
Let $U:=(\mbox{int}(K(h_{2})))\setminus K(h_{1}).$ 
Suppose that $(h_{1},h_{2})$ satisfies the open set condition with $U.$ 
Then,  we have all of the following.
\begin{itemize}
\item[{\em (i)}] 
$J_{\ker }(h_{1},h_{2})=\emptyset .$ Moreover, 
$T(h_{1},h_{2},p,\cdot )$ is H\"{o}lder continuous on $\CCI $ and 
locally constant on $F(G).$  Moreover, 
$J(G)=\{ z_{0}\in \CCI \mid 
\mbox{for each neighborhood }V \mbox{of }z_{0}, T(h_{1},h_{2},p,\cdot )|_{V} 
\mbox{ is not constant} \} .$ 
\item[{\em (ii)}] 
There exists a unique minimal set $L$ for 
$(\langle h_{1},h_{2}\rangle , \hat{K}(h_{1},h_{2}))$ and the 
set of minimal sets for $(\langle h_{1},h_{2}\rangle , \CCI )$ is 
$\{ \{ \infty \}, L\} .$ 
\item[{\em (iii)}] For each $z\in \CCI $ there exists a Borel subset 
${\cal B}_{z}$ of $\{ h_{1},h_{2}\} ^{\NN }$ with 
$\tilde{\tau } _{h_{1},h_{2},p}({\cal B}_{z})=1$ such that for each 
$\gamma =(\gamma _{1},\gamma _{2},\ldots )\in {\cal B}_{z}$, we have 
$d(\gamma _{n}\cdots \gamma _{1}(z), \{ \infty \} \cup L) \rightarrow 0$ as $n\rightarrow \infty .$ 
\item[{\em (iv)}] 
There exists a unique $M_{\tau _{h_{1},h_{2},p}}^{\ast }$-invariant Borel probability measure 
$\nu $ on $\hat{K}(h_{1},h_{2})$ such that for each $\varphi \in C(\CCI )$, 
$$M_{h_{1},h_{2},p}^{n}(\varphi )(z)\rightarrow T(h_{1},h_{2},p,z)\cdot \varphi (\infty )
+(1-T(h_{1},h_{2},p,z))\cdot \int \varphi d\nu \mbox{ as }n\rightarrow \infty $$
uniformly on $\CCI .$

\item[{\em (v)}]   
For each $z\in \CCI $, the function  
$p\mapsto T(h_{1},h_{2},p,z) $ is real-analytic on   
$(0,1)$. Moreover,   
for each $n\in \NN \cup \{ 0\} $, the function  
$(p,z)\mapsto (\partial ^{n}T/\partial p^{n})
(h_{1},h_{2},p,z)$ is continuous on $(0,1)\times \CCI $.   
The function $z\mapsto T(h_{1},h_{2},p,z)$ is characterized by the unique element $\varphi \in C(\CCI )$
such that 
$M_{h_{1},h_{2},p}(\varphi )=\varphi ,\varphi |_{\hat{K}(G)}\equiv 0,  
\varphi |_{F_{\infty }(G)}\equiv 1$. 
Furthermore, inductively, 
for any $n\in \NN \cup \{ 0\} $, 
the function $z\mapsto (\partial ^{n+1}T/\partial p^{n+1})(h_{1},h_{2},p,z)$
is characterized by the unique element $\varphi \in C(\CCI )$ such that  \\ 
$\varphi (z)\equiv M_{h_{1},h_{2},p}(\varphi )(z)+
(n+1)\left((\partial ^{n}T/\partial p^{n})(h_{1},h_{2},p,h_{1}(z))-
(\partial ^{n}T/\partial p^{n})(h_{1},h_{2},p,h_{2}(z))\right),\\   
\varphi |_{\hat{K}(G)\cup F_{\infty }(G)}\equiv 0$.

Moreover, the function 
$z\mapsto (\partial ^{n+1}T/\partial p^{n+1})(h_{1},h_{2},p,z)$ 
is locally constant on $F(G).$ Moreover, setting $\psi _{n,h_{1},h_{2},p}(z)=(\partial ^{n}T/\partial p^{n})(h_{1},h_{2},p,z)$ for each 
$p\in (0,1), z\in \CCI, n\in \NN \cup \{ 0\} $, we have that 
$\psi _{n+1,h_{1},h_{2},p}=\sum _{j=0}^{\infty }M_{h_{1},h_{2},p}^{j}
((n+1)(\psi _{n,h_{1},h_{2},p}\circ h_{1}-\psi _{n,h_{1},h_{2},p}\circ h_{2}))$ 
in $(C(\CCI ), \| \cdot \| _{\infty })$ for each $n\in \NN \cup \{ 0\} .$   

\end{itemize} 
\end{thm}
\begin{proof}
Let $A$ be the connected component of $\mbox{int}(K(h_{2}))$ with $K(h_{1})\subset A.$ 
Since $K(h_{1})\subset \mbox{int}(K(h_{2}))$, we have 
$h_{2}^{-1}(K(h_{1}))\subset \mbox{int}(K(h_{2})).$ 
Since $(h_{1},h_{2})$ satisfies the open set condition with $U$, we have 
$h_{2}^{-1}(K(h_{1}))\supset K(h_{1}).$  Moreover, since $(h_{1},h_{2})\in {\cal B}$, 
we have that $h_{2}^{-1}(K(h_{1}))$ is connected. 
It follows that $h_{2}^{-1}(K(h_{1}))\subset A.$ Thus 
$h_{2}^{-1}(A)\subset A.$ It implies that $\mbox{int}(K(h_{2}))$ is 
connected. Since $h_{2}(K(h_{1}))\subset K(h_{1})\subset \mbox{int}(K(h_{2}))$, 
it follows that 
there exists an attracting fixed point $z_{0}$ of $h_{2}$ 
in $K(h_{1}).$  Since $(h_{1},h_{2})$ satisfies the open set condition with $U$, 
we have $h_{1}^{-1}(U)\cap h_{2}^{-1}(U)=\emptyset .$ 
Since $U=(\mbox{int}(K(h_{2})))\setminus K(h_{1})$ and 
$\mbox{int}(K(h_{2}))=U\cup K(h_{1})$, 
we obtain that 
$h_{2}^{-1}(U)\cap (h_{1}^{-1}(\mbox{int}(K(h_{2}))))=
h_{2}^{-1}(U)\cap (h_{1}^{-1}(U)\cup K(h_{1}))
\subset h_{2}^{-1}(U)\cap K(h_{1})\subset U\cap K(h_{1})=\emptyset .$ 
Therefore $\overline{h_{2}^{-1}(U)}\subset \CCI \setminus h_{1}^{-1}(\mbox{int}(K(h_{2}))).$ 
We now want to show that $z_{0}\in \mbox{int}(K(h_{1})).$ 
Suppose to the contrary that $z_{0}\in J(h_{1}).$ 
Then $z_{0}\in h_{2}^{-1}(z_{0})\subset h_{2}^{-1}(\overline{U})=\overline{h_{2}^{-1}(U)}.$ 
It implies that 
$z_{0}\in \CCI \setminus h_{1}^{-1}(\mbox{int}(K(h_{2})))
\subset \CCI \setminus h_{1}^{-1}(K(h_{1}))=\CCI \setminus K(h_{1})$, 
which contradicts $z_{0}\in J(h_{1}).$ Thus we must have that $z_{0}\in \mbox{int}(K(h_{1})).$ 

Since $h_{2}(K(h_{1}))\subset K(h_{1})$, setting $G=\langle h_{1},h_{2}\rangle $, we have 
$g(K(h_{1}))\subset K(h_{1}).$ Hence 
\begin{equation}
\label{eq:hatkgkh1}
\hat{K}(G)=K(h_{1}).
\end{equation}
Therefore $G(\mbox{int}(K(h_{1})))\subset \mbox{int}(K(h_{1}))$. 
Thus $\mbox{int}(K(h_{1}))\subset F(G).$ Since $z_{0}$ is an attracting fixed point of $h_{2}$ 
and it belongs to $\mbox{int}(K(h_{1}))\subset F(G)$, it follows that for each $z\in K(h_{1})$, 
there exists a number $n\in \NN $ such that $h_{2}^{n}(z)\in F(G).$ 
Moreover, if $z\in \CCI \setminus K(h_{1})$, there exists a number $m$ such that 
$h_{1}^{m}(z)\in F_{\infty }(G)\subset F(G).$ Hence, we obtain that 
$J_{\ker }(G)=\emptyset .$ Combining this with \cite[Theorem 3.29]{Scp} and \cite[Theorem 3.22]{Splms10},  
we obtain that the function $\psi _{p}:=T(h_{1},h_{2},p,\cdot )$ is H\"{o}lder continuous on $\CCI $ 
and $M_{h_{1},h_{2},p}(\psi _{p})=\psi _{p}.$ 
By \cite[Lemma 3.24]{Splms10}, $\psi _{p}$ is locally constant on $F(G).$ 
By Lemma~\ref{l:oscinte} and the arguments in the proof of Lemma~\ref{l:intT1}, 
it follows that 
$J(G)=\{ z_{0}\in \CCI \mid 
\mbox{for each neighborhood }V \mbox{ of }z_{0}, \psi _{p}|_{V} \mbox{ is not constant}\} .$ 
By using the arguments in the proof of Lemma~\ref{l:pfmainth10}, 
we can show that items (ii)--(iv) of our theorem hold. Moreover, 
by \cite[Proposition 3.26]{Splms10}, 
$\psi _{p}$ is characterized by the unique element $\psi \in C(\CCI )$ such that 
$\psi =M_{h_{1},h_{2},p}(\psi )=\psi , \psi _{F_{\infty }(G)}\equiv 1, \psi _{\hat{K}(G)}\equiv 0.$

Since $(h_{1},h_{2})$ satisfies the open set condition with $U$, 
we have $h_{1}^{-1}(\mbox{int}(K(h_{2})))\subset h_{1}^{-1}(U\cup K(h_{1}))\subset U\cup K(h_{1})=\mbox{int}(K(h_{2})).$ 
Hence $h_{1}^{-1}(K(h_{2}))\subset K(h_{2}).$ Therefore $h_{1}(F_{\infty }(h_{2}))\subset F_{\infty }(h_{2}).$ 
Thus $G(F_{\infty }(h_{2}))\subset F_{\infty }(h_{2}).$ It follows that 
\begin{equation}
\label{eq:finftygh2}
F_{\infty }(G)=F_{\infty }(h_{2}).
\end{equation}
Since $\psi _{p}$ is continuous on $\CCI $ and $M_{h_{1},h_{2},p}(\psi _{p})=\psi _{p}$, 
we obtain that 
\begin{equation}
\label{eq:psipzp}
\psi _{p}(z)=p\psi _{p}(h_{1}(z))+(1-p)\psi _{p}(h_{2}(z))\mbox{ for each } z\in \CCI, \ 
\psi _{p}|_{\hat{K}(G)}\equiv 0,\ \psi _{p}|_{\overline{F_{\infty }(G)}}\equiv 1.
\end{equation}
Let $\varphi \in C(\CCI )$ be an element such that $\varphi _{\hat{K}(G)}\equiv 0$, 
$\varphi |_{\overline{F_{\infty }(G)}}\equiv 1.$ 
By items (ii) and (iii) of our theorem, which have been already proved, 
we obtain that $T(h_{1},h_{2},p,z)=\lim _{n\rightarrow \infty }M_{h_{1},h_{2},p}^{n}(\varphi )(z)$ 
for each $z\in \CCI .$
Let $A:= \{ p\in \CC \mid |p|<1, |1-p|<1\} .$ 
For each $p\in A$ and for each $\psi \in C(\CCI )$, 
we set $M_{p}(\psi )(z)=p\psi (h_{1}(z))+(1-p)\psi (h_{2}(z)).$ 
For each $p\in A$, we set $p_{1}=p$ and $p_{2}=1-p.$ 
For each $n\in \NN $ and for each $w=(w_{1},\ldots ,w_{n})\in \{ 1,2\} ^{n}$, 
we set $h_{w}=h_{w_{n}}\cdots h_{w_{1}} $ and $p_{w}=p_{w_{n}}\cdots p_{w_{1}}.$  
Moreover, we set $B_{n,z}=\{ (w_{1},\ldots ,w_{n})\in \{ 1,2\} ^{n} \mid 
h_{w_{n}}\cdots h_{w_{1}}(z)\in \overline{F_{\infty }(G)}\} $ and 
$C_{n,z}=\{ (w_{1},\ldots ,w_{n})\in \{ 1,2\} ^{n}\mid 
h_{w_{n}}\cdots h_{w_{1}}(z)\in U\} .$ 
Furthermore, for each $w=(w_{1},\ldots ,w_{n})\in \{ 1,2\} ^{n}$ 
and for each $m\leq n$, we set $w|_{m}:= (w_{1},\ldots ,w_{m})\in \{ 1,2\} ^{m}.$ 
Then for each $p\in A$, for each $n\in \NN $ and for each $z\in \CCI $, we have 
\begin{eqnarray*}
& \ & |M_{p}^{n+1}(\varphi )(z)-M_{p}^{n}(\varphi )(z)|\\ 
& = & \left| \sum _{\omega \in B_{n+1,z}}p_{\omega }
 + \sum _{\omega \in C_{n+1,z}}p_{\omega }\varphi (h_{\omega }(z))
 -\sum _{\gamma \in B_{n,z}}p_{\gamma }
      -\sum _{\gamma \in C_{n,z}}p_{\gamma }\varphi (h_{\gamma }(z))\right|\\ 
& = & | \sum  _{\omega \in B_{n+1,z}, \omega |_{n}\in B_{n,z}}p_{\omega } 
      +\sum _{\omega \in B_{n+1,z}, \omega |_{n}\in C_{n,z}}p_{\omega }
      + \sum _{\omega \in C_{n+1,z}}p_{\omega }\varphi (h_{\omega }(z))\\ 
& \ & \ \ \ \ \ \ \ \ \ \ \ \ \ \ \ \ \ \ \ \ \ \ \ \ \ \ \ \ \ \ \ \ \ 
  - \sum _{\gamma \in B_{n,z}}p_{\gamma }-\sum _{\gamma \in C_{n,z}}p_{\gamma }\varphi (h_{\gamma }(z))| .         
\end{eqnarray*}
 Since $\sum _{\omega \in B_{n+1,z}, \omega |_{n}\in B_{n,z}}p_{\omega }=
 \sum _{\gamma \in B_{n,z}}p_{\gamma }$, 
 we obtain 
\begin{equation}
\label{eq:mpn1-mpn}
|M_{p}^{n+1}(\varphi )(z)-M_{p}^{n}(\varphi )(z)| 
=\left| \sum _{\omega \in B_{n+1,z}, \omega |_{n}\in C_{n,z}}p_{\omega }
+\sum _{\omega \in C_{n+1,z}}p_{\omega }\varphi (h_{\omega }(z))
-\sum _{\gamma \in C_{n,z}}p_{\gamma }\varphi (h_{\gamma }(z))\right| . 
\end{equation} 
 Let $K$ be a non-empty compact subset of $A.$ Then there exists a constant $c_{K}\in (0,1)$ 
 such that $K\subset \{ |z|<c_{K}\} .$ 
 By (\ref{eq:mpn1-mpn}) and that $(h_{1},h_{2})$ satisfies the open set condition with $U$, 
it follows that 
 $|M_{p}^{n+1}(\varphi )(z)-M_{p}^{n}(z)|
 \leq c_{K}^{n}+(c_{K}^{n+1}+c_{K}^{n})\sup _{a\in \CCI }|\varphi (a)|.$ 
 Therefore $S _{\infty, p}(z):=\lim _{n\rightarrow \infty }M_{p}^{n}(\varphi )(z)$ 
 exists in $A\times \CCI $, 
$(p, z)\rightarrow S _{\infty ,p}(z)$ is continuous on $A\times \CCI $, 
and 
\begin{equation} 
\label{eq:varphiphol}
\mbox{for each } z\in \CCI , \mbox{ the function }p\mapsto S _{\infty ,p}(z) \mbox{ is holomorphic in } A. 
\end{equation}  
Combining these with Cauchy's integral formula, we obtain that 
for each $n\in \NN $, 
$\frac{\partial ^{n}S _{\infty ,p}(z)}{\partial p^{n}}$ is continuous on $A\times \CCI .$ 
From these arguments, it follows that for each $z\in \CCI $, the function  
$p\mapsto T(h_{1},h_{2},p,z) $ is real-analytic on   
$(0,1)$ and   
for each $n\in \NN \cup \{ 0\} $, the function  
$(p,z)\mapsto (\partial ^{n}T/\partial p^{n})
(h_{1},h_{2},p,z)$ is continuous on $(0,1)\times \CCI $.   

Let $\psi _{n,p}(z)=\frac{\partial ^{n}T}{\partial p^{n}}(h_{1},h_{2},p,z)$ 
for each $n\in \NN \cup \{ 0\} $ and $z\in \CCI .$  
By taking the $n$-th derivatives of equations in (\ref{eq:psipzp}), 
we obtain that 
\begin{equation}
\label{eq:psin1p}
\psi _{n+1,p}(z)=M_{p}(\psi _{n+1,p})(z)+(n+1)(\psi _{n,p}(h_{1}(z))-\psi _{n,p}(h_{2}(z))), 
\psi _{n+1,p}|_{\overline{F_{\infty }(G)}\cup \hat{K}(G)}\equiv 0 
\end{equation}
for each $n\in \NN \cup \{ 0\} $, $p\in (0,1)$ and $z\in \CCI .$ 
Let $\phi _{n,p}(z):=(n+1)(\psi _{n,p}(h_{1}(z))-\psi _{n,p}(h_{2}(z))).$ 
Then $\phi _{n,p}|_{\overline{F_{\infty }(G)}\cup \hat{K}(G)}\equiv 0.$ 
Let $\psi \in C(\CCI )$ be an element such that 
$\psi _{\overline{F_{\infty }(G)}\cup \hat{K}(G)}\equiv 0.$ 
Then $M_{p}^{k}(\psi )(z)=\sum _{\omega \in \{ 1,2\} ^{k}}p_{\omega }\psi (h_{\omega }(z))
=\sum _{\omega \in C_{k,z}}p_{\omega }\psi (h_{\omega }(z))$. 
Since $(h_{1},h_{2})$ satisfies the open set condition with $U$, 
it follows that 
\begin{equation}
\label{eq:mpkpsi}
\| M_{p}^{k}(\psi )\| _{\infty }\leq (\max\{ p,1-p\} )^{k}\| \psi \| _{\infty }.
\end{equation}
If $\alpha \in C(\CCI )$ is an element such that 
\begin{equation}
\label{eq:alphamp}
\alpha (z)=M_{p}(\alpha )(z)+(n+1)(\psi _{n,p}(h_{1}(z))-\psi _{n,p}(h_{2}(z))), 
\alpha |_{\overline{F_{\infty }(G)}\cup \hat{K}(G)}\equiv 0 
\end{equation}
then we have  
\begin{equation}
\label{eq:i-mpkalpha}
(I-M_{p}^{k})(\alpha )(z)=\sum _{j=0}^{k}M_{p}^{j}(\phi _{n,p})(z) \mbox{ for each } z\in \CCI , k\in \NN . 
\end{equation}
By (\ref{eq:mpkpsi}), letting $k\rightarrow \infty $ in (\ref{eq:i-mpkalpha}) 
we obtain that $\alpha =\sum _{j=0}^{\infty }M_{p}^{j}(\phi _{n,p})$ in 
$(C(\CCI ), \| \cdot \| _{\infty }).$ 
Therefore, for each $n\in \NN \cup \{ 0\} $, the element 
$\psi  _{n+1,p}\in C(\CCI) $ is characterized by the unique element $\alpha \in C(\CCI )$ such that 
$\alpha (z)=M_{p}(\alpha )(z)+(n+1)(\psi _{n,p}(h_{1}(z))-\psi _{n,p}(h_{2}(z))), 
\alpha |_{\overline{F_{\infty }(G)}\cup \hat{K}(G)}\equiv 0 .$ 
Since $\psi _{p}=T(h_{1},h_{2},p,\cdot )$ is locally constant on $F(G)$, 
$\psi _{n+1,p}$ is locally constant on $F(G).$ 

Combining all of these arguments, we see that item (v) of our theorem holds. 


Thus we have proved Theorem~\ref{t:oscanal}. 
\end{proof} 
\begin{rem}
\label{r:toscanal}
Let $h_{1}\in {\cal P}$ and $d\in \NN $ with $d\geq 2$ and $p\in (0,1)$.  Suppose that $\langle h_{1}\rangle $ is postcritically bounded, 
int$(K(h_{1}))\neq\emptyset $ and $(\deg (h_{1}),d)\neq (2,2).$ Then, 
by Theorem~\ref{t:spgenex}, there exists an element $h_{2}\in {\cal P}$ with 
$\deg (h_{2})=d$ such that  $(h_{1},h_{2})\in 
((\partial {D})\cap 
{\cal C}\cap {\cal B})\setminus {\cal I}$
and such that $(h_{1},h_{2},p)$ satisfies the assumptions of Theorem~\ref{t:oscanal}. 
Note that if $h_{1}$ has a parabolic cycle or a  Siegel disk cycle, then the above $h_{2}$ can be 
taken so that $\{ h_{1},h_{2}\} $ is {\bf not} mean stable. In fact, in order to have such an $h_{2}$, 
we take a point $b\in \mbox{int}(K(h_{1}))$
so that $b$ belongs to the basin of parabolic cycle or the Siegel disk cycle of $h_{1}$ and 
then use the method in the proof of Lemma~\ref{l:pfexample}.  
In \cite{Scp}, the author showed several results on the random dynamical systems 
for which the associated kernel Julia sets are empty and all minimal sets are included in the 
Fatou sets. However, the author did not deal with the case for which some minimal sets meet 
the Julia sets. 
We remark that if $h_{1}\in {\cal P}$ has a parabolic cycle, 
$\langle h_{1}\rangle $ is postcritically bounded, and $d\geq 2$ satisfies $(\deg (h_{1}), d)\neq (2,2)$,   
then we can take 
an $h_{2}\in {\cal P}$ so that $\deg (h_{2})=d$, 
$(h_{1},h_{2})\in ((\partial {\cal D})\cap {\cal C}\cap {\cal B})\setminus 
{\cal I}$, 
$(h_{1},h_{2})$ satisfies the assumptions of Theorem~\ref{t:oscanal},  
$J_{\ker }(h_{1},h_{2} )=\emptyset $ and the bounded minimal set of $\langle h_{1},h_{2}\rangle $ 
meets $J(h_{1},h_{2})$ (we take a point $b$ in the basin of the parabolic cycle of $h_{1}$ and 
use the method of Lemma~\ref{l:pfexample}).  Thus,   
 Theorems~\ref{t:spgenex} and \ref{t:oscanal} deal with a new case.   
\end{rem} 
\begin{lem}
\label{l:pfmainth11}
Statement~\ref{spacemainth11} in Theorem~\ref{t:timportant} holds. 
\end{lem}
\begin{proof}
Let $(h_{1},h_{2})\in ({\cal D}\cap {\cal B}).$ 
Then by Proposition~\ref{p:dosc}, we may assume that $K(h_{1})\subset \mbox{int}(K(h_{2})).$ 
By Proposition~\ref{p:dosc} again, $(h_{1},h_{2})$ satisfies the open set condition with 
$U:=(\mbox{int}(K(h_{2}))^\setminus K(h_{1}).$ Thus, Theorem~\ref{t:oscanal} implies that 
statement~\ref{spacemainth11} in Theorem~\ref{t:timportant} holds. 
\end{proof}
We now give a result on the set of connected components of the Fatou set, which is shown 
by applying Theorem~\ref{t:timportant}-\ref{spacemainth10}-(i) and Theorem~\ref{spacemainth}-\ref{spacemainth4-2}.
\vspace{-2mm} 
\begin{thm}
\label{t:infcomp}
Let $(h_{1},h_{2})\in ((\partial {\cal C})\cap {\cal B}\cap {\cal H})\setminus {\cal I}.$ 
Then there exists a neighborhood $V$ of $(h_{1},h_{2})$ in ${\cal B}\cap {\cal H}$ with 
$V\cap \mbox{int}({\cal C})\cap {\cal B}\cap {\cal H}\neq \emptyset $ 
 such that 
for each $(g_{1},g_{2})\in V$, 
the set of connected components $U$ of $F(g_{1},g_{2})$ with
 $U\cap ( F_{\infty }(g_{1},g_{2})\cup \hat{K}(g_{1},g_{2}))=\emptyset $ is infinite. In particular, 
for each $(g_{1},g_{2})\in V$, there are infinitely many connected components of $F(g_{1},g_{2}).$ 
\end{thm}
\begin{proof}
By Theorem~\ref{t:timportant}-\ref{spacemainth10}-(i) and Theorem~\ref{spacemainth}-\ref{spacemainth4-2}, 
there exists a neighborhood $V$ of $(h_{1},h_{2})$ in ${\cal B}\cap {\cal H}$ such that 
for each $(g_{1},g_{2})\in V$, the set $\{ g_{1},g_{2}\} $ is mean stable and 
$\dim _{H}(J(g_{1},g_{2}))<2.$ 
Thus for each $(g_{1},g_{2})\in V$, we have  
$\hat{K}(g_{1},g_{2})\neq \emptyset, J_{\ker }(g_{1},g_{2})=\emptyset $ and 
$\mbox{int}(J(g_{1},g_{2}))=\emptyset .$ Combining this with \cite[Theorem 3.34]{Splms10}, 
we see that for each $(g_{1},g_{2})\in V$, 
 the set of connected components $U$ of $F(g_{1},g_{2})$ with
 $U\cap ( F_{\infty }(g_{1},g_{2})\cup \hat{K}(g_{1},g_{2}))=\emptyset $ is infinite. In particular, 
for each $(g_{1},g_{2})\in V$, there are infinitely many connected components of $F(g_{1},g_{2}).$   
Since $(g_{1},g_{2})\in (\partial {\cal C})\cap {\cal B}\cap {\cal H}$, 
Theorem~\ref{t:spacetopology}-\ref{spacemainth2} implies that 
$V\cap \mbox{int}({\cal C})\cap {\cal B}\cap {\cal H}\neq \emptyset $. 
Thus we have proved our theorem. 
\end{proof} 
\vspace{-6mm} 
\subsection{Proof of Theorem~\ref{t:spacend}}
\vspace{-2mm} 
In this subsection, we prove Theorems~\ref{t:spacend}. 
Also, we show some related results in which we do not assume hyperbolicity
 (Theorems~\ref{t:oscholgen}, \ref{t:oschg2}). 
We need the following. 
\vspace{-2.8mm} 
\begin{df}
Let $h=(h_{1},\ldots ,h_{m})\in (\Rat)^{m}.$ 
Let $p=(p_{1},p_{2},\ldots, p_{m})\in (0,1)^{m}$ with $\sum _{j=1}^{m}p_{j}=1.$ 
Let $\tilde{\mu }$ be an $\tilde{h}$-invariant Borel probability measure on $J(\tilde{h})$
 (i.e., $\tilde{\mu }(A)=\tilde{\mu }(\tilde{h}^{-1}(A))$ for each Borel subset $A$ of $J(\tilde{h})$).  
We set 
\vspace{-5mm} 
$$v(h,p,\tilde{\mu }):=\frac{-\int _{\Sigma _{m}\times \CCI }\log p_{w_{1}}d\tilde{\mu }(w,x)}
{\int _{\Sigma _{m}\times \CCI }\log \| D(h_{w_{1}})_{x}\| _{s}d\tilde{\mu}(w,x)}\in (0,\infty)$$
\vspace{-2mm} 
(when the denominator is positive).  
\end{df}
\begin{df}
\label{d:attminset}
Let $(h_{1},\ldots ,h_{m})\in (\Rat)^{m}$ and let $G=\langle h_{1},\ldots ,h_{m}\rangle .$ 
Let $L$ be a minimal set for $(G,\CCI ).$ 
We say that $L$ is attracting (for $(G,\CCI )$) if there exist non-empty open subsets $U,V $ of 
$F(G)$ and a number $n\in \NN $ such that both of the following hold. 
\begin{itemize}
\vspace{-2mm} 
\item $L\subset V\subset \overline{V}\subset U\subset \overline{U}\subset F(G), \sharp (\CCI \setminus V)\geq 3.$
\vspace{-2mm} 
\item For each $(w_{1},\ldots ,w_{n})\in \{ 1,\ldots ,m\} ^{n}$, 
$h_{w_{1}}\cdots h_{w_{n}}(\overline{U})\subset V.$
\end{itemize}
\end{df}
\vspace{-3mm} 
\begin{df}
Let $\tau \in {\frak M}_{1}(\Rat )$ and let 
$L$ be a minimal set for $(G_{\tau },\CCI )$. 
For each $z\in \CCI $, we set 
$T_{L,\tau }(z)=\tilde{\tau }(\{ \gamma =(\gamma _{1},\gamma _{2},\ldots )\in (\Rat)^{\NN }\mid 
d(\gamma _{n}\cdots \gamma_{1}(z),L)\rightarrow 0 \mbox{ as }n\rightarrow \infty \} ).$

\end{df}
\vspace{-4mm} 
\begin{df}
Let $\tau \in {\frak M}_{1}(\Rat ).$ 
We denote by $M_{\tau }^{\ast }:{\frak M}_{1}(\CCI )\rightarrow {\frak M}_{1}(\CCI )$ 
the dual map of $M_{\tau }:C(\CCI )\rightarrow C(\CCI )$, i.e., 
$\int _{\CCI }\varphi (z) d(M_{\tau }^{\ast }(\mu ))(z)=\int _{\CCI} M_{\tau }(\varphi )(z) d\mu (z)$ 
for each $\varphi \in C(\CCI)$ and $\mu \in {\frak M}_{1}(\CCI ).$  
By using the topological embedding $z\in \CCI \mapsto \delta _{z}\in {\frak M}_{1}(\CCI )$, 
we regard $\CCI $ as a compact subset of 
the compact metric space ${\frak M}_{1}(\CCI ).$ 
We denote by $F_{pt}^{0}(\tau )$ the set of points $z_{0}\in \CCI $ 
for which the sequence $\{ (M_{\tau }^{\ast })^{n}|_{\CCI }:\CCI \rightarrow {\frak M}_{1}(\CCI )\} _{n\in \NN} $ 
is equicontinuous at the one point $z_{0}.$  
\end{df}
\vspace{-4mm} 
\begin{lem}
\label{l:attminset}
Let $(h_{1},\ldots ,h_{m})\in (\emRat)^{m}$ and let $G=\langle h_{1},\ldots ,h_{m}\rangle .$ 
Let $L$ be an attracting minimal set for $(G,\CCI ).$ 
Let $p=(p_{1},\ldots ,p_{m})\in (0,1)^{m}$ with $\sum _{j=1}^{m}p_{j}=1$ and 
let $\tau =\sum _{j=1}p_{j}\delta _{h_{j}}.$  
Then $T_{L,\tau }$ is locally constant on $F(G)$ and $M_{\tau }(T_{L,\tau })=T_{L,\tau }.$  
Moreover, $T_{L,\tau }$ is continuous at every point of $F_{pt}^{0}(\tau ).$ 
\end{lem}
\begin{proof}
Let $U$ be the open set coming from Definition~\ref{d:attminset} for $L$. 
By \cite[Remark 3.5]{Scp}, 
for each $z\in U$ and for each $\gamma \in \Sigma _{m}$, 
we have $d(h_{\gamma _{n}}\circ \cdots \circ h_{\gamma _{1}}(z),L)\rightarrow 0$ as $n\rightarrow \infty $ and 
$\| D(h_{\gamma _{n}}\circ \cdots \circ h_{\gamma _{1}})_{z}\| _{s}\rightarrow 0$ as $n\rightarrow \infty .$ 
Hence, for each $\omega \in \Sigma _{m}$, for each connected component $W$ of $F(G)$, 
if $x$ is a point of $W$ and 
$d(h_{\omega _{n}}\circ \cdots \circ h_{\omega _{1}}(x), L)\rightarrow 0$ as $n\rightarrow \infty $, 
then for any $y\in W$, we have 
$d(h_{\omega _{n}}\circ \cdots \circ h_{\omega _{1}}(y), L)\rightarrow 0$ as $n\rightarrow \infty $. 
Therefore $T_{L,\tau }$ is constant on $W.$ Thus $T_{L,\tau }$ is locally constant on $F(G).$ 

Let $\epsilon >0$ be a small number such that $B(L,\epsilon )\subset U.$ 
Let $\varphi \in C(\CCI )$ be an element such that 
$\varphi |_{L}\equiv 1$ and $\varphi |_{\CCI \setminus B(L,\epsilon )}\equiv 0.$ 
Since for each $z\in U$ and for each $\gamma \in \Sigma _{m}$, 
we have $d(h_{\gamma _{n}}\circ \cdots \circ 
h_{\gamma _{1}}(z),L)\rightarrow 0$ as $n\rightarrow \infty $, it follows that 
$M_{\tau }^{n}(\varphi )(z)\rightarrow T_{L,\tau }(z)$ as $n\rightarrow \infty $ for each $z\in \CCI .$ 
Therefore by \cite[Lemma 4.2-2]{Splms10}, we obtain that 
$T_{L,\tau }$ is continuous at every point of $F_{pt}^{0}(\tau ).$   
Moreover, for each $z\in\CCI $, 
$M_{\tau}(T_{L,\tau })(z)=M_{\tau }(\lim _{n\rightarrow \infty }M_{\tau }^{n}(\varphi ))(z)=
T_{L,\tau }(z).$ 
Thus we have proved our lemma. 
\end{proof}

\begin{df}
Let $h=(h_{1},\ldots ,h_{m})\in (\Rat) ^{m}.$ 
Let $w=(w_{1},\ldots ,w_{n})\in \{ 1,\ldots ,m\} ^{n}.$ 
We set $h_{w}=h_{w_{n}}\circ \cdots \circ h_{w_{1}}.$ 
Moreover, for each $\gamma =(\gamma _{1}, \gamma _{2},\ldots )\in \Sigma _{m}$ and for each $n\in \NN $, 
we set $\gamma |_{n}=(\gamma _{1},\ldots ,\gamma _{n})\in \{ 1,\ldots ,m\} ^{n}.$ 
\end{df}
\begin{thm}
\label{t:oscholgen} Let $m\in \NN $ with $m\geq 2.$ 
Let $h=(h_{1},\ldots ,h_{m})\in (\emRat) ^{m}.$ Let $G=\langle h_{1},\ldots ,h_{m}\rangle .$ 
Suppose that $\sharp J(G)\geq 3.$ 
Let $p=(p_{1},\ldots ,p_{m})\in (0,1)^{m}$ with $\sum _{j=1}^{m}p_{j}=1$ and 
let $\tau =\sum _{j=1}^{m}p_{j}\delta _{h_{j}}.$  
Suppose that $h$ satisfies the open set condition with an open set $U.$ 
Let $\tilde{\mu }$ be an $\tilde{h}$-invariant ergodic Borel probability measure on $J(\tilde{h})$ 
such that $\int \log \| D(\gamma _{1})_{x}\| _{s}d\tilde{\mu }(\gamma ,x)>0.$  
Let $\mu := (\pi _{\CCI })_{\ast }(\tilde{\mu }).$ 
Suppose that $\mu (U\setminus P(G))>0.$ 
Let $L$ be an attracting minimal set for $(G,\CCI )$. 
Suppose that there exists a point $\xi \in U$ such that $T_{L, \tau }$ is not constant on any neighborhood of $\xi .$ 
Then for $\mu $-a.e. $z_{0}\in J(G)$, we have that 
$z_{0}\in F_{pt}^{0}(\tau )$, $T_{L,\tau }$ is continuous at $z_{0}$ 
and $\emHol (T_{L,\tau },z_{0})\leq v(h,p,\tilde{\mu }).$   
\end{thm}
\begin{proof}
We assume that there exists an attracting minimal set $L$ for $(G,\CCI ).$ 
Let $\tilde{U}:=\pi _{\CCI }^{-1}(U\setminus P(G)).$ 
Then $\tilde{h}^{-1}(\tilde{U})\subset \tilde{U}$ and 
$\tilde{\mu }(\tilde{U})>0.$ 
Hence 
$\tilde{\mu }(\cap _{n=0}^{\infty }\tilde{h}^{-n}(\tilde{U}))
=\lim _{n\rightarrow \infty }\tilde{\mu }(\tilde{h}^{-n}(\tilde{U}))=
\tilde{\mu }(\tilde{U})>0.$ 
Therefore by the ergodicity of $\tilde{\mu }$, we obtain 
$\tilde{\mu }(\cap _{n=0}^{\infty }\tilde{h}^{-n}(\tilde{U}))=1.$ 
Let $(\omega ,a)\in \mbox{supp}(\tilde{\mu })\cap 
\cap _{n=0}^{\infty }\tilde{h}^{-n}(\tilde{U})\cap J(\tilde{h})$ be a point. 
By Birkhoff's ergodic theorem, we have that 
for $\tilde{\mu }$-a.e. $(\gamma ,x)\in J(\tilde{h})$, there exists a strictly increasing 
sequence $\{ n_{j}\} $ in $\NN $ such that 
$\tilde{h}^{n_{j}}(\gamma ,x)\rightarrow (\omega ,a)$ as $j\rightarrow \infty .$ 
Let $\tilde{A}:=\{ (\gamma ,x)\in J(\tilde{h})\mid \exists \{ n_{j}\} \rightarrow \infty 
\mbox{ s.t. }  
\tilde{h}^{n_{j}}(\gamma ,x)\rightarrow (\omega ,a)\} .$ 
Then $\tilde{\mu }(\tilde{A})=1.$ 
We now prove the following claim. 

\noindent Claim. For each $b\in \pi _{\CCI }(\cap _{n=0}^{\infty }\tilde{h}^{-n}(\tilde{U}))\cap J(G)$, 
there exists a unique $\alpha \in \Sigma _{m}$ such that 
$b\in \cap _{j=1}^{\infty }h_{w_{1}}^{-1}\cdots h_{w_{j}}^{-1}(J(G)).$ 
 
To prove this claim, since $b\in \pi _{\CCI }(\cap _{n=0}^{\infty }\tilde{h}^{-n}(\tilde{U}))$, 
there exists an element $\alpha \in \Sigma _{m}$ such that 
$(\alpha ,b)\in \cap _{n=0}^{\infty }\tilde{h}^{-n}(\tilde{U}).$ 
Therefore $b\in h_{\alpha |_{n}}^{-1}(U)$ for each $n.$ 
Moreover, since $b\in J(G)=\pi _{\CCI }(J(\tilde{h}))$ (see \cite[Lemma 3.5]{SdpbpIII}),  
there exists an element $\alpha '\in \Sigma _{m}$ such that 
$(\alpha ', b)\in J(\tilde{h}).$ 
Then $h_{\alpha '|_{n}}(b)\in J(G)$ for each $n.$ Since $J(G)\subset \overline{U}$, 
we obtain $b\in \overline{h_{\alpha '|_{n}}^{-1}(U)}.$ 
Hence $h_{\alpha |_{n}}^{-1}(U)\cap \overline{h_{\alpha '|_{n}}^{-1}(U)}\neq \emptyset .$ 
Therefore $h_{\alpha |_{n}}^{-1}(U)\cap h_{\alpha '|_{n}}^{-1}(U)\neq \emptyset .$ 
Since $h$ satisfies the open set condition with $U$, it follows that 
$\alpha |_{n}=\alpha '|_{n}$ for each $n\in \NN .$ Hence $\alpha =\alpha '.$ 
Therefore $b\in \cap _{n=0}^{\infty }h_{\alpha |_{n }}^{-1}(J(G)).$ 
Let $\beta \in \Sigma _{m}$ be an element such that 
$b\in \cap _{n=0}^{\infty }h_{\beta |_{n}}^{-1}(J(G)).$ 
Then $b\in h_{\alpha |_{n}}^{-1}(U)\cap h_{\beta |_{n}}^{-1}(J(G))
=h_{\alpha |_{n}}^{-1}(U)\cap \overline{h_{\beta |_{n}}^{-1}(U)} $ for each $n\in \NN .$ 
 Since $h$ satisfies the open set condition with $U$, 
 we obtain that $\alpha |_{n}=\beta |_{n}$ for each $n\in \NN .$ 
 Hence $\alpha =\beta .$ Therefore our claim holds. 
 
By the above claim and \cite[Lemma 4.3]{Splms10}, 
we obtain that 
$\pi _{\CCI }(\cap _{n=0}^{\infty }\tilde{h}^{-n}(\tilde{U}))\cap J(G)\subset F_{pt}^{0}(\tau ).$ 
By Lemma~\ref{l:attminset}, 
for each point 
$z_{0}\in \pi _{\CCI }(\cap _{n=0}^{\infty }\tilde{h}^{-n}(\tilde{U}))\cap J(G)$, 
the function $T_{L,\tau }$ is continuous at $z_{0}.$ 
Since $a\in \pi _{\CCI }(\cap _{n=0}^{\infty }\tilde{h}^{-n}(\tilde{U}))\cap J(G)$, 
it follows that $T_{L,\tau }$ is continuous at $a.$ 
Moreover, let $\{ K_{n}\} $ be a sequence of compact subsets of $\cap _{n=0}^{\infty }\tilde{h}^{-n}(\tilde{U})$ 
such that $\tilde{\mu }((\cap _{n=0}^{\infty }\tilde{h}^{-n}(\tilde{U}))\setminus \cup _{n=1}^{\infty }K_{n})=0.$ 
Then $B:=\pi _{\CCI }(\cup _{n=1}^{\infty }K_{n})\cap J(G)$ is Borel measurable and 
$\mu (B)=1.$ 
Thus we obtain that for $\mu $-a.e.$z_{0}\in J(G)$, we have 
$z_{0}\in F_{pt}^{0}(\tau )$ and 
$T_{L,\tau }$ is continuous at $z_{0}.$ 

We now want to prove that for $\mu $-a.e. $z_{0}\in J(G)$, 
$\Hol(T_{L,\tau },z_{0})\leq v(h,p,\tilde{\mu }).$ 
We have 
$v(h,p,\tilde{\mu })<\infty $, i.e., $\int |\log \| D(\gamma _{1})_{x}\| _{s}|d\tilde{\mu }(\gamma ,x)<\infty .$  
For each $k\in \NN $, let $\theta _{k}=v(h,p,\tilde{\mu })+\frac{1}{k}.$ 
Then \\ 
$\int |\log (p_{\gamma _{1}}\| D(h_{\gamma _{1}})_{x}\| _{s}^{\theta _{k}})| d\tilde{\mu }(\gamma ,x) <\infty $. 
Let $$\tilde{A}_{k}:= \{ (\gamma ,x)\in J(\tilde{h})\mid 
\frac{1}{n}\log (p_{\gamma _{n}}\cdots p_{\gamma _{1}}
\| D(h_{\gamma |_{n}})_{x}\| _{s}^{\theta _{k}})\rightarrow 
\int \log (p_{\gamma _{1}}\| D(h_{\gamma _{1}})_{x}\| _{s}^{\theta _{k}}) d\tilde{\mu }(\gamma ,x) \mbox{ as }n\rightarrow \infty \} $$ for each $k\in \NN .$ 
By Birkhoff's ergodic theorem, we have $\tilde{\mu }(\tilde{A}_{k})=1.$ 
Moreover, since $$\int \log (p_{\gamma _{1}}\| D(h_{\gamma _{1}})_{x}\| _{s}^{\theta _{k}}) d\tilde{\mu }(\gamma ,x) >0,$$  
we obtain that for each $(\gamma ,x)\in \tilde{A}_{k}$, 
$p_{\gamma _{n}}\cdots p_{\gamma _{1}}
\| D(h_{\gamma |_{n}})_{x}\| _{s}^{\theta _{k}}\rightarrow \infty $ as 
$n\rightarrow \infty .$  
 Let $(\gamma ,x)\in \tilde{A}\cap \tilde{A}_{k}.$ 
Then there exists 
a strictly increasing sequence $\{ n_{j}\} $ in $\NN $ 
such that $h_{\gamma |_{n_{j}}}(x)\rightarrow a$ as $j\rightarrow \infty .$ 
Let $W$ be a small open disk neighborhood of $a$ with $\overline{W}\subset U\setminus P(G)$. 
We may assume that $h_{\gamma |_{n_{j}}}(x)\in W$ for each $j.$  
Let $\zeta _{j}:W \rightarrow \CCI $ be the inverse branch of 
$h_{\gamma |_{n_{j}}}$ with $\zeta _{j}(h_{\gamma |_{n_{j}}}(x))=x.$ 
Then $\zeta _{j}(W)\subset h_{\gamma |_{n_{j}}}^{-1}(W)
\subset h_{\gamma |_{n_{j}}}^{-1}(U).$ 
Since $h$ satisfies the open set condition with $U$ and $J(G)\subset \overline{U}$, 
we obtain that for each $(\rho _{1},\ldots ,\rho _{n_{j}})\in \{ 1,\ldots ,m\} ^{n_{j}}$ 
with $(\gamma _{1},\ldots ,\gamma _{n_{j}})\neq (\rho _{1},\ldots ,\rho _{n_{j}})$, 
we have
$\zeta _{j}(W)\cap (h_{\rho _{n_{j}}}\cdots h_{\rho _{1}})^{-1}(J(G))=\emptyset .$ 
In particular, $h_{\rho _{n_{j}}}\cdots h_{\rho _{1}}(\zeta _{j}(W))$ 
is included in $F(G).$ Since  $h_{\rho _{n_{j}}}\cdots h_{\rho _{1}}(\zeta _{j}(W))$ 
is connected, it is included in a connected component of $F(G).$ 
Since $T_{L,\tau }$ is locally constant on $F(G)$ and 
$M_{\tau }(T_{L,\tau })=T_{L,\tau }$ (see Lemma~\ref{l:attminset}), it follows that 
for each $y\in \zeta _{j}(W)$, 
we have 
$$|T_{L,\tau }(x)-T_{L,\tau }(y)|=p_{\gamma _{n_{j}}}\cdots p_{1}
|T_{L,\tau }(h_{\gamma |_{n_{j}}}(x))-T_{L,\tau }(h_{\gamma |_{n_{j}}}(y))|.$$ 
Since we are assuming that there exists a point $\xi \in U$ such that 
$T_{L,\tau }$ is not constant in any neighborhood of $\xi $, 
Lemma~\ref{l:oscncp} implies that 
$T_{L,\tau }$ is not constant in any neighborhood of $a.$ 
Therefore there exists a point $b\in W$ such that 
$T_{L,\tau }(a)\neq T_{L,\tau }(b).$ 
Let $\eta =|T_{L,\tau }(a)-T_{L,\tau }(b)|>0.$ 
Let $c_{j}:=\zeta _{j}(b)\in\zeta _{j}(W).$ 
Since $T_{L,\tau }$ is continuous at $a$, 
for each large $j$, we have 
$$|T_{L,\tau }(h_{\gamma |_{n_{j}}}(x))-T_{L,\tau }(h_{\gamma |_{n_{j}}}(c_{j}))|
=|T_{L,\tau }(h_{\gamma |_{n_{j}}}(x))-T_{L,\tau }(b)|\geq \frac{\eta }{2}.$$ 
Therefore for each large $j$, we have 
$|T_{L,\tau }(x)-T_{L,\tau }(c_{j})|\geq p_{\gamma _{n_{j}}}\cdots p_{1}\frac{\eta }{2}.$ 
By the Koebe distortion theorem, there exists a constant $C>0$ such that 
$c_{j}\in B(x,C\| D(h_{\gamma |_{n_{j}}})_{x}\| _{s}^{-1})$ for each $j.$ 
Then 
$$\sup _{y\in B(x,C\|D(h_{\gamma |_{n_{j}}})_{x}\|_{s}^{-1})}
\frac{|T_{L,\tau }(x)-T_{L,\tau }(y)|}{C^{\theta _{k}}|\|D(h_{\gamma |_{n_{j}}})_{x}\|_{s}^{-\theta _{k}}}
\geq \frac{p_{\gamma _{n_{j}}}\cdots p_{1}}
{C^{\theta _{k}}\|D(h_{\gamma |_{n_{j}}})_{x}\| _{s}^{-\theta_{k}}}\frac{\eta }{2}
\rightarrow \infty  \mbox{ as } j\rightarrow \infty .$$   
Therefore 
$\limsup _{y\rightarrow x,y\neq x}\frac{|T_{L,\tau }(x)-T_{L,\tau }(y)|}{d(y,x)^{\theta _{k}}}=\infty .$  
Thus $\Hol(T_{L,\tau },x)\leq \theta _{k}.$ 
Hence for each $(\gamma ,x)\in \tilde{A}\cap \cap _{k=1}^{\infty }\tilde{A}_{k}$, 
$\Hol(T_{L,\tau },x)\leq v(h,p,\tilde{\mu }).$    
Let $\{ E_{n}\} $ be a sequence of compact subsets of $\tilde{A}\cap \cap _{k=1}^{\infty }\tilde{A}_{k}$ 
such that $\tilde{\mu }((\tilde{A}\cap \cap _{k=1}^{\infty }\tilde{A}_{k})\setminus \cup _{n=1}^{\infty }E_{n})=0.$ 
Then $D:=\pi _{\CCI }(\cup _{n=1}^{\infty }E_{n})$ is Borel measurable and 
$\mu (D)=1.$ Moreover, for each $x\in D$, $\Hol (T_{L,\tau },x)\leq v(h,p,\tilde{\mu }).$ 
Thus we have proved Theorem~\ref{t:oscholgen}. 
\end{proof}
\begin{lem}
\label{l:dbhqhol}
Let $h=(h_{1},h_{2})\in (\overline{{\cal D}}\cap {\cal B}\cap {\cal H})\setminus {\cal Q}.$ 
Let $G=\langle h_{1},h_{2}\rangle .$ 
Then $(h_{1},h_{2})$ satisfies the open set condition with an open set $U$ for which 
the following (i) and (ii) holds. 
(i) $(U\cap J(G))\setminus P(G)\neq \emptyset .$ (ii) There exists a point $\xi \in U$ such that 
$T(h_{1},h_{2},p,\cdot )$ is not constant in any neighborhood of $\xi .$ 

Moreover, if $\tilde{\mu }$ is an $\tilde{h}$-invariant ergodic Borel probability measure on $J(\tilde{h})$ with 
supp$\tilde{\mu }=J(\tilde{h})$, then for each $p\in (0,1)$, setting 
$\mu =(\pi _{\CCI })_{\ast }(\tilde{\mu })$ and 
$\tau =p\delta _{h_{1}}+(1-p)\delta _{h_{2}}$, 
 for $\mu $-a.e.$z_{0}\in J(G)$, we have $z_{0}\in F_{pt}^{0}(\tau )$, 
 $T(h_{1},h_{2},p,\cdot )$ is continuous at $z_{0}$, and $\emHol (T(h_{1},h_{2},p,\cdot ),z_{0})\leq u(h_{1},h_{2},p,\tilde{\mu }).$  
\end{lem}
\begin{proof}
By Proposition~\ref{p:bdyosc}, we may assume that $K(h_{1})\subset K(h_{2})$. 
By Proposition~\ref{p:bdyosc} again, we obtain that $(h_{1},h_{2})$ satisfies the open set condition with 
$U:=(\mbox{int}(K(h_{2})))\setminus K(h_{1}).$ 
Let $a\in J(h_{2})\setminus J(h_{1})$ and let $b\in J(h_{1})\setminus J(h_{2}).$ 
Then there exists a sequence $\{ a_{j}\} $ of points with $a_{j}\in h_{1}^{-n_{j}}(a)$ for some $n_{j}\in \NN $, 
such that $a_{j}\rightarrow b$ as $j\rightarrow \infty .$ 
Then, for a large $j$, we have $a_{j}\in (\mbox{int}(K(h_{2})))\setminus K(h_{1})=U.$ 
Since $a_{j}\in J(G)$, it follows that $U\cap J(G)\neq \emptyset .$ 
Moreover, by Lemma~\ref{l:pfmainth8}, for each $z_{0}\in J(G)$, 
the function 
$T(h_{1},h_{2},p,\cdot )$ is not constant in any neighborhood of $z_{0}.$ 
Hence there exists a point $\xi \in U$ such that $T(h_{1},h_{2},p,\cdot )$ is not 
constant in any neighborhood of $\xi .$ 
Furthermore, since $G$ is hyperbolic and $U\cap J(G)\neq \emptyset $, we obtain 
$(U\cap J(G))\setminus P(G)\neq \emptyset .$ 

 Let $\tilde{\mu }$ be an $\tilde{h}$-invariant ergodic Borel probability measure on $J(\tilde{h})$ with 
supp$\,\tilde{\mu }=J(\tilde{h})$. Let $p\in (0,1)$. We set 
$\mu :=(\pi _{\CCI })_{\ast }(\tilde{\mu })$ and 
$\tau :=p\delta _{h_{1}}+(1-p)\delta _{h_{2}}$. 
Since supp$\,\mu =\pi _{\CCI }(J(\tilde{h}))=J(G)$ (see \cite[Lemma 3.5]{SdpbpIII}),  
we obtain that $\mu (U\setminus P(G))>0.$ 
By Theorem~\ref{t:oscholgen},  it follows that 
for $\mu $-a.e.$z_{0}\in J(G)$, we have $z_{0}\in F_{pt}^{0}(\tau )$, 
 $T(h_{1},h_{2},p,\cdot )$ is continuous at $z_{0}$, and $\Hol (T(h_{1},h_{2},p,\cdot ),z_{0})\leq u(h_{1},h_{2},p,\tilde{\mu }).$  
Thus we have proved our lemma. 
\end{proof}
\vspace{-2mm} 
We now prove Theorem~\ref{t:spacend}. 
\vspace{-2mm} 
\begin{lem}
\label{l:pfnondiff}
Statement~\ref{nondiff} in Theorem~\ref{t:spacend} holds. 
\end{lem}
\begin{proof}
Let $h=(h_{1},h_{2})\in \overline{{\cal D}}\cap {\cal B}\cap {\cal H}$ and let $0<p<1.$ 
By \cite[Theorem 4.3]{S3}, supp$\,\lambda _{h_{1},h_{2},p}=J(G).$ 
Let $p_{1}=p, p_{2}=1-p.$ 
Since $\pi _{\ast }(\tilde{\lambda }_{h_{1},h_{2},p})$ is equal to the Bernoulli measure 
$\otimes _{n=1}^{\infty }(\sum _{i=1}^{2}p_{i}\delta _{i})$ on $\Sigma _{2}$, 
we have 
$\int _{\Sigma _{2}\times \CCI }-\log p_{\gamma _{1}}d\tilde{\lambda }_{h_{1},h_{2},p}(\gamma ,x)=
-\sum _{i=1}^{2}p_{i}\log p_{i}.$ 
Moreover, by \cite[Lemma 5.52]{Splms10}, 
$\int _{\Sigma _{2}\times \CCI }\log \| D(h_{\gamma _{1}})_{x}\| _{s}d\tilde{\lambda }_{h_{1},h_{2},p}(\gamma ,x)=
\sum _{i=1}^{2}p_{i}\log \deg (h_{i}).$ 
Therefore we get that $u(h_{1},h_{2},p,\tilde{\lambda }_{h_{1},h_{2},p})=$ $
\frac{-\sum _{i=1}^{2}p_{i}\log p_{i}}{\sum _{i=1}^{2}p_{i}\log (\deg (h_{i}))}.$ 
Since $\sum _{i=1}^{2}-p_{i}\log p_{i}\leq \log 2$ and $(\deg (h_{1}),\deg (h_{2}))\neq (2,2)$ (see Lemma~\ref{l:d1d2not22}), 
we obtain $\frac{-\sum _{i=1}^{2}p_{i}\log p_{i}}{\sum _{i=1}^{2}p_{i}\log (\deg (h_{i}))}<1.$ 

 If $(h_{1},h_{2})\in {\cal Q}$, then 
 $J(G)=J(h_{1})=J(h_{2})$, $F_{\infty }(G)=F_{\infty }(h_{1})=F_{\infty }(h_{2})$, 
 and $\hat{K}(G)=K(h_{1})=K(h_{2}).$ Since $T(h_{1},h_{2},p,\cdot )|_{F_{\infty }(G)}\equiv 1$ (see \cite[Lema 5.24]{Splms10}) 
and $T(h_{1},h_{2},p,\cdot )|_{\hat{K}(G)}\equiv 0$, we obtain that 
 for each $z_{0}\in J(G)$, the function $T(h_{1},h_{2},p,\cdot )$ is not continuous at $z_{0}.$ 
In particular, for each $z_{0}\in J(G)$, $0=\Hol(T(h_{1},h_{2},p,\cdot ),z_{0})
\leq u(h_{1},h_{2},p,\tilde{\lambda }_{h_{1},h_{2},p}).$

We now suppose that $(h_{1},h_{2})\in (\overline{{\cal D}}\cap {\cal B}\cap {\cal H})\setminus {\cal Q}.$ 
By \cite[Theorem 4.3]{S3}, supp$\tilde{\lambda }_{h_{1},h_{2},p}=J(\tilde{h}).$ 
Hence, Lemma~\ref{l:dbhqhol} implies that 
 for $\lambda _{h_{1},h_{2},p}$-a.e. $z_{0}\in J(G)$, 
we have that the function $T(h_{1},h_{2},p,\cdot ):\CCI \rightarrow \RR $ is continuous at $z_{0}$ 
and $\Hol (T(h_{1},h_{2},p,\cdot ), z_{0})\leq u(h_{1},h_{2},p,\tilde{\lambda }_{h_{1},h_{2},p}).$ 

We now suppose that $(h_{1},h_{2})\in {\cal D}\cap {\cal B}\cap {\cal H}.$ 
Then by Lemma~\ref{l:bdjconti}, $h_{1}^{-1}(J(G))\cap h_{2}^{-1}(J(G))=\emptyset .$ 
By \cite[Theorem 3.82]{Splms10}, 
it follows that for $\lambda _{h_{1},h_{2},p}$-a.e. $z_{0}\in J(G)$, 
$\Hol(T(h_{1},h_{2},p,\cdot ),z_{0})=u(h_{1},h_{2},p,\tilde{\lambda }_{h_{1},h_{2},p}).$ 
Thus statement~\ref{nondiff} in Theorem~\ref{t:spacend} holds.  
\end{proof}
\begin{lem}
\label{l:pfnonorable}
Statement~\ref{nonorable} in Theorem~\ref{t:spacend} holds. 
\end{lem}
\begin{proof}
Let $h=(h_{1},h_{2})\in (\overline{{\cal D}}\cap {\cal B}\cap {\cal H})\setminus {\cal Q}.$ Then 
by Lemma~\ref{l:pfmainthosc}, $(h_{1},h_{2})$ satisfies the open set condition. 
By \cite{S6}, it follows that item (i) of statement~\ref{nonorable} holds. 
Moreover, we have supp$\,\nu =J(\tilde{h}).$ Hence supp$\,\eta =J(\tilde{h}).$ 
By Lemma~\ref{l:dbhqhol} and item (i) of statement~\ref{nonorable}, 
it follows that for almost every  $z_{0}\in J(G)$ with respect to 
$H^{v}$,  the function $T(h_{1},h_{2},p,\cdot ):\CCI \rightarrow [0,1] $ is continuous at $z_{0}$ and 
$\mbox{H\"{o}l}(T(h_{1},h_{2},p,\cdot ),z_{0})\leq u(h_{1},h_{2},p,\eta )$.

 We now suppose that $(h_{1},h_{2})\in {\cal D}\cap {\cal B}\cap {\cal H}$. Then 
 by Lemma~\ref{l:bdjconti}, we have $h_{1}^{-1}(J(G))\cap h_{2}^{-1}(J(G))=\emptyset .$ 
 By \cite[Theorem 3.84]{Splms10}, it follows that 
 for almest every $z_{0}\in J(G)$ with respect to $H^{v}$, 
 $\mbox{H\"{o}l}(T(h_{1},h_{2},p,\cdot ),z_{0})= u(h_{1},h_{2},p,\eta )$.
 
 Thus statement~\ref{nonorable} in Theorem~\ref{t:spacend} holds.
\end{proof}
We now show the following result which is proved by using Theorem~\ref{t:oscholgen}.
\vspace{-2mm} 
\begin{thm}
\label{t:oschg2}
Let $(h_{1},h_{2})\in {\cal B}$ with $(\deg (h_{1}),\deg (h_{2}))\neq (2,2).$ 
Let $G=\langle h_{1},h_{2}\rangle. $ 
Let $p\in (0,1).$ 
Suppose that $K(h_{1})\subset (\mbox{{\em int}}(K(h_{2}))).$  
Let $U:=(\mbox{{\em int}}(K(h_{2})))\setminus K(h_{1}).$ 
Suppose that $(h_{1},h_{2})$ satisfies the open set condition with $U.$ 
Then,  
 supp $\lambda  _{h_{1},h_{2},p}=J(G)$, $\lambda _{h_{1},h_{2},p}$ is non-atomic, 
and for almost every point $z_{0}\in J(G)$ with respect to  
 $\lambda _{h_{1},h_{2},p}$, 
\begin{equation} 
\label{eq:holgth1}
\mbox{{\em H\"{o}l}}(T(h_{1},h_{2},p,\cdot ),z_{0})\leq 
u(h_{1},h_{2},p,\tilde{\lambda }_{h_{1},h_{2},p})
=-\frac{p\log p+(1-p)\log (1-p)}
{p\log (\deg (h_{1}))+(1-p)\log (\deg (h_{2}))}
<1
\end{equation}  
and $T(h_{1},h_{2},p,\cdot )$ is not differentiable at $z_{0}$.   
In particular, there exists an uncountable dense subset $A$ of $J(G)$ such that 
at every point of $A$, the function 
$T(h_{1},h_{2},p,\cdot )$ is not differentiable. 
Moreover, for each $\alpha \in (u(h_{1},h_{2},p, \tilde{\lambda}_{h_{1},h_{2},p}), 1)$ 
and for each $\varphi \in C^{\alpha}(\CCI )$ such that $\varphi (\infty )=1$ and $\varphi |_{\hat{K}(G)}\equiv 0$,  
we have $\| M_{h_{1},h_{2},p}^{n}(\varphi )\| _{\alpha }\rightarrow \infty $ as $n\rightarrow \infty .$ 

\end{thm}
\begin{proof}
From our assumption, we have $J(h_{1})\cap J(h_{2})=\emptyset .$ 
Also, by Lemma~\ref{l:oscinte}, we have $h_{1}^{-1}(K(h_{2}))\subsetneqq h_{2}^{-1}(K(h_{1})).$ 
By Theorem~\ref{t:oscanal} (i), we have that 
for each open set $W$ in $\CCI $ with $J(G)\cap W\neq \emptyset $, 
$T(h_{1},h_{2},p,\cdot )$ is not constant on $W.$ 
Let $a\in J(h_{2})$ and let $b\in J(h_{1}).$ 
Then there exists a sequence $\{ a_{j}\} $ of points with $a_{j}\in h_{1}^{-n_{j}}(a)$ for some $n_{j}\in \NN $, 
such that $a_{j}\rightarrow b$ as $j\rightarrow \infty .$ 
Then, for a large $j$, we have $a_{j}\in (\mbox{int}(K(h_{2})))\setminus K(h_{1})=U.$ 
Since $a_{j}\in J(G)$ and $U\subset (\CC \setminus \hat{K(h_{1})})
\subset \CC \setminus P^{\ast }(G)$, it follows that $U\cap J(G)\cap (\CCI \setminus P(G))
\neq \emptyset .$ Also, by \cite[Theorem 4.3]{S3}, supp$\tilde{\lambda }_{h_{1},h_{2},p}=J(G).$ 
Hence $\lambda _{h_{1},h_{2},p}(U\setminus P(G))>0.$ 
Moreover, by \cite[Lemma 5.52]{Splms10}, we have 
$\int \log \| D(\gamma _{1})_{x}\| _{s}d\tilde{\lambda }_{h_{1},h_{2},p}(\gamma ,x)
=p\log \deg (h_{1})+(1-p)\log \deg (h_{2})>0.$ 
Combining these with Theorem~\ref{t:oscholgen}, 
we see that for $\lambda _{h_{1},h_{2},p}$-a.e.$z_{0}\in J(G)$, 
\vspace{-3mm} 
$$\Hol(T(h_{1},h_{2},p,\cdot ), z_{0})\leq u(h_{1},h_{2},p, \tilde{\lambda }_{h_{1},h_{2},p})
=-\frac{p\log p+(1-p)\log (1-p)}{p\log (\deg (h_{1}))+(1-p)\log (\deg (h_{2}))}.$$ 
Since $-p\log p-(1-p)\log (1-p)\leq \log 2$ and 
$(\deg (h_{1}),\deg (h_{2}))\neq (2,2)$, we have that \\ 
$-\frac{p\log p+(1-p)\log (1-p)}{p\log (\deg (h_{1}))+(1-p)\log (\deg (h_{2}))}<1.$ 
In particular, for  $\lambda _{h_{1},h_{2},p}$-a.e.$z_{0}\in J(G)$, 
$T(h_{1},h_{2},p,\cdot )$ is not differentiable at $z_{0}.$ 
Moreover, by \cite[Lemma 5.1]{S3}, $\lambda _{h_{1},h_{2},p}$ is non-atomic. 
Hence  there exists an uncountable dense subset $A$ of $J(G)$ such that 
at every point of $A$, the function 
$T(h_{1},h_{2},p,\cdot )$ is not differentiable. 

 We now let  $\alpha \in (u(h_{1},h_{2},p, \tilde{\lambda}_{h_{1},h_{2},p}), 1)$ 
and let $\varphi \in C^{\alpha}(\CCI )$ such that $\varphi (\infty )=1$ and $\varphi |_{\hat{K}(G)}\equiv 0$.   By Theorem~\ref{t:oscanal}, we have 
$M_{h_{1},h_{2},p}^{n}(\varphi )(z)\rightarrow T(h_{1},h_{2},p,z)$ as $n\rightarrow \infty $ 
uniformly on $\CCI .$ 
If there exists a constant $C>0$ and a strictly increasing sequence $\{ n_{j}\} $ in $\NN $ 
such that $\| M_{h_{1},h_{2},p}^{n_{j}}(\varphi )\| _{\alpha }\leq C $ for each $j$, 
then we obtain $T(h_{1},h_{2},p,\cdot )\in C^{\alpha }(\CCI ).$ However, this is a contradiction. 
Hence  $\| M_{h_{1},h_{2},p}^{n}(\varphi )\| _{\alpha }\rightarrow \infty $ as $n\rightarrow \infty .$ 
Thus we have proved our theorem.
\end{proof}
\begin{rem}
\label{r:oschg3}
Let $h_{1}\in {\cal P}$ and $d\in \NN $ with $d\geq 2$ and $p\in (0,1)$.  Suppose that $\langle h_{1}\rangle $ is postcritically bounded, 
int$(K(h_{1}))\neq\emptyset $ and $(\deg (h_{1}),d)\neq (2,2).$
 Then, as in Remark~\ref{r:toscanal},  
by Theorem~\ref{t:spgenex}, there exists an element $h_{2}\in {\cal P}$ with 
$\deg (h_{2})=d$ such that  
$(h_{1},h_{2})\in ((\partial {D})\cap {\cal C}\cap {\cal B})\setminus 
{\cal I}$ 
and such that $(h_{1},h_{2},p)$ satisfies the assumptions of 
Theorems~\ref{t:oscanal} and \ref{t:oschg2}.  These two theorems imply that 
the associated random dynamical system does not have chaos in the $C^{0}$ sense 
(note that this is a randomness-induced phenomenon which cannot hold in the deterministic iteraton 
dynamics of an $f\in {\cal P}$),  
but still has a kind of chaos in the $C^{\alpha }$ sense for some $0<\alpha <1$.  
More precisely,  
there exists a number $\alpha _{0}\in (0,1)$ such that 
for each $\alpha \in (\alpha _{0}, 1)$, 
the system behaves chaotically  
on the Banach space $C^{\alpha }(\CCI )$ (and on the Banach space $C^{1}(\CCI )$ as well).  
Namely, as in Remark~\ref{r:gradation}, the above results indicate that 
regarding the random dynamical systems, we have a kind of gradation between chaos and 
order. Note that in Theorems~\ref{t:oscanal} and \ref{t:oschg2} we do not assume 
hyperbolicity, and as in Remark~\ref{r:toscanal}, if $h_{1}$ has a parabolic cycle or 
Siegel disk cycle, then the above $h_{2}$ can be taken so that $\langle h_{1},h_{2}\rangle $ 
is not mean stable. Moreover, as in Remark~\ref{r:toscanal} again, 
if $h_{1}$ has a parabolic cycle, then the above $h_{2}$ can be taken so that 
$J_{\ker }(h_{1},h_{2})=\emptyset $ and 
a minimal set of $\langle h_{1}, h_{2}\rangle $ meets the Julia set of $\langle h_{1},h_{2}\rangle .$ 
Thus, regarding the gradation between the chaos and order, 
Theorems ~\ref{t:oscanal} and \ref{t:oschg2} deal with  a new case.  

\end{rem}
\vspace{-4mm} 
\subsection{Proof of Theorem~\ref{t:j1orj2}}
\vspace{-2mm} 
In this subsection, we prove Theorem~\ref{t:j1orj2}. 
We need several lemmas. 
\begin{lem}
\label{l:suppmdisc}
Let $m\in \NN $ with $m\geq 2.$ 
Let $h=(h_{1},\ldots, h_{m})\in (\emRat )^{m}.$ Let $G=\langle h_{1},\ldots, h_{m}\rangle .$ 
Let $p=(p_{1},\ldots ,p_{m})\in (0,1)^{m}$ with $\sum _{j=1}^{m}p_{j}=1$ and 
let $\tau =\sum _{j=1}^{m}p_{j}\delta _{h_{j}}.$ 
Let $L$ be an attracting minimal set for $(G,\CCI )$. 
Let $\tilde{\mu }$ be an $\tilde{h}$-invariant ergodic Borel probability measure on $J(\tilde{h})$. 
Let $\mu =(\pi _{\CCI })_{\ast }(\tilde{\mu }).$ 
Suppose that all of the following holds. 
\begin{itemize}
\item[{\em (a)}] 
{\em supp}$\,\mu \subset J(G)\setminus \bigcup _{(i,j):i\neq j}(h_{i}^{-1}(J(G))\cap h_{j}^{-1}(J(G))).$ 
\item[{\em (b)}] 
%
There exists a point $a \in \mbox{{\em supp}}\,\mu \setminus P(G)$ such that 
$T_{L,\tau }$ is not constant in any neighborhood of $a.$  
\item[{\em (c)}] 
$\int \log \| D(\gamma _{1})_{x}\| _{s}d\tilde{\mu }(\gamma ,x)>0.$ 
\end{itemize}  
Then, $T_{L,\tau }$ is continuous at each point of supp$\,\mu $ and 
for $\mu$-a.e. $z_{0}\in J(G)$, $\emHol(T_{L,\tau },z_{0})\leq 
v(h,p,\tilde{\mu }).$ 

\end{lem}
\begin{proof}
Let $x\in \mbox{supp}\,\mu .$ Then there exists an element $\gamma \in \Sigma _{m}$ 
such that $(\gamma ,x)\in \mbox{supp}\,\tilde{\mu }.$ 
Since $\tilde{h}^{n}(\mbox{supp}\,\tilde{\mu })\subset \mbox{supp}\,\tilde{\mu }$ for each 
$n\in \NN $, we obtain that 
\vspace{-2mm} 
\begin{equation}
\label{eq:hgnxins}
h_{\gamma |_{n}}(x)\in \mbox{supp}\,\mu \subset J(G)\setminus 
\bigcup _{(i,j):i\neq j}(h_{i}^{-1}(J(G))\cap h_{j}^{-1}(J(G))) 
\mbox{ for each }n\in \NN .
\end{equation} 
Let $\alpha \in \Sigma _{m}$ be an element such that 
$x\in \cap _{n=0}^{\infty }h_{\alpha |_{n}}^{-1}(J(G)).$ 
Then by (\ref{eq:hgnxins}), it follows that $\alpha =\gamma .$ 
Therefore 
$\{ \omega \in \Sigma _{m}\mid x\in \cap _{n=0}^{\infty }h_{\omega |_{n}}^{-1}(J(G))\} =\{ \gamma \} .$ 
By \cite[Lemma 4.3]{Splms10}, we obtain that 
supp$\,\mu \subset F_{pt}^{0}(\tau ).$ By Lemma~\ref{l:attminset}, 
it follows that $T_{L,\tau }$ is continuous at every point of supp$\,\mu .$ 
 
Since $a\in \mbox{supp}\,\mu $, there exists a point $\omega \in \Sigma _{m}$ 
such that $(\omega ,a)\in \mbox{supp}\,\tilde{\mu }.$ 
Let 
\vspace{-2mm} 
$$\tilde{A}:=\{ (\gamma ,x)\in J(\tilde{h})\mid \exists \{ n_{j}\} \rightarrow \infty \mbox{ s.t. } 
\tilde{h}^{n_{j}}(\gamma, x)\rightarrow (\omega ,a) \mbox{ as } j\rightarrow \infty \} .$$ 
Then by Birkhoff's ergodic theorem, we have $\tilde{\mu }(\tilde{A})=1.$ 
By using the assumptions of our lemma, (\ref{eq:hgnxins}) and the method in the proof of 
Theorem~\ref{t:oscholgen}, it is easy to see that 
for $\mu $-a.e. $z_{0}\in J(G)$, $\Hol(T_{L,\tau },z_{0})\leq v(h,p,\tilde{\mu }).$ 
Thus we have proved our lemma. 
\end{proof}
\vspace{-4mm} 
\begin{lem}
\label{l:jhjscs}
Let $(h_{1},h_{2})\in (\overline{{\cal D}}\cap {\cal B}\cap {\cal H})\setminus {\cal I}.$ 
Then $\cup _{j=1}^{2}J(h_{j})\subset \CCI \setminus (h_{1}^{-1}(J(G))\cap h_{2}^{-1}(J(G))).$ 
Moreover, if, in addition to the assumption, $K(h_{1})\subset K(h_{2})$, then 
$h_{2}(K(h_{1}))\subset \mbox{int}(K(h_{1}))$ and 
$h_{1}(\overline{F_{\infty }(h_{2})})\subset F_{\infty }(h_{2}).$ 
\end{lem}
\begin{proof}
By Proposition~\ref{p:bdyosc}, either  
$K(h_{1})\subset K(h_{2})$ or $K(h_{2})\subset K(h_{1}).$ 
We assume $K(h_{1})\subset K(h_{2}).$ 
 By Proposition~\ref{p:bdyosc} again, 
we obtain $K(h_{1})\subset h_{1}^{-1}(K(h_{2}))\subset h_{2}^{-1}(K(h_{1}))\subset K(h_{2})$ and 
$(h_{1},h_{2})$ satisfies the open set condition with 
$U=(\mbox{int}(K(h_{2})))\setminus K(h_{1}).$ 
Then $J(G)\subset \overline{U}\subset K(h_{2})\setminus \mbox{int}(K(h_{1})).$ 
Suppose $J(h_{1})\cap h_{2}^{-1}(J(G))\neq \emptyset .$ 
Then  $J(h_{1})\cap h_{2}^{-1}(J(G))\subset 
J(h_{1})\cap h_{2}^{-1}(K(h_{2})\setminus \mbox{int}(K(h_{1})))
\subset J(h_{1})\cap h_{2}^{-1}(J(h_{1})).$ 
Hence $J(h_{1})\cap h_{2}^{-1}(J(h_{1}))\neq \emptyset .$ 
Since $J(h_{1})\subset h_{1}^{-1}(K(h_{2}))\subset h_{2}^{-1}(K(h_{1}))$, we 
obtain that $\emptyset \neq J(h_{1})\cap h_{2}^{-1}(J(h_{1}))\subset J(h_{1})\cap h_{1}^{-1}(J(h_{2})).$ 
Thus $J(h_{1})\cap J(h_{2})\neq \emptyset .$ However, this contradicts $(h_{1},h_{2})\not\in {\cal I}.$ 
Hence we must have that $J(h_{1})\cap h_{2}^{-1}(J(G))=\emptyset .$ 
Similarly, we can show that $J(h_{2})\cap h_{1}^{-1}(J(G))=\emptyset .$ 
Since $K(h_{1})\subset h_{1}^{-1}(K(h_{2}))\subset h_{2}^{-1}(K(h_{1}))\subset K(h_{2})$, 
$J(h_{1})\cap h_{2}^{-1}(J(G))=\emptyset $, and $J(h_{2})\cap h_{1}^{-1}(J(G))=\emptyset $, 
we obtain that  $h_{2}(K(h_{1}))\subset \mbox{int}(K(h_{1}))$ and 
$h_{1}(\overline{F_{\infty }(h_{2})})\subset F_{\infty }(h_{2}).$
Thus we have proved our lemma. 
\end{proof}
\vspace{-3mm} 
\begin{lem}
\label{l:intt-10}
Let $\tau \in {\frak M}_{1}({\cal P}).$ Suppose $\infty \in F(G_{\tau }).$ 
Then $\mbox{{\em int}}(T_{\infty ,\tau }^{-1}(\{ 0\} ))\subset F(G_{\tau }).$ 
\end{lem}
\begin{proof}
Since $\infty \in F(G_{\tau })$, \cite[Lemma 5.24]{Splms10} implies that 
for each $\gamma \in X_{\tau }$, $\gamma _{n,1}\rightarrow \infty $ locally uniformly on $F_{\infty }(G_{\tau }).$  
We now prove the following claim.\\ 
Claim. For each $z_{0}\in T_{\infty ,\tau }^{-1}(\{ 0\} )$, 
there exists no $g\in G_{\tau }$ with $g(z_{0})\in F_{\infty }(G_{\tau }).$

To prove this claim, let $z_{0}\in T_{\infty ,\tau }^{-1}(\{ 0\} )$ and 
suppose there exists an element $g\in G_{\tau }$ with $g(z_{0})\in F_{\infty }(G_{\tau }).$ 
Let $h_{1},\ldots ,h_{m}\in \G_{\tau }$ be some elements with 
$g=h_{m}\circ \cdots \circ h_{1}.$ 
Then there exists a neighborhood $W$ of $(h_{1},\ldots ,h_{m})$ in $\G_{\tau }^{m}$ 
such that for each $\omega =(\omega _{1},\ldots ,\omega _{m})\in W$, 
$\omega _{m}\cdots \omega _{1}(z_{0})\in F_{\infty }(G_{\tau }).$ 
Therefore for each $\g \in X_{\tau }$ with $(\g _{1},\ldots ,\g _{m})\in W$, 
$\g _{n,1}(z_{0})\rightarrow \infty $ as $n\rightarrow \infty .$   
Thus 
$T_{\infty ,\tau }(z_{0})\geq \tilde{\tau }(\{ \g \in X_{\tau }\mid (\g _{1},\ldots ,\g _{m})\in W\} )>0.$ 
This is a contradiction. Hence the claim holds. 

 From this claim, $G_{\tau }(\mbox{int}(T_{\infty ,\tau }^{-1}(\{ 0\} )))\subset \CCI \setminus F_{\infty }(G_{\tau }).$  
Therefore $ \mbox{int}(T_{\infty ,\tau }^{-1}(\{ 0\} ))\subset F(G_{\tau }).$ 
\end{proof}
We now prove Theorem~\ref{t:j1orj2}. \\ 
\noindent {\bf Proof of Theorem~\ref{t:j1orj2}.} 
By Lemma~\ref{l:d1d2not22}, 
we have $\deg (h_{j})\geq 2$ for each $j$ and $(\deg (h_{1}),\deg (h_{2}))\neq (2,2)$. Therefore 
there exists an $i\in \{ 1,2\} $ such that $-\frac{\log p_{i}}{\log \deg (h_{i})}<1.$ 

By Proposition~\ref{p:bdyosc}, 
we may assume that $K(h_{1})\subset K(h_{2}).$ 
Since $(g_{1},g_{2})\mapsto J(g_{1},g_{2})$ is continuous in a neighborhood of $(h_{1},h_{2})$ 
with respect to the Hausdorff metric (see \cite[Theorem 2.4.1]{S1}), 
Lemma~\ref{l:jhjscs} implies that  
there exists an open neighborhood $V$ of $(h_{1},h_{2})$ such that 
for each $(g_{1},g_{2})\in V,$  we have 
$\cup _{j=1}^{2}J(g_{j})\subset \CCI \setminus (g_{1}^{-1}(J(g_{1},g_{2}))\cap g_{2}^{-1}(J(g_{1},g_{2})))$, 
$g_{2}(K(g_{1}))\subset \mbox{int}(K(g_{1}))$ and $g_{1}(\overline{F_{\infty }(g_{2})})\subset F_{\infty }(g_{2}).$ 
Then for each $(g_{1},g_{2})\in V$, $\hat{K}(g_{1},g_{2})=K(g_{1})$ and 
$F_{\infty }(g_{1},g_{2})=F_{\infty }(g_{2}).$   
If $V$ is small enough, then  Lemma~\ref{l:pfmainth10} implies that 
 for each $g=(g_{1},g_{2})\in V$, $T(g_{1},g_{2},p,\cdot )$ is continuous on $\CCI .$ 
 Hence for each $g=(g_{1},g_{2})\in V$, 
 $T(g_{1},g_{2},p,\cdot )|_{\overline{F_{\infty }(g_{1},g_{2})}}\equiv 1.$ 
Let $g=(g_{1},g_{2})\in V.$ 
 By Lemma~\ref{l:intT1}, it follows that for each point $a\in \partial F_{\infty }(g_{1},g_{2})=J(g_{2})$, 
 the function $T(g_{1},g_{2},p,\cdot )$ is not constant in any neighborhood of $a.$ 
Moreover, by lemma~\ref{l:intt-10}, for each point $a\in \partial \hat{K}(g_{1},g_{2})=J(g_{1})$, 
the function $T(g_{1},g_{2},p,\cdot )$ is not constant in any neighborhood of $a.$ 
Furthermore, for each $j=1,2,$ there exists a $\tilde{g}$-invariant Borel probability measure 
$\tilde{\mu}_{j}$ on $J(\tilde{g})$ such that $(\pi _{\CCI })_{\ast }(\tilde{\mu }_{j})=\mu _{j}$. 
Since $\mu _{j}$ is ergodic with respect to $g_{j}$, we obtain that 
$\tilde{\mu }_{j}$ is ergodic. 
Moreover, for each $j=1,2,$ we have 
$\int \log \| D(\gamma _{1})_{x}\| _{s}d\tilde{\mu }_{j}(\gamma ,x)=
\int \log \| D(g_{j})_{x}\| _{s}d\mu _{j}(x)=\log \deg (g_{j})>0.$ 
Hence, 
Lemma~\ref{l:suppmdisc} implies that for $\mu _{j}$-a.e. $z_{0}\in J(g_{j})$, 
we have $\Hol(T(g_{1},g_{2},p,\cdot ),z_{0})\leq -\frac{\log p_{j}}{\int \log \| D(g_{j})_{x}\| _{s}d\mu _{j}(x)}
=-\frac{\log p_{j}}{\log \deg g_{j}}.$ Therefore items (i) (ii) of our theorem hold. 
We now prove item (iii) of our theorem. Let  $\alpha \in (-\frac{\log p_{i}}{\log \deg (g_{i})}, 1)$  and 
let $\varphi \in C^{\alpha}(\CCI )$ be an element 
such that $\varphi (\infty )=1$ and $\varphi |_{\hat{K}(G)}\equiv 0$. 
By Lemma~\ref{l:pfmainth10}, if $V$ is small enough, 
item (v) in statement~\ref{spacemainth10} in Theorem~\ref{t:timportant} holds. 
Thus $M_{g_{1},g_{2},p}^{n}(\varphi )(z)\rightarrow T(g_{1},g_{2},p,z)$ as $n\rightarrow \infty $ uniformly on $\CCI .$ 
If there exists a constant $C>0$ and a strictly increasing sequence $\{ n_{j}\} $ in $\NN $ 
such that $\| M_{g_{1},g_{2},p}^{n_{j}}(\varphi )\| _{\alpha }\leq C $ for each $j$, 
then we obtain $T(g_{1},g_{2},p,\cdot )\in C^{\alpha }(\CCI ).$ However, this contradicts item (ii) of 
Theorem~\ref{t:j1orj2}, which we have already proved. Therefore, 
item (iii) of our theorem holds. Thus, we have proved Theorem~\ref{t:j1orj2}.  
\qed

\end{document}